\documentclass[11pt]{article}

\usepackage[margin=1in]{geometry}

\usepackage{amsmath,amsthm,amssymb,amsfonts}
\usepackage{mathtools} % extends/fixes amsmath

\usepackage{graphicx}
\usepackage{booktabs}
\usepackage{tabularx}
\usepackage{array}      % for newcolumntype
\usepackage{enumitem}
\usepackage{float}
\usepackage{subcaption}
\usepackage{caption}
\usepackage[section]{placeins}
\usepackage{comment}

\usepackage{algorithm}
\usepackage{algpseudocode}

\usepackage{bm}
\usepackage{mathrsfs}

\usepackage{microtype}

\usepackage{xcolor}
\colorlet{red}{black}

\usepackage{hyperref}
\hypersetup{
  colorlinks = true,
  linkcolor  = blue,
  citecolor  = blue,
  urlcolor   = blue,
  pdfauthor  = {Ayoub Lailoune and Valerio Loi and Stefano Serra-Capizzano},
  pdftitle   = {Weyl distributions, spectral properties, and circulant approximation results for quaternion block multilevel Toeplitz matrix sequences}
}
\usepackage[nameinlink,capitalise]{cleveref}

\numberwithin{equation}{section}
\newtheorem{theorem}{Theorem}[section]
\newtheorem{lemma}[theorem]{Lemma}
\newtheorem{corollary}[theorem]{Corollary}
\newtheorem{proposition}[theorem]{Proposition}
\theoremstyle{definition}
\newtheorem{definition}[theorem]{Definition}
\theoremstyle{remark}
\newtheorem{remark}[theorem]{Remark}

\newcommand{\RR}{\mathbb{R}}
\newcommand{\CC}{\mathbb{C}}
\newcommand{\HH}{\mathbb{H}}
\newcommand{\TT}{\mathbb{T}}

\newcommand{\I}{\mathbf{i}}\newcommand{\J}{\mathbf{j}}\newcommand{\K}{\mathbf{k}}

\DeclarePairedDelimiterX{\inner}[2]{\langle}{\rangle}{#1,#2}
\DeclareMathOperator{\diag}{diag}

\newcommand{\btheta}{\bm{\theta}}

\providecommand{\NN}{\mathbb{N}}

\providecommand{\ZZ}{\mathbb{Z}}
\providecommand{\TT}{\mathbb{T}}

\providecommand{\bk}{\boldsymbol{k}}
\providecommand{\brho}{\boldsymbol{\rho}}
\providecommand{\bn}{\boldsymbol{n}}
\providecommand{\bm}{\mathbf{m}}

\providecommand{\Bdiag}{\operatorname{Bdiag}}

\DeclareMathOperator{\dist}{dist}

\newcommand{\PhiSymp}{\Phi}

\newcommand{\val}[1]{\textcolor{red}{#1}}

\makeatletter
\newcommand{\MovedProofMarker}[4]{%
  \expandafter\xdef\csname mp@section@#1\endcsname{\arabic{section}}%
  \expandafter\xdef\csname mp@equation@#1\endcsname{\arabic{equation}}%
  \expandafter\xdef\csname mp@currentlabel@#1\endcsname{\@currentlabel}%
  \addtocounter{equation}{#2}%
  \par\smallskip
  \noindent\hyperref[#3]{\val{\emph{The proof is given in Appendix~\ref*{#4}.}}}%
  \par\smallskip
}
\newenvironment{MovedProofContext}[1]{%
  \begingroup
  \setcounter{equation}{\csname mp@equation@#1\endcsname}%
  \renewcommand{\theequation}{\csname mp@section@#1\endcsname.\arabic{equation}}%
  \renewcommand{\theHequation}{movedproof.#1.\arabic{equation}}%
  \edef\@currentlabel{\csname mp@currentlabel@#1\endcsname}%
}{%
  \endgroup
}
\makeatother
\newcolumntype{L}{>{\raggedright\arraybackslash}X}

\title{Weyl distributions, spectral properties, and circulant approximation results for quaternion block multilevel Toeplitz matrix sequences}

\author{%
Ayoub Lailoune\thanks{The UM6P Vanguard Center, Mohammed VI Polytechnic University -- Rabat Campus, Morocco. \texttt{ayoub.lailoune@um6p.ma}}
\and
Valerio Loi\thanks{Department of Science and High Technology, University of Insubria. \texttt{vloi@uninsubria.it}}
\and
Stefano Serra-Capizzano\thanks{Department of Science and High Technology, University of Insubria. \texttt{s.serracapizzano@uninsubria.it}}
\thanks{Department of Information Technology, Division of Scientific Computing, Uppsala University. \texttt{stefano.serra@it.uu.se}}
}

\date{\today}

\begin{document}
\maketitle

\begin{abstract}
The present work contains a comprehensive treatment of Weyl eigenvalue and singular value distributions, \val{Schatten $p$-norm estimates, spectral localization and positive-definiteness criteria} for single-axis quaternion block multilevel Toeplitz matrix sequences generated by $s\times t$ quaternion matrix-valued, $d$-variate, Lebesgue integrable generating functions\val{, where $s,t,d$ are positive integers and where spectral localization and positive-definiteness are studied in the Hermitian setting, with $s=t$}.  Furthermore, in view of concrete applications, we are interested in preconditioning and matrix approximation results. To this end, a crucial step is the extension of the notion of an approximating class of sequences (a.c.s.) to the case of matrix sequences with quaternion entries, since it allows us to decompose the difference between a matrix and its preconditioner into low-norm plus (relatively) low-rank terms. As a specific example, we consider classes of quaternion block multilevel circulant matrix sequences as an a.c.s. for quaternion block multilevel Toeplitz matrix sequences. These approximation results lay the foundations for fast preconditioning methods when dealing with large quaternion linear systems stemming from modern applications. We conclude our study with numerical experiments and directions for future research.
\end{abstract}


% ---------------------------------------------------------------------
\section{Introduction}\label{sec:intro}

The linear algebra and matrix theory of the quaternions \cite{zhang_quaternions_1997} and their numerical applications have grown significantly in recent years; see, e.g., \cite{jacobson1974basic} for the foundations in the context of noncommutative algebras.
%%%%%%%%%%%%%%
In particular, structured matrices such as Toeplitz-type quaternion matrices are widely considered in applications such as discrete convolutions on tori and quaternion signal analysis (cf., for example, the motivations of the analysis in \cite{lin_hermitian_2025}), while structured quaternion matrix theory is emerging in fields such as color imaging, polarization imaging, and low-rank quaternion modeling; see, e.g., \cite{jia_ng_song_rqmc_2019,chen_ng_inpainting_2022,pan_ng_sqmf_2023,liu_ling_jia_randomized_qsvd_2022} and the references therein.
In these applications, there is a need for fast algorithms, as they are well-studied for the case of large linear systems with complex coefficient matrices. In particular, when Toeplitz structures are encountered, a successful strategy has consisted \val{in} approximating them using matrix algebras associated with fast transforms such as circulants, $\omega$-circulants, and trigonometric matrix algebras (see \cite{Ng-Book,benedetto_optimal_2000,Koro2,garoni_multilevel_2018,GLTblock2} and the references therein). In turn, these approximations are used in connection with Krylov methods, with the goal of designing effective preconditioners, or in connection with multigrid (multi-iterative) solvers for finding appropriate smoothers and projection operators. In this direction, it is worth emphasizing that there has been recent interest in designing structure-preserving quaternion Krylov methods (such as QGMRES, QBiCG, QFOM), for which both the modeling benefits and the algorithmic constraints are analyzed when working directly over $\HH$ \cite{jia_ng_qgmres_2021,li_wang_qbicg_2024,li_wang_qfom_2023,li_wang_zhang_glqfom_2025}.

On the other hand, the analysis of the efficiency of such preconditioners depends on an accurate spectral and singular value analysis of the original matrices and matrix sequences and of the preconditioned ones. In this direction, having in mind the Toeplitz setting in the complex case, we need precise localization results, a study of the extreme spectral/singular value behavior, and an asymptotic study of the Weyl eigenvalue and singular value distributions of Toeplitz matrix sequences with quaternion entries, using the specific features of quaternion matrix theory; see \cite{zhang_quaternions_1997,rodman_topics_2014}. The asymptotic and analytic study of the quaternion Toeplitz spectrum parallels the theory of canonical-eigenvalue computations \cite{mackey_quat_eig_1995} and modern quaternion low-rank/SVD analysis \cite{song_ding_ng_lowrank_2021,huang_jia_li_rcg_qsvd_2025}.

More concretely, our aim is to develop an asymptotic spectral and singular value distribution theory for multilevel block quaternion Toeplitz sequences generated by Fourier transforms along a single-axis, and an approximation framework (that is, the approximating class of sequences (a.c.s.) theory) that helps to make the approximation theory of quaternion matrix sequences as systematic as in the complex case. Even though this step is far from trivial, we look at it as the first task for treating more general problems where there are locally varying coefficients. In fact, \val{a more advanced} goal is to build the quaternion counterpart of the systematic asymptotic spectral and singular value theory of the complex generalized locally Toeplitz (GLT) framework \cite{garoni_Toeplitz_2017,garoni_multilevel_2018,GLTblock2}. We believe that developing such theoretical machinery would represent a strong tool that could accelerate this emerging field.

The two main problems we face when developing quaternion spectral theory are well known: the noncommutativity of the quaternion algebra, which makes left and right actions genuinely different, and a theory of eigenvalues that is substantially different from the case of complex-valued matrices. Some intuitive solutions to these problems, which we adopt, are to analyze complex slices, \val{$\CC_{\I}$, $\I$ being the standard imaginary unit of the complex field
$\CC$ and also one of the units of the quaternions,} separately and to use the ``symplectic embedding'' trick introduced by Lee \cite{lee_quaternion_1947} for many problems related to the properties of quaternions; see also \cite{jacobson1974basic}.
Beyond the applications we suggest, quaternion models also appear in control, collective dynamics, discrete geometry, and computational number theory. This further motivates the development of a robust spectral and matrix theory over $\HH$ \cite{fu_kou_wang_control_2023,degond_collective_2018,solymosi_swanepoel_incidence_2008,kv_quat_orders_2010,kv_corrigendum_2012}.

In the current paper, we contribute in the following way:
\begin{enumerate}
\item
We give a practical definition of eigenvalue and singular value symbols for quaternion sequences by analyzing the canonical \val{eigenvalues}, which reproduces most complex spectral asymptotic properties as given in \cite{tyrtyshnikov1998,TilliNota,TilliComplex,Donatelli2012}. Canonical eigenvalues provide lossless representatives of all infinite right-eigenvalue classes. Thus, we regard this choice as ideal for our purposes. The singular values of a quaternion matrix behave, in an asymptotic sense, in the same way as their complex matrix counterparts, so the setting we define is analogous. Since we consider block multilevel Toeplitz matrix sequences with $s\times t$ quaternion matrix-valued generating functions, the corresponding complex counterparts are represented by the works \cite{TilliNota,Donatelli2012}. We refer the reader to \cite{capizzano_distribution_2001,TilliLT} for the original definition and its first use, and to \cite[Chapter~5]{garoni_Toeplitz_2017} for the a.c.s. calculus in the complex setting.

\item
We extend approximating classes of sequences to rectangular quaternionic matrix sequences and prove the equivalence of the defined approximation with complex a.c.s.\ after embedding, with matching rank and norm normalizations. \val{The resulting calculus consists of the complex a.c.s.\ axioms
\(\mathrm{ACS}1\)-\(\mathrm{ACS}5\) and their quaternionic counterparts
\(\mathrm{ACS}_{\HH}1\)-\(\mathrm{ACS}_{\HH}5\). We prove
the product axiom for composable rectangular sequences with varying aspect
ratios under an explicit product-compatibility condition.}

\item
We formalize the three single-axis kernels and recover standard transport identities under the embedding for the rectangular matrix-valued case and the reduction to right kernels. In the scalar case, these are classical tools in the field of quaternion signal processing.

\item
For multilevel block Toeplitz matrices we show that any sandwich Toeplitz matrix equals a right Toeplitz matrix with a reflected symbol; as a consequence, the corresponding matrix sequence shares the same property. Further, we explicitly prove adjoint identities and a sharp criterion for proving the Hermitian character in the block, multilevel setting, together with tight localization results for the associated spectra.

\item
We give exact complex $2s\times 2t$ singular and canonical eigenvalue symbols for the left, right, and sandwich cases.
We prove that multilevel quaternion Toeplitz sequences with $L^{1}$ symbols are distributed in \val{the sense of} the singular values as the embedded symbols $G_\tau$ for $\tau\in\{L,R,(S_L,S_R)\}$. For eigenvalues, we treat (i) Hermitian sequences without additional regularity and (ii) non-Hermitian symbols under a Tilli-type separation hypothesis of the embedded generating function; for the separation hypothesis see \cite{TilliComplex,Donatelli2012}.

\item
We recover the Schatten $p$-norm estimates for Toeplitz matrices with $L^{p}$ symbols in the same scaling as the complex case, with the exception of a factor $\val{4}$; see Theorem~\ref{thm:q-schatten-coarse-mv}. For Hermitian sequences we localize spectra of all finite matrices inside the essential range of the embedded Hermitian symbol and we obtain positivity criteria. Here the reference \val{for the complex setting is} paper \cite{serra-capizzano_unitarily_2002}, where a rich variety of variational and norm inequalities are derived.

\item
Adapting the Pan-Ng splitting and decomposition \cite{PanCirculant2024} for our purposes and settings, we show that right Toeplitz sequences admit a circulant a.c.s., providing an alternative route to singular/spectral distribution results for quaternion Toeplitz sequences that recover the classical approach for the corresponding complex Toeplitz sequences.
\end{enumerate}
In summary, our theoretical tools and matrix analysis results are based on \cite{zhang_quaternions_1997,rodman_topics_2014}, with the canonical-eigenvalue selection from \cite{zhang_quaternions_1997}. The symplectic embedding trick, to our knowledge, dates back to \cite{lee_quaternion_1947} and is here used at the level of matrix sequences. The use of a quaternion Fourier analysis in signal processing is discussed in \cite{ell_quaternion_2013,Ell_2014}. In the complex case, the multilevel GLT machinery is well developed and it is explained in an exhaustive way in the surveys \cite{garoni_multilevel_2018,GLTblock2}. For the unilevel scalar Hermitian quaternion Toeplitz with bounded generating function, see \cite{lin_hermitian_2025}.

\paragraph{Organization of the work.}
In Section~\ref{sec:quaternion} we review algebraic primitives, symplectic embedding, canonical eigenvalues, Schatten norms, and function-space preliminaries. Section~\ref{sec:q_asymptotic_spectral} contains the embedding-based notion of spectral and singular value symbols. In Section~\ref{sec:QFT-Toeplitz} we develop the study of single-axis quaternion Toeplitz sequences, we prove kernel reduction, adjoint and properties of quaternion Hermitian character, together with proper localization results of the spectra. Furthermore, we give exact embedded symbols and we derive distribution theorems and Schatten/localization bounds. In Section~\ref{sec:acs-circulant} we establish the circulant a.c.s.\ and the QDFT fiber picture, giving an alternative proof of the spectral and singular value Toeplitz distributions. Section~\ref{sec:apps} is devoted to numerical tests for verifying and visualizing distribution results and for checking whether they are already evident in practical and low dimensional contexts. Conclusions and further remarks are reported in Section~\ref{sec:conclusion}, together with a coherent plan for future investigations.

\section{Definitions and preliminary tools}\label{sec:quaternion}

% ========================================================\section{Definitions and preliminary tools}\label{sec:quaternion}
In the present section we set the basic notation for quaternions\val{, its units,} and quaternion matrices, \val{and then we introduce} the slice algebras \(\CC_{\mu}\) \val{for a given unit $\mu$} and the spectral facts used later. For most of the notations, we follow the standard conventions for quaternion algebra and linear algebra. For a systematic and more theoretical treatment of quaternion linear algebra, see \cite{rodman_topics_2014}. The link between quaternion and complex linear algebra is mostly taken from \cite{zhang_quaternions_1997} and will be used repeatedly.

%--------------------------------------------------------------------
\subsection{Algebraic primitives}\label{subsec:alg_primitives}
The classical Hamilton quaternion algebra is the skew-field
\[
  \HH := \operatorname{span}_{\RR}\{1,\I,\J,\K\},\qquad
  \I^{2}=\J^{2}=\K^{2}=\I\J\K=-1,
\]
where, by construction,
\[
  \I\J=\K,\qquad \J\K=\I,\qquad \K\I=\J.
\]
  For simplicity in our calculations, in this paper, the unit \(\K\) will always be represented in the explicit form \(\K=\I\J\). \\ By construction \(q\in\HH\) has the Cartesian form \(q=q_{0}+q_{1}\I+q_{2}\J+q_{3}\K=q_{0}+q_{1}\I+q_{2}\J+q_{3}\I\J\) with \(q_{j}\in\RR\) the real components. The component \(\Re q := q_{0}\) is called the real part, while \(\Im q := q_{1}\I+q_{2}\J+q_{3}\K\) is the imaginary part. \\
The quaternion conjugation and the modulus (or norm) are defined in a similar way as for complex numbers, that is, we have
\[
  \overline{q}=q_{0}-q_{1}\I-q_{2}\J-q_{3}\K,
  \qquad
  |q|:=\sqrt{q\,\overline{q}},
\]
so that \(|pq|=|p|\,|q|\), exactly as in the complex setting.

\val{For a quaternionic or complex matrix \(A=(a_{k\ell})\), we write
\(\overline A=(\overline{a_{k\ell}})\) for its entrywise conjugate.}

\smallskip
\noindent
We call \(\mu\in\HH\) an imaginary unit if \(\Re\mu=0\) and \(\mu^{2}=-1\). Given such \(\mu\), we can define the complex slice
\[
  \CC_{\mu}:=\operatorname{span}_{\RR}\{1,\mu\}\cong\CC.
\]
If \(\nu\) is another normalized imaginary unit orthogonal to \(\mu\), that means \(\nu\mu=-\mu\nu\), then we find the important orthogonal decomposition (see \cite[Thm.~2.1]{zhang_quaternions_1997})
\begin{equation}\label{eq:orth_dec}
  \HH \;=\; \CC_{\mu}\oplus \CC_{\mu}\nu,
  \qquad\text{i.e.}\quad
  q=z+w\,\nu \ \text{with}\ z,w\in\CC_{\mu},
\end{equation}
and the representation in \eqref{eq:orth_dec} is unique. Unless stated otherwise, when we perform an orthogonal slice decomposition, we fix for simplicity \(\mu=\I\) and \(\nu=\J\). The representation \eqref{eq:orth_dec} is convenient for our matrix analysis, since it mimics complex algebra representations and enables a smooth use of complex-analytic arguments via the symplectic embedding defined in the next section. \\
In many algebraic arguments of the present paper, we use the following elementary identities
\begin{equation}\label{eq:comm}
    \J\alpha=\overline{\alpha}\,\J \quad \text{for all} \quad \alpha\in\CC_{\I},
\end{equation}
and the associated matrix version
\begin{equation}\label{eq:comm_mat}
    \J A=\overline{A}\,\J \quad \text{for all} \quad A\in\CC_{\I}^{m \times n}.
\end{equation}
For readability purposes, we declare explicitly these identities when we use them.

%--------------------------------------------------------------------
\subsection{Symplectic embedding}\label{subsec:symp}
We start by recalling the classical complex symplectic embedding. Writing \(x=z+w\J\) with \(z,w\in\CC_{\I}\) as in \eqref{eq:orth_dec} and taking into account the results in \cite{lee_quaternion_1947}, define
\begin{equation}\label{eq:symp_embed}
  \PhiSymp(x)
  :=
  \begin{bmatrix}
      z   & w\\
     -\overline{w} & \overline{z}
  \end{bmatrix}
  \ \in\ \CC^{2\times2}.
\end{equation}
The definition of \(\PhiSymp\) can be extended to matrices as follows. If \(A=Z+W\J\) with \(Z,W\in\CC_{\I}^{\,m\times n}\), then we set
\[
  \PhiSymp_{1}(A)\;:=\;
  \begin{bmatrix}
      Z   & W\\
     -\overline{W} & \overline{Z}
  \end{bmatrix}
  \ \in\ \CC^{2m\times2n}.
\]
Equivalently, one may apply \eqref{eq:symp_embed} entrywise to every entry of a matrix. With a slight abuse of notation, we also denote the matrix \emph{entrywise} extension by \(\PhiSymp\). \val{More precisely, for $A=(a_{k\ell})\in\HH^{m\times n}$ we set $\PhiSymp(A):=\bigl(\PhiSymp(a_{k\ell})\bigr)_{k=1,\dots,m}^{\ell=1,\dots,n}\in\CC^{2m\times 2n}$, the $m\times n$ block matrix whose $(k,\ell)$ block is $\PhiSymp(a_{k\ell})$ as in \eqref{eq:symp_embed}.} The two extensions are related by a fixed permutation \val{and in fact we have}
\begin{equation}\label{eq:similarity}
  P_m^{*}\,\PhiSymp(A)\,P_n \;=\; \PhiSymp_{1}(A),
\end{equation}
where \(P_n\in\{0,1\}^{2n\times 2n}\) is the permutation
\begin{equation*}
  P_n \;:=\; \sum_{i=1}^{n}\bigl(e_{2i-1}e_i^{*}+e_{2i}e_{n+i}^{*}\bigr),
\end{equation*}
and \(\{e_j\}_{j=1}^{2n}\) is the Euclidean basis of \(\CC^{2n}\)\val{, $^*$ denoting the standard conjugate transpose for complex vectors or matrices}.

The key structural property is indeed that \(\PhiSymp\) is a real-linear \(*\)-algebra embedding into complex matrices, see \cite{lee_quaternion_1947,zhang_quaternions_1997}. Throughout, if either $X \in \HH^{m \times t}$ or  $X \in\CC^{m \times t}$, $X^*$ denotes its transpose-conjugate, with respect to quaternions or complex numbers, respectively.

\begin{lemma}[$*$-algebra property]\label{lem:phi_star}
\(\PhiSymp\) is an injective real-linear \(*\)-algebra map. In particular,
\[
  \PhiSymp(AB)=\PhiSymp(A)\PhiSymp(B),
  \qquad
  \PhiSymp(A^{*})=\PhiSymp(A)^{*}.
\]
\end{lemma}

%--------------------------------------------------------------------
\subsection{Spectral tools for quaternion matrices}
Because of non-commutativity, left and right spectra differ in general and they have different properties. In this paper we always work with \emph{right} eigenvalues, which together with their related eigenvectors, are defined as follows.

\begin{definition}[Right eigenpair]\label{def:right_eig}
A nonzero \(x\in\HH^{m}\) and \(\lambda\in\HH\) form a right-eigenpair of
\(A\in\HH^{m\times m}\) if \(Ax=x\lambda\).
\end{definition}

Right eigenvalues in general are infinitely many, but we can group them with quaternion conjugacy classes as reported in the next result. As a consequence, using these equivalence classes, we can manage the spectra of the considered matrix sequences as in the complex setting.

\begin{lemma}[Conjugacy classes]\label{lem:conj_class}
If \(Ax=x\lambda\), then \(A(xq)=(xq)(q^{-1}\lambda q)\) for every nonzero \(q\). Thus right eigenvalues occur in conjugacy classes
\[
  [\lambda]:=\{\,q^{-1}\lambda q : q \ne 0\,\}.
\]
\end{lemma}
An important property is \val{the following:} while eigenvalues are infinitely many, right conjugacy classes are not.

\begin{lemma}[{\cite[Thm.~5.4]{zhang_quaternions_1997}}]\label{lem:right_exist}
Every \(A\in\HH^{n\times n}\) has exactly \(n\) right-eigenvalue classes, counted with multiplicity.
\end{lemma}
Note that, by Lemma \ref{lem:conj_class}, if an eigenvalue is real, then its conjugacy class collapses to the eigenvalue singleton. The latter is an important property of real eigenvalues that simplifies the analysis of Hermitian or normal matrices, for example.

\paragraph{Canonical eigenvalues.}
We now recall a standard way to select meaningful representatives of right eigenvalue classes. This enables us to perform a complete matrix spectral analysis using a similar approach, as done in the complex setting. Indeed, Zhang proved the existence of the \val{complex canonical eigenvalues, as reported in the subsequent result.}

\begin{theorem}[{\cite[Thm.~5.4]{zhang_quaternions_1997}}]\label{thm:canonical}
Every \(A \in \HH^{n \times n}\) has exactly \(n\) right eigenvalues which are complex numbers with nonnegative imaginary parts, counted with multiplicity.
\end{theorem}

A simple way to compute canonical eigenvalues is to compute the spectrum of \(\PhiSymp(A)\in\CC^{2n\times 2n}\). By construction, the eigenvalues of \(\PhiSymp(A)\) occur in conjugate pairs \(\{z,\overline z\}\)\val{. Hence, by retaining every eigenvalue in the open upper half-plane and half the algebraic multiplicity of each real eigenvalue, we obtain the \(n\) canonical eigenvalues, counted with multiplicity. {By Lemma~\ref{lem:borel-eigenvalue-enumeration} in Appendix~\ref{app:preliminary-proofs}, the canonical eigenvalues 
obtained in this way admit a Borel enumeration.}}

%--------------------------------------------------------------------
\subsubsection*{Quaternionic SVD and Schatten norms}
The singular value decomposition (SVD) over \(\HH\) mirrors the complex case.

\begin{theorem}[Quaternion SVD, {\cite[Thm.~7.2]{zhang_quaternions_1997}}]\label{thm:q_svd}
For \(A\in\HH^{m\times n}\) there exist quaternionic unitary matrices \(U,V\) of the proper dimensions, \val{$\rho=\operatorname{rank}_{\HH}(A)$}, and
\(\Sigma=\operatorname{diag}(\sigma_{1},\dots,\sigma_{\val{\rho}})\) with \(\sigma_{j}>0\), $j=1,\ldots,\val{\rho}$, such that
\[
  A
  =U\begin{bmatrix}\Sigma&0\\[1pt]0&0\end{bmatrix}V^{*}.
\]
The \(\sigma_{j}\) are the positive square roots of the positive eigenvalues of \(A^{*}A\), the other eigenvalues being all equal to zero.
\end{theorem}

\begin{definition}[Schatten \(p\)-norms]\label{def:schatten}
Let \(A\in\HH^{m\times n}\) and set \(r:=\min\{m,n\}\). Let the singular values be \(\sigma_1(A)\ge\cdots\ge\sigma_r(A)\ge0\). For \(p\in[1,\infty)\) the
\emph{Schatten \(p\)-norm} of \(A\) is
\[
  \|A\|_{p}
  \;:=\;\Biggl(\,\sum_{j=1}^{r}\sigma_j(A)^{\,p}\,\Biggr)^{\!1/p},
\]
and for \(p=\infty\) set \(\|A\|_{\infty}:=\sigma_1(A)\).
These norms are unitarily invariant over \(\HH\): \(\|UAV\|_{p}=\|A\|_{p}\) for all quaternionic unitary \(U,V\) of compatible sizes.
In accordance with the standard terminology in the complex setting, we call \(\|A\|_{1}\) the trace norm, \(\|A\|_{2}\) the Frobenius norm, and \(\|A\|_{\infty}\) the spectral norm.
\end{definition}
By convention, from this point on, we write \(\|A\|\) to denote spectral norm.

\begin{proposition}[Symplectic embedding duplicates singular values]\label{prop:phi_dup_sv}
Let \(A\in\HH^{m\times n}\). If \(\sigma_{1}\ge\cdots\ge\sigma_{r}>0\) are the nonzero singular values of \(A\), then the singular values of \(\PhiSymp(A)\in\CC^{2m\times 2n}\) are
\[
  \sigma_{1},\sigma_{1},\;\sigma_{2},\sigma_{2},\;\dots,\;\sigma_{r},\sigma_{r},
\]
followed by zeros to reach size \(\min\{2m,2n\}\).
\end{proposition}
\MovedProofMarker{p23}{0}{proof:p23}{app:preliminary-proofs}

\begin{corollary}\label{cor:s_p}
For \(1\le p\le\infty\),
\[
  \|A\|_{p}
  =2^{-1/p}\,\bigl\|\PhiSymp(A)\bigr\|_{p}.
\]
\end{corollary}

%--------------------------------------------------------------------
\subsection{Function theory preliminaries and notation}
\label{subsec:measLpH}
Let \(D\subset\RR^{k}\) be a real measurable set with Lebesgue measure \(|D|\).
We can consider \(\HH^{m\times n}\cong\RR^{4mn}\) using Cartesian coordinates and say that a quaternion matrix function \(f:D\to\HH^{m\times n}\) is measurable if and only if all its real coordinates are measurable. Equivalently, using real-linearity and continuity of \(\PhiSymp\), \(f:D\to\HH^{m\times n}\) is measurable if and only if \(\PhiSymp\circ f:D\to\CC^{2m\times 2n}\) is measurable: in the rest of the manuscript, for our purposes, the measurable domain $D$ is assumed to have finite and positive measure.\\
For asymptotic spectral and singular value theory of quaternion matrix sequences, we need to use the Schatten norms of measurable matrix-valued functions, their singular and canonical eigenvalue fields, and ensure that they are measurable. Fortunately, in analogy with the complex setting, analogous properties hold true.

\val{We denote by \(L^{0}(D;\CC^{m\times n})\) the space of equivalence
classes, modulo equality almost everywhere, of measurable functions
\(D\to\CC^{m\times n}\). The analogous notation is used for quaternion
matrix-valued functions.}

\val{A function \(g:\CC^{r}\to\CC\) is symmetric if
\[
g(z_{\pi(1)},\ldots,z_{\pi(r)})=g(z_{1},\ldots,z_{r})
\]
for every \((z_{1},\ldots,z_{r})\in\CC^{r}\) and every permutation
\(\pi\) of \(\{1,\ldots,r\}\).}

\begin{lemma}\label{lem:meas-spectral-functional}
Let \(f:D\to\HH^{r\times r}\) be a measurable function. Then for any \(p\in[1,\infty]\), the Schatten norm map \(x\mapsto\|f(x)\|_{p}\) is measurable; the singular value map \(x\mapsto(\sigma_{1}(f(x)),\ldots,\sigma_{r}(f(x)))\) is measurable; \val{and, if \(\lambda_1^+,\ldots,\lambda_r^+\) are the Borel maps
provided by Lemma~\ref{lem:borel-eigenvalue-enumeration}, then for every continuous
symmetric \(g:\CC^r\to\CC\), the map
\[
x\longmapsto
g\bigl(\lambda_1^+(f(x)),\ldots,\lambda_r^+(f(x))\bigr)
\]
is measurable}.
\end{lemma}
\MovedProofMarker{p01}{0}{proof:p01}{app:preliminary-proofs}
Quaternion matrix-valued $L^p$ spaces are used to define Toeplitz matrices:
\begin{definition}[\(L^{p}\) spaces and norms]\label{def:LpH-functions}
Given \(1\le p<\infty\) and measurable \(f:D\to\HH^{m\times n}\), the $L^p$ norm of $f$ is
\[
  \|f\|_{L^{p}(D;\HH^{m\times n})}
  :=\Big(\int_{D}\|f(x)\|_{p}^{p}\,dx\Big)^{1/p},
  \qquad
  \|f\|_{L^{\infty}(D;\HH^{m\times n})}
  :=\operatorname*{ess\,sup}_{x\in D}\|f(x)\|.
\]
\val{We denote by \(L^{p}(D;\HH^{m\times n})\) the space of equivalence classes, modulo equality almost everywhere, of measurable functions with finite \(L^{p}\)-norm.}
\end{definition}

\begin{remark}[Compatibility with \(\PhiSymp\)]\label{rem:PhiLp-functional}
By Corollary~\ref{cor:s_p}, for \(1\le p\le\infty\),
\[
  \|f\|_{L^{p}(D;\HH^{m\times n})}
  =2^{-1/p}\,\|\PhiSymp\circ f\|_{L^{p}(D;\CC^{2m\times 2n})},
\]
with the usual convention \(2^{1/\infty}=1\).
\end{remark}
%When dealing with spectral distributions of matrix sequences, we use the following set.

%\begin{definition}[Essential range of the canonical spectrum]\label{def:ERH-functional}
%For measurable \(f:D\to\HH^{r\times r}\) define
%\[
%  ER_{\mathrm{can}}(f)
%  :=\bigl\{z\in\CC:\ \mu_{k}\bigl(\{x\in D:\ \min_{1\le j\le r}|\lambda_{j}^{+}(x)-z|<\varepsilon\}\bigr)>0\ \forall\varepsilon>0\bigr\}.
%\]
%\end{definition}

%\begin{remark}
%\(ER_{\mathrm{can}}(f)\) equals the essential closure of \(\bigcup_{j=1}^{r}\{\lambda_{j}^{+}(x):x\in D\}\), thus is closed.
%\end{remark}

% ---------------------------------------------------------------------
\subsection{Left, right, and sandwich trigonometric polynomials}\label{subsec:trig-sandwich}
\val{We recall single-axis quaternionic trigonometric
polynomials and the facts needed below.
Throughout, \(\I\) denotes the fixed Hamilton unit introduced in
Subsection~\ref{subsec:alg_primitives}, and all quaternionic exponential
kernels are taken in the corresponding slice \(\CC_{\I}\)}. \\
Consider the torus \(\TT^{d}=[-\pi,\pi)^{d}\) endowed with the Lebesgue measure \(d\btheta\).
For \(\bm m=(m_{1},\ldots,m_{d})\in\ZZ^{d}\) and \(\btheta=(\theta_{1},\ldots,\theta_{d})\in\TT^{d}\),
we set \(\langle\bm m,\btheta\rangle:=\sum_{j=1}^{d}m_{j}\theta_{j}\). Due to non-commutativity of the quaternion algebra, we define polynomials that might have the exponential monomial applied on the left, right, or both, to the coefficients. We call the latter case ``sandwich'' if we do not need to specify the exact disposition of variables. \\
\val{For \(d\in\NN\), we use the shorthand \([d]:=\{1,\ldots,d\}\).}
More formally, given an ordered disjoint partition \(S_{L}\dot\cup S_{R}=[d]\), define
\[
  \langle\bm m,\btheta\rangle_{S_{L}}:=\sum_{j\in S_{L}}m_{j}\theta_{j},\qquad
  \langle\bm m,\btheta\rangle_{S_{R}}:=\sum_{j\in S_{R}}m_{j}\theta_{j},
\]
and we denote left/right exponential kernels by
\[
  L_{S_{L}}(\btheta,\bm m):=\exp\bigl(-\I\langle\bm m,\btheta\rangle_{S_{L}}\bigr),\qquad
  R_{S_{R}}(\btheta,\bm m):=\exp\bigl(-\I\langle\bm m,\btheta\rangle_{S_{R}}\bigr).
\]
The purely left and purely right partitions are of special interest. The left case is defined \(L_{[d]}(\btheta,\bm m)=\exp(-\I\langle\bm m,\btheta\rangle)\), \(R_{\varnothing}\equiv 1\).
Purely right is instead \(L_{\varnothing}\equiv 1\), \(R_{[d]}(\btheta,\bm m)=\exp(-\I\langle\bm m,\btheta\rangle)\). \\
The definition of trigonometric polynomials depends on the kernel choice, as described below.

\begin{definition}[Left, right, and sandwich trigonometric polynomials]\label{def:LR-sandwich-trig-functional}
Let \(A_{\bm m}\in\HH^{m\times n}\) and \(\mathcal{F}\subset\ZZ^{d}\) be finite. Define
\[
  p_{\mathrm{L}}(\btheta)=\sum_{\bm m\in\mathcal{F}} e^{-\I\langle\bm m,\btheta\rangle}A_{\bm m},\qquad
  p_{\mathrm{R}}(\btheta)=\sum_{\bm m\in\mathcal{F}} A_{\bm m}e^{-\I\langle\bm m,\btheta\rangle},
\]
and, for a fixed partition \(S_{L}\dot\cup S_{R}=[d]\),
\[
  p^{(S_{L},S_{R})}(\btheta)
  =\sum_{\bm m\in\mathcal{F}} L_{S_{L}}(\btheta,\bm m)\,A_{\bm m}\,R_{S_{R}}(\btheta,\bm m).
\]
We call \(p^{(S_{L},S_{R})}\) a sandwich trigonometric polynomial. The degree is
\(\deg p:=\max_{\bm m\in\mathcal{F}}\|\bm m\|_{\infty}\).
\end{definition}

The Fourier coefficients and partial sums mimic the complex case, taking into account the kernel positions.

\begin{definition}[Sandwich Fourier coefficients and partial sums]\label{def:Fourier-sandwich-functional}
For \(F\in L^{1}(\TT^{d};\HH^{m\times n})\) and \(\bm m\in\ZZ^{d}\), define
\[
  \widehat{F}^{(S_{L},S_{R})}(\bm m)
  :=(2\pi)^{-d}\int_{\TT^{d}}
    L_{S_{L}}(\btheta,\bm m)\,F(\btheta)\,R_{S_{R}}(\btheta,\bm m)\,d\btheta\in\HH^{m\times n}.
\]
For \(N\in\NN\), the rectangular partial sums are
\[
  S_{N}^{(S_{L},S_{R})}F(\btheta)
  :=\sum_{\|\bm m\|_{\infty}\le N}
  \val{
  L_{S_{L}}(\btheta,-\bm m)\,
  \widehat{F}^{(S_{L},S_{R})}(\bm m)\,
  R_{S_{R}}(\btheta,-\bm m)}.
\]
The left and right cases are obtained by
\(S_{L}=[d],S_{R}=\varnothing\) and
\(S_{L}=\varnothing,S_{R}=[d]\), respectively.
\end{definition}

\val{If
\[
p^{(S_L,S_R)}(\btheta)
=
\sum_{\bm m\in\mathcal F}
L_{S_L}(\btheta,\bm m)\,
A_{\bm m}\,
R_{S_R}(\btheta,\bm m),
\]
then
\[
\widehat p^{(S_L,S_R)}(\bm k)=A_{-\bm k},
\]
where coefficients outside \(\mathcal F\) are understood to be zero.}

\begin{lemma}[Transport under the symplectic embedding]\label{lem:Phi-sandwich-functional}
Write \(F(\btheta)=Z(\btheta)+W(\btheta)\J\) with \(Z,W:\TT^{d}\to\CC_{\I}^{m\times n}\).
Define the sign-flip \(\tilde{\bm m}\in\ZZ^{d}\) by
\(\tilde m_{j}=m_{j}\) if \(j\in S_{L}\) and \(\tilde m_{j}=-m_{j}\) if \(j\in S_{R}\).
Then, for every \(\bm m\in\ZZ^{d}\),
\[
  \PhiSymp\!\bigl(\widehat{F}^{(S_{L},S_{R})}(\bm m)\bigr)
  =\begin{bmatrix}
      \widehat{Z}(\bm m) & \widehat{W}(\tilde{\bm m})\\[2pt]
      -\widehat{\overline{W}}(-\tilde{\bm m}) & \widehat{\overline{Z}}(-\bm m)
    \end{bmatrix},
\]
where the hats on the right denote ordinary complex matrix Fourier coefficients on \(\TT^{d}\).
\end{lemma}
\MovedProofMarker{p02}{0}{proof:p02}{app:preliminary-proofs}

In addition, as it can be easily checked via the embedding trick, the following result holds.

\begin{proposition}[Density in \(L^{p}\)]\label{prop:density-sandwich-functional}
For \(1\le p<\infty\), the linear span of sandwich trigonometric polynomials
\(\{p^{(S_{L},S_{R})}\}\) is dense in \(L^{p}(\TT^{d};\HH^{m\times n})\).
\end{proposition}
\MovedProofMarker{p24}{0}{proof:p24}{app:preliminary-proofs}

In the single-axis setting, using \eqref{eq:comm} to switch from different trigonometric classes is standard, and rectangular matrix-valued polynomials follow analogous arguments. In the signal analysis setting, this can be viewed as a special and trivial case of the splitting technique of general discrete Fourier transforms with two axes
(see \cite{Hitzer2007}), to convert sandwich kernels to a unidirectional one, such as the classical right kernel. For matrix-valued polynomials and Fourier coefficients, we indeed have

\begin{proposition}[Reduction to the right case]\label{prop:reduction-right}
Let \(S_{L}\dot\cup S_{R}=[d]\) and define
\[
\tilde m_{j}=
\begin{cases}
m_{j}, & j\in S_{L},\\
-m_{j}, & j\in S_{R}.
\end{cases}
\]
For finite \(\mathcal{F}\subset\ZZ^{d}\) and \(A_{\bm m}\in\HH^{m\times n}\), write \(A_{\bm m}=Z_{\bm m}+W_{\bm m}\J\) with \(Z_{\bm m},W_{\bm m}\in\CC_{\I}^{m\times n}\). With \(L_{S_{L}},R_{S_{R}}\) as in Definitions \ref{def:LR-sandwich-trig-functional}, \ref{def:Fourier-sandwich-functional}, the sandwich polynomial
\[
p^{(S_{L},S_{R})}(\btheta)
=\sum_{\bm m\in\mathcal{F}}L_{S_{L}}(\btheta,\bm m)\,A_{\bm m}\,R_{S_{R}}(\btheta,\bm m)
\]
reduces to a sum of two right polynomials (in the slice \(\CC_{\I}\)):
\begin{equation}\label{eq:poly-reduction}
p^{(S_{L},S_{R})}(\btheta)
=\sum_{\bm m\in\mathcal{F}}\Bigl(Z_{\bm m}e^{-\I\langle\bm m,\btheta\rangle}
+W_{\bm m}e^{-\I\langle\tilde{\bm m},\btheta\rangle}\J\Bigr).
\end{equation}
Moreover, for \(F=Z+W\J\in L^{1}(\TT^{d};\HH^{m\times n})\) with \(Z,W:\TT^{d}\to\CC_{\I}^{m\times n}\), the sandwich Fourier coefficients
\[
\widehat{F}^{(S_{L},S_{R})}(\bm m)
=(2\pi)^{-d}\int_{\TT^{d}}L_{S_{L}}(\btheta,\bm m)\,F(\btheta)\,R_{S_{R}}(\btheta,\bm m)\,d\btheta
\]
satisfy
\begin{equation}\label{eq:coeff-reduction}
\widehat{F}^{(S_{L},S_{R})}(\bm m)=\widehat{Z}(\bm m)+\widehat{W}(\tilde{\bm m})\,\J,
\end{equation}
where the hats on the right are the usual complex (right) Fourier coefficients on \(\TT^{d}\) in the slice \(\CC_{\I}\). In particular,
\begin{equation}\label{eq:left-right}
\widehat{F}_{\mathrm{L}}(\bm m)=\overline{\,\widehat{\,\overline{F}\,}_{\mathrm{R}}(-\bm m)\,}.
\end{equation}
\end{proposition}

\MovedProofMarker{p03}{0}{proof:p03}{app:preliminary-proofs}
\begin{remark}
The Proposition applied to the purely left kernel gives us
\(\widehat{F}_{\mathrm R}(\bm m)=\overline{\,\widehat{\,\overline{F}\,}_{\mathrm L}(-\bm m)\,}\).
\end{remark}

\section{Asymptotic spectral distribution of quaternion matrix sequences}
\label{sec:q_asymptotic_spectral}

In this section we present the asymptotic spectral and singular value distribution formalism for quaternion matrix sequences. We adopt the point of view of Theorem \ref{thm:canonical} and Proposition \ref{prop:phi_dup_sv}. Regarding the spectral analysis, canonical eigenvalues encode the eigenvalue information of quaternion matrices, while, for singular values, the analysis reduces to the complex case under the symplectic embedding \(\PhiSymp\). This allows us to inherit the well-developed framework for asymptotic spectral distributions of complex matrix sequences as developed in \cite{GLTblock2,extr2,garoni_multilevel_2018,serra-capizzano_unitarily_2002,tyrtyshnikov1998} and references therein. Throughout, eigenvalues are listed with algebraic multiplicity and singular values are listed in nonincreasing order. \val{By Lemma~\ref{lem:borel-eigenvalue-enumeration}, the eigenvalues of
a measurable complex matrix field admit a Borel enumeration, while the
ordered singular values depend continuously on the matrix entries.}

\medskip
\noindent\textbf{Eigenvalue distributions via the symplectic embedding.}

\begin{definition}[Embedding-based spectral symbol]\label{def:emb-symbol}
Let \(A_n\in\HH^{d_n\times d_n}\) be square quaternionic matrices with \(d_n\to\infty\). Let \(D\subset\RR^k\) be measurable with \(0<|D|<\infty\), and let \(G:D\to \CC^{q\times q}\) be a measurable matrix-valued function. We say that the quaternionic sequence \(\{A_n\}_n\) is \val{distributed in the sense of the eigenvalues as \(G\) over \(D\)}, and we write
\[
   \{A_n\}_n \sim_{\lambda} (G,D),
\]
if and only if the complex embeddings satisfy \(\{\PhiSymp(A_n)\,\}_n \sim_{\lambda} (G,D)\) in the classical sense. Concretely, for every \(F\in C_c(\CC)\),
\begin{equation}\label{eq:def-ev-emb}
   \lim_{n\to\infty}\frac{1}{2d_n}\sum_{j=1}^{2d_n}
      F\!\bigl(\lambda_j(\PhiSymp(A_n))\bigr)
   \;=\; \frac{1}{|D|}\int_D \frac{1}{q}\sum_{i=1}^{q}
      F\!\bigl(\lambda_i(G(x))\bigr)\,dx,
\end{equation}
where the eigenvalues \(\lambda_j(\PhiSymp(A_n))\) and \(\lambda_i(G(x))\) are listed with algebraic multiplicity.
\end{definition}

Let \(\lambda_1^{+}(A_n),\dots,\lambda_{d_n}^{+}(A_n)\) denote the canonical eigenvalues of \(A_n\), \val{as defined in Theorem~\ref{thm:canonical}}. Then the eigenvalues of \(\PhiSymp(A_n)\) occur in conjugate pairs \(\{z,\overline z\}\); real eigenvalues appear as pairs \(\{z,z\}\). Therefore, for every \(F\in C_c(\CC)\),
\begin{equation}\label{eq:norm-ev-can}
\frac{1}{2d_n}\sum_{j=1}^{2d_n}F\bigl(\lambda_j(\PhiSymp(A_n))\bigr)
\;=\;
\frac{1}{d_n}\sum_{j=1}^{d_n}\widetilde F\bigl(\lambda^{+}_j(A_n)\bigr),
\qquad
\widetilde F(z):=\tfrac12\bigl(F(z)+F(\overline z)\bigr).
\end{equation}
Equation \eqref{eq:norm-ev-can} shows that the embedding-based definition averages each conjugate pair \(\{z,\overline z\}\) via \(\widetilde F\). In particular, the two sides coincide when \(F\) is conjugation-symmetric, i.e., \(F(\overline z)=F(z)\).

\medskip
\noindent\textbf{Alternative quaternion-side formulation for eigenvalues.}
 If an entirely quaternionic formulation is preferred, then the following definition can be taken into consideration.

\begin{definition}[Canonical-eigenvalue distribution]\label{def:can-ev-symbol}
Let \(A_n\in\HH^{d_n\times d_n}\) with \(d_n\to\infty\), and let \(\lambda_1^{+}(A_n),\dots,\lambda_{d_n}^{+}(A_n)\) be the canonical eigenvalues as in Theorem~\ref{thm:canonical}. Let \(D\subset\RR^k\) be measurable with \(0<|D|<\infty\), and let \(G\in L^0(D;\CC^{q\times q})\). We write that \(\{A_n\}\sim_{\lambda^{+}}(G,D)\) if, for every \(F\in C_c(\CC)\) with \(F(\overline z)=F(z)\),
\[
\lim_{n\to\infty}\frac{1}{d_n}\sum_{j=1}^{d_n}F\bigl(\lambda_j^{+}(A_n)\bigr)
\;=\;
\frac{1}{|D|}\int_D \frac{1}{q}\sum_{i=1}^q F\bigl(\lambda_i(G(x))\bigr)\,dx.
\]
\end{definition}

\begin{remark}\label{rem:eq-emb-can}
By \eqref{eq:norm-ev-can}, if $\{A_n\}\sim_{\lambda}(G,D)$ in the sense of Definition~\ref{def:emb-symbol}, then $\{A_n\}\sim_{\lambda^{+}}(G,D)$ in the sense of Definition~\ref{def:can-ev-symbol}. Conversely, if for a.e.\ $x\in D$ the multiset $\{\lambda_i(G(x))\}_{i=1}^{q}$ is invariant under complex conjugation, then $\{A_n\}\sim_{\lambda^{+}}(G,D)$ implies $\{A_n\}\sim_{\lambda}(G,D)$.
\end{remark}

\begin{remark}[\val{Measure-preserving reparametrizations}]
\val{Let \(\varphi:D\to D\) be measurable and measure-preserving, meaning that}
\val{\[
|\varphi^{-1}(E)|=|E|
\]}
\val{for every measurable set \(E\subseteq D\). Then}
\val{\[
\{A_n\}_n\sim_{\lambda}(G,D)
\quad\Longleftrightarrow\quad
\{A_n\}_n\sim_{\lambda}(G\circ\varphi,D),
\]}
\val{The analogous equivalence holds for \(\sim_{\lambda^{+}}\), with the
conjugation-symmetric test functions required in
Definition~\ref{def:can-ev-symbol}. The measure-preserving property leaves
the integrals in the defining limit relations unchanged when \(G\) is replaced
by \(G\circ\varphi\). Hence, an eigenvalue symbol is determined
only up to measure-preserving reparametrization of its spatial variable.}
\end{remark}

\medskip
\noindent\textbf{Singular value distributions via the symplectic embedding.}

\begin{definition}[Embedding-based singular value symbol]\label{def:sv-symbol-rect}
Let \(A_n\in\HH^{m_n\times n_n}\) be rectangular quaternionic matrices and set \(r_n:=\min\{m_n,n_n\}\to\infty\). Let \(D\subset\RR^k\) be measurable with \(0<|D|<\infty\), and let \(H\in L^0(D;\CC^{s\times t})\) be measurable. We write that \(\{A_n\}_n\) is \val{distributed in the sense of the singular values as \(H\) over \(D\)}, and we write
\[
   \{A_n\}_n \sim_{\sigma} (H,D),
\]
if and only if \(\{\,\PhiSymp(A_n)\,\}_n \sim_{\sigma} (H,D)\) in the classical sense, i.e., for every \(\psi\in C_c([0,\infty))\),
\begin{equation}\label{eq:def-sv-emb}
   \lim_{n\to\infty}\frac{1}{2r_n}
   \sum_{i=1}^{2r_n}
      \psi\!\bigl(\sigma_i(\PhiSymp(A_n))\bigr)
   \;=\; \frac{1}{|D|}\int_D \frac{1}{\min\{s,t\}}
      \sum_{i=1}^{\min\{s,t\}}\psi\!\bigl(\sigma_i(H(x))\bigr)\,dx.
\end{equation}
\end{definition}

\val{The same invariance under measure-preserving reparametrizations holds for
singular-value symbols.}

\begin{remark}[Duplication of singular values under the embedding]\label{rem:dup-sv}
By Proposition \ref{prop:phi_dup_sv}, the singular values of \(\PhiSymp(A_n)\) are exactly those of \(A_n\), each duplicated. \val{Consequently, for every \(\psi\in C_c([0,\infty))\),}

\begin{equation}\label{eq:norm-sv}
\frac{1}{2r_n}\sum_{i=1}^{2r_n}
\psi\bigl(\sigma_i(\PhiSymp(A_n))\bigr)
\;=\;
\frac{1}{r_n}\sum_{i=1}^{r_n}
\psi\bigl(\sigma_i(A_n)\bigr),
\qquad \forall\,\psi\in C_c([0,\infty)).
\end{equation}
\val{Thus the empirical singular-value distribution may be evaluated directly
on the quaternionic matrices.}
\end{remark}

%------------------------------------------------------------------
\subsection{Definition of a.c.s.}
We introduce the notion of rectangular a.c.s. over \(\HH\). Throughout this section, \(\|\cdot\|\) denotes the operator (spectral) norm on quaternionic matrices (cf.\ Definition~\ref{def:schatten} for Schatten \(p\)-norms).

\begin{definition}[Approximating class of sequences]\label{def:rect-acs}
Let \(\{A_{n}\}_n\) be a matrix sequence with \(A_n \in \HH^{d_{n}\times e_{n}}\) and set
\(r_n:=d_n\wedge e_n \to \infty\).
\val{A class of matrix sequences
\(\bigl\{\{B_{n,m}\}_{n}\bigr\}_{m\in\NN}\)} with \(B_{n,m}\in\HH^{d_{n}\times e_{n}}\)
is an \emph{approximating class of sequences} for \(\{A_{n}\}_{n}\)
if, for every fixed \(m\), there exist \(R_{n,m},N_{n,m}\in\HH^{d_{n}\times e_{n}}\) and
\(n_m\) such that, for all \(n\ge n_m\),
\begin{align}
A_{n}&=B_{n,m}+R_{n,m}+N_{n,m},\label{eq:rect-acs-decomp}\\
\frac{\operatorname{rank}_{\HH}(R_{n,m})}{r_{n}}&\le c(m),\qquad
\|N_{n,m}\|\le \omega(m),\label{eq:rect-acs-bounds}
\end{align}
where \(c(m)\to0\) and \(\omega(m)\to0\) as \(m\to\infty\).
We write \(\{B_{n,m}\}_n \xrightarrow[\HH]{\mathrm{a.c.s.}} \{A_n\}_n\) when the quaternion side must be explicit, otherwise \(\{B_{n,m}\}_n \xrightarrow{\mathrm{a.c.s.}} \{A_n\}_n\), and it must be emphasized that the limit is taken with respect to the external variable $m$.
\end{definition}
\noindent\val{The term
\emph{approximating class of sequences} is standard in the analysis of matrix
sequences. Equivalently, a class of matrix sequences may be viewed as a sequence of matrix
sequences.}\par
\noindent\val{For multi-indices, we adopt the conventions of
\cite[Sections~2.1.2-2.1.3]{garoni_multilevel_2018}. Thus inequalities
between multi-indices are componentwise, lexicographic order is used to
enumerate entries, and
\[
\bn\to\infty
\quad\Longleftrightarrow\quad
\min_{1\le\ell\le d}n_\ell\to\infty.
\]
A multilevel matrix sequence extracted from a family
\(\{A_{\bn}\}_{\bn\in\NN^d}\) is any scalar-indexed sequence
\(\{A_{\bn(n)}\}_n\) with \(\bn(n)\to\infty\). All a.c.s. statements below
for multi-indexed families are understood along every such extraction, while
the external approximation index \(m\in\NN\) remains scalar.}\par
\val{In complex matrix sequences, Definition~\ref{def:rect-acs} applies verbatim with
\(\HH\) and \(\operatorname{rank}_{\HH}\) replaced by
\(\CC\) and \(\operatorname{rank}_{\CC}\), respectively.} \\
\smallskip

For the properties of a.c.s. we need the following definition\val{, where for a finite set \(E\), we denote its cardinality by \(\#E\).}

\begin{definition}[Sparsely unbounded quaternion matrix sequences]\label{def:suH}
Let $A_n\in\HH^{m_n\times n_n}$ be (rectangular) quaternionic matrices and set
$r_n:=\min\{m_n,n_n\}\to\infty$. Denote the singular values of $A_n$
(by definition, the square roots of the eigenvalues of $A_n^*A_n$,
with $^*$ the quaternionic conjugate transpose) in nonincreasing order by
$\sigma_1(A_n)\ge\cdots\ge\sigma_{r_n}(A_n)\ge0$.
We write that the sequence $\{A_n\}_n$ is \emph{sparsely unbounded} (s.u.) if
there exists a function $r:(0,\infty)\to[0,1]$ with $\lim_{M\to\infty} r(M)=0$
such that, for every $M>0$, there is $n_M$ with the property that, for all
$n\ge n_M$,
\[
\frac{1}{r_n}\,\#\Bigl\{\,i\in\{1,\dots,r_n\}:\ \sigma_i(A_n)>M\,\Bigr\}
\;\le\; r(M).
\]
\end{definition}

For later use we record the basic norm/rank relations induced by the symplectic embedding.
\noindent
\val{For positive-integer sequences
\(a=(a_n)\), \(b=(b_n)\), and \(c=(c_n)\), write
\[
\operatorname{PC}(a,b,c)
\quad\Longleftrightarrow\quad
(a_n\wedge b_n)+(b_n\wedge c_n)
=
O(a_n\wedge c_n)
\qquad(n\to\infty).
\]
For complex matrix sequences, sparse unboundedness is defined exactly as
in Definition~\ref{def:suH}, with \(\HH\) and
\(\operatorname{rank}_{\HH}\) replaced by \(\CC\) and
\(\operatorname{rank}_{\CC}\), respectively.}

\begin{lemma}[Norm and rank under \(\PhiSymp\)]\label{lem:norm-rank-phi}
For all \(X\in\HH^{d\times e}\) and \(1\le p\le\infty\),
\[
\|\PhiSymp(X)\|=\|X\|,\qquad
\|\PhiSymp(X)\|_{S_p}=2^{1/p}\|X\|_{p},\qquad
\operatorname{rank}_{\CC}\!\bigl(\PhiSymp(X)\bigr)=2\,\operatorname{rank}_{\HH}(X).
\]
\end{lemma}
\MovedProofMarker{p04}{0}{proof:p04}{app:preliminary-proofs}

Since \(\PhiSymp\) preserves the operator norm and duplicates nonzero singular values, the following equivalence with the complex a.c.s.\ framework (see \cite{capizzano_distribution_2001}) holds.

\begin{theorem}[Embedding equivalence]\label{thm:rect-embedding}
\begingroup
\color{red}
Let \(\{A_n\}_n\) be a matrix sequence with
\(A_n\in\HH^{d_n\times e_n}\) and
\(r_n:=d_n\wedge e_n\to\infty\), and let
\(\bigl\{\{B_{n,m}\}_n\bigr\}_{m\in\NN}\) be given, with
\(B_{n,m}\in\HH^{d_n\times e_n}\).
Then the following are equivalent:
\begin{enumerate}
\item[\textup{(a)}]
\(\bigl\{\{B_{n,m}\}_n\bigr\}_{m\in\NN}\) is a quaternionic a.c.s.\ for
\(\{A_n\}_n\).

\item[\textup{(b)}]
\(\bigl\{\{\PhiSymp(B_{n,m})\}_n\bigr\}_{m\in\NN}\) is a complex a.c.s.\ for
\(\{\PhiSymp(A_n)\}_n\).
\end{enumerate}
\endgroup
\end{theorem}

\MovedProofMarker{p05}{1}{proof:p05}{app:preliminary-proofs}

%------------------------------------------------------------------
\textbf{Consequences.}
%------------------------------------------------------------------
\val{Together, Theorem~\ref{thm:complex-acs-properties} and the embedding equivalence in
Theorem~\ref{thm:rect-embedding} yield the following quaternionic counterparts,
with $r_n=d_n\wedge e_n$:}

\begin{enumerate}
\item[\textup{ACS\(_\HH\)1.}]
\(\{A_n\}\sim_{\sigma} F\) if and only if there exist matrix sequences
\(\{B_{n,m}\}_n\sim_{\sigma} F_m\) with
\(\{B_{n,m}\}_n \xrightarrow[\HH]{\mathrm{a.c.s.}} \{A_n\}_n\) and \(F_m\to F\) in measure.

\item[\textup{ACS\(_\HH\)2.}]
If every \(A_n\) is Hermitian and square, then
\(\{A_n\}\sim_{\lambda} f\) if and only if there exist Hermitian sequences
\(\{B_{n,m}\}_n\sim_{\lambda} f_m\) with
\(\{B_{n,m}\}_n \xrightarrow[\HH]{\mathrm{a.c.s.}} \{A_n\}_n\) and \(f_m\to f\) in measure.

\item[\textup{ACS\(_\HH\)3.}]
\val{Let
\(A_n,B_{n,m}\in\HH^{d_n\times e_n}\) and
\(A'_n,B'_{n,m}\in\HH^{d'_n\times e'_n}\).
Assume
\[
\{B_{n,m}\}_n
\xrightarrow[\HH]{\mathrm{a.c.s.}}
\{A_n\}_n,
\qquad
\{B'_{n,m}\}_n
\xrightarrow[\HH]{\mathrm{a.c.s.}}
\{A'_n\}_n.
\]
Then}
\begin{align*}
&\{B_{n,m}^*\}_n
\xrightarrow[\HH]{\mathrm{a.c.s.}}
\{A_n^*\}_n,\\
&\{\alpha B_{n,m}+\beta B'_{n,m}\}_n
\xrightarrow[\HH]{\mathrm{a.c.s.}}
\{\alpha A_n+\beta A'_n\}_n
\quad
\val{\text{if }(d'_n,e'_n)=(d_n,e_n),}\\
&\{B_{n,m}B'_{n,m}\}_n
\xrightarrow[\HH]{\mathrm{a.c.s.}}
\{A_nA'_n\}_n
\quad
\val{\text{if }d'_n=e_n,\
\{A_n\}_n,\{A'_n\}_n\text{ are s.u., and }
\operatorname{PC}(d,e,e'),}\\
&\{B_{n,m}C_n\}_n
\xrightarrow[\HH]{\mathrm{a.c.s.}}
\{A_nC_n\}_n
\quad
\val{\text{if }C_n\in\HH^{e_n\times f_n},\
\{C_n\}_n\text{ is s.u., and }
\operatorname{PC}(d,e,f).}
\end{align*}
\val{Here \(\alpha,\beta\in\HH\) act on the left, and each product
a.c.s.\ is normalized by the minimum dimension of the product output.
The corresponding \(\operatorname{PC}\) condition ensures that this
minimum dimension tends to infinity.} \\
\val{The rectangular a.c.s.\ normalization is that of
\cite[Def.~2.6]{rectangular_glt_2022}. For the classical common-size
a.c.s.\ product proof, see
\cite[Prop.~5.5, pp.~87-88]{garoni_Toeplitz_2017}; for fixed-block
rectangular GLT multiplication, see
\cite[Thm.~4.5]{rectangular_glt_2022}. The varying-aspect-ratio
rectangular a.c.s.\ product assertions above follow from the same argument,
with the output-rank normalization controlled by
\(\operatorname{PC}\).}

\MovedProofMarker{p25}{0}{proof:p25}{app:preliminary-proofs}

\item[\textup{ACS\(_\HH\)4.}]
If \(A_n=[A_{n,ij}]_{i,j=1}^{s}\) and \(B_n^{(m)}=[B^{(m)}_{n,ij}]_{i,j=1}^{s}\) with
\(\{B^{(m)}_{n,ij}\}_n \xrightarrow[\HH]{\mathrm{a.c.s.}} \{A_{n,ij}\}_n\) for every \(i,j\), then
\(\{B_n^{(m)}\}_n \xrightarrow[\HH]{\mathrm{a.c.s.}} \{A_n\}_n\).

\item[\textup{ACS\(_\HH\)5.}]
(Practical norm test.)
Let \(p\in[1,\infty]\) and assume that for each \(m\) there exists \(n_m\) such that, for all \(n\ge n_m\),
\[
   \|A_n-B_{n,m}\|_{p}\ \le\ \varepsilon(m,n)\, r_n^{1/p},
\]
with the convention \(r_n^{1/\infty}=1\), where \(\|\cdot\|_{p}\) denotes the Schatten \(p\)-norm (Definition~\ref{def:schatten}),
\(\lim_{m\to\infty}\limsup_{n\to\infty}\varepsilon(m,n)=0\),
and \(r_n=d_n\wedge e_n\).
Then \(\{B_{n,m}\}_n \xrightarrow[\HH]{\mathrm{a.c.s.}} \{A_n\}_n\).
\end{enumerate}

%========================================================

\section{Single-axis quaternion Toeplitz sequences}
\label{sec:QFT-Toeplitz}
This section contains the basic asymptotic theory for multilevel block Toeplitz sequences generated by multivariate matrix-valued quaternion functions in $L^1$. Our starting point is, for most results, \cite{lin_hermitian_2025} for the case of unilevel scalar Toeplitz matrices generated by bounded functions. Unlike the complex case, the Fourier quaternion kernel is not unique and can depend on two orthogonal imaginary axes (see \cite{ell_quaternion_2013} for a discussion of the topic). Here, we focus our attention on a single-axis and consider single-axis Toeplitz sequences. Under the single-axis hypothesis, the question of how to apply the Fourier kernel remains due to the non-commutativity of quaternions; however, we will see that even if, numerically, a given generating function yields different quaternion Toeplitz matrices depending on the kernel, algebraically, every Toeplitz matrix sequence belongs to the same class, and this is independent of the construction of Fourier coefficients. \val{Furthermore, in the current} section we recover all the fundamental classical properties of Toeplitz sequences that one can find in standard GLT theory surveys, such as \cite[Chapter 6]{garoni_Toeplitz_2017},
\cite[Chapter 4]{garoni_multilevel_2018}, \cite{GLTblock2}.

\noindent
 We write $\TT^d=[-\pi,\pi)^d$ and use the Lebesgue measure $d\btheta$ on $\TT^d$\val{.}

\vspace{.5\baselineskip}
%---------------------------------------------------------------------

%-----------------------------------------------------------------------
\subsection{$d$-level quaternion Toeplitz matrices}
\label{subsec:Toeplitz-1axis}
To construct quaternion Toeplitz matrices we use the kernel $L_{S_L}$ and $R_{S_R}$, applied respectively on the left and right, given in Definition \ref{def:LR-sandwich-trig-functional}, and the Fourier coefficients \(F^{(S_L,S_R)}(\bm m)\) are as in Definition \ref{def:Fourier-sandwich-functional}. \\
Thus, for $\bm n=(n_{1},\dots ,n_{d})\in\mathbb{N}^{d}$ and a fixed multi-index set
\begin{equation}\label{eq:Lambda}
  \Lambda_{\bm n}:=\{0,\dots ,n_{1}-1\}\times\cdots\times\{0,\dots ,n_{d}-1\},
\end{equation}
endowed with the standard lexicographical order, we define the $d$-level (block) Toeplitz matrix as
\begin{equation}\label{eq:Toeplitz}
  T_{ \bm n}^{(S_L,S_R)}\!\bigl(F\bigr)
  :=\bigl[\widehat F^{(S_L,S_R)}(\alpha-\beta)\bigr]_{\alpha,\beta\in\Lambda_{\mathbf n}}
  \in\HH^{(N_{\bm n}s)\times(N_{\bm n}t)},
  \qquad N_{\bm n}=\prod_{\ell=1}^{d}n_{\ell}.
\end{equation}
\val{For the purely left and purely right kernels, respectively, we use the abbreviations
\[
T_{\bn}^{(L)}(F):=T_{\bn}^{([d],\varnothing)}(F),
\qquad
T_{\bn}^{(R)}(F):=T_{\bn}^{(\varnothing,[d])}(F).
\]}
The structure is analogous to that of its complex counterpart; however, the algebra this structure generates is different. First, we observe that all Toeplitz matrices belong to the same class. This is because the underlying kernel for Toeplitz matrices is interchangeable.

\begin{proposition}[Right-kernel reduction for Toeplitz matrices]
\label{prop:Toeplitz-right-reduction}
Let $F=Z+W\J\in L^{1}(\TT^{d};\HH^{s\times t})$ with $Z,W:\TT^{d}\to\CC_{\I}^{s\times t}$ and fix
an ordered partition $S_{L}\dot\cup S_{R}=[d]$. Consider
\[
(r_{S_{L}}(\btheta))_{j}=
\begin{cases}
-\theta_{j},& j\in S_{L},\\
\ \theta_{j},& j\in S_{R},
\end{cases}
\]
and define
\begin{equation}\label{eq:symb_transp}
    H(\btheta):=Z(\btheta)+\bigl(W\circ r_{S_{L}}\bigr)(\btheta)\,\J.
\end{equation}

Then, for every $\bn\in\mathbb{N}^{d}$,
\begin{equation}\label{eq:Toeplitz-reduction}
T_{\bn}^{(S_{L},S_{R})}(F)=T_{\bn}^{(R)}(H).
\end{equation}
\val{In the left-kernel case $(S_{L}=[d],S_{R}=\varnothing)$, we obtain}
\begin{equation}\label{eq:left-right-reduction}
T_{\bn}^{(L)}(F)=T_{\bn}^{(R)}\!\Bigl(Z+\bigl(W\circ(-\mathrm{id})\bigr)\J\Bigr).
\end{equation}
Moreover, for a fixed $\bn\in\mathbb N^d$ and the Cartesian set
\[
\Delta(\bn):=\{\,\bm k\in\ZZ^d:\ -n_j+1\le k_j\le n_j-1,\ j\in[d]\,\},
\]
we have
\begin{equation}\label{eq:evenness-criterion-box}
T_{\bn}^{(L)}(F)=T_{\bn}^{(R)}(F)
\ \Longleftrightarrow\
\widehat W(\bm k)=\widehat W(-\bm k)\quad\forall\,\bm k\in\Delta(\bn).
\end{equation}
\val{Moreover,}
\begin{equation}\label{eq:evenness-criterion-global}
\val{T_{\bn}^{(L)}(F)=T_{\bn}^{(R)}(F)\quad\text{for every }\bn\in\NN^d
\ \Longleftrightarrow\
\widehat W(\bm k)=\widehat W(-\bm k)\quad\forall\,\bm k\in\ZZ^d,}
\end{equation}
and the latter is equivalent to $W(\btheta)=W(-\btheta)$ a.e. on $\TT^d$.
\end{proposition}

\MovedProofMarker{p06}{0}{proof:p06}{app:toeplitz-proofs}

\subsection{Linearity, Adjoint, and Hermitian properties}
\label{subsec:AdjointFormal}

Here we report a few results concerning linearity, adjoint and Hermitian character in the quaternionic setting.
%------------------------------------------------------------
% Linearity of Toeplitz constructors and Toeplitz classes
%------------------------------------------------------------

\begin{lemma}[$\CC_{\I}$-linearity of sandwich Fourier coefficients]
\label{lem:Fourier-linearity}
Fix an ordered partition $S_{L}\dot\cup S_{R}=[d]$.
For $F,G\in L^{1}(\TT^{d};\HH^{s\times t})$ and $a,b\in\CC_{\I}$, define
\begin{equation}\label{eq:Fourier-sandwich}
\widehat{F}^{(S_{L},S_{R})}(\bm k)
:=(2\pi)^{-d}\!\int_{\TT^{d}}
e^{-\I\langle\bm k,\boldsymbol\theta\rangle_{S_{L}}}\,F(\boldsymbol\theta)\,
e^{-\I\langle\bm k,\boldsymbol\theta\rangle_{S_{R}}}\,d\boldsymbol\theta,
\qquad \bm k\in\mathbb Z^{d}.
\end{equation}
Then, for all $\bm k\in\mathbb Z^{d}$,
\begin{equation}\label{eq:Fourier-CI-linear}
\widehat{\,aF+bG\,}^{(S_{L},S_{R})}(\bm k)
= a\,\widehat{F}^{(S_{L},S_{R})}(\bm k)
+ b\,\widehat{G}^{(S_{L},S_{R})}(\bm k).
\end{equation}
\end{lemma}
\MovedProofMarker{p07}{0}{proof:p07}{app:toeplitz-proofs}

\begin{proposition}[\val{$\CC_{\I}$-linearity of quaternion Toeplitz matrices}]
\label{prop:Toeplitz-linear}
Let $T_{\bn}^{(S_{L},S_{R})}(F)\in\HH^{(N_{\bn}s)\times(N_{\bn}t)}$
be defined by
\[
\bigl[T_{\bn}^{(S_{L},S_{R})}(F)\bigr]_{\alpha,\beta}
=\widehat{F}^{(S_{L},S_{R})}(\bm k),
\qquad
\bm k:=\alpha-\beta,\ \ \alpha,\beta\in\Lambda_{\bn},
\]
with $\Lambda_{\bn}$ as in \eqref{eq:Lambda}.
Then, for all $a,b\in\CC_{\I}$ and $F,G\in L^{1}(\TT^{d};\HH^{s\times t})$,
\begin{equation}\label{eq:Toeplitz-CI-linear}
T_{\bn}^{(S_{L},S_{R})}(aF+bG)
= a\,T_{\bn}^{(S_{L},S_{R})}(F)
+ b\,T_{\bn}^{(S_{L},S_{R})}(G).
\end{equation}
\end{proposition}
\MovedProofMarker{p08}{0}{proof:p08}{app:toeplitz-proofs}

Now, we characterize the adjoints of quaternion Toeplitz matrices and we provide criteria for checking the Hermitian character. Consider
\[
  F:\TT^{d}\to\HH^{s\times t},\qquad
  F(\btheta)=Z(\btheta)+W(\btheta)\J,
\]
with $Z,W\in L^{1}(\TT^{d};\CC_{\I}^{\,s\times t})$.
\val{For $\bk=(k_1,\ldots,k_d)\in\ZZ^d$, define its coordinate projections onto
$S_L$ and $S_R$ by
\[
(\bk_L)_j:=
\begin{cases}
k_j, & j\in S_L,\\
0, & j\notin S_L,
\end{cases}
\qquad
(\bk_R)_j:=
\begin{cases}
k_j, & j\in S_R,\\
0, & j\notin S_R.
\end{cases}
\]
Thus $\bk=\bk_L+\bk_R$.}
\val{With this notation, for any ordered partition $S_L\dot\cup S_R=[d]$ we have}
\begin{equation}\label{eq:mv-Sandwich-coeffs}
  \widehat F^{(S_L,S_R)}(\bm k)
  =\widehat Z(\,\bm k_L+\bm k_R\,)+\widehat W(\,\bm k_L-\bm k_R\,)\,\J,
\end{equation}
where the hats denote classical complex Fourier coefficients.

\begin{theorem}[Adjoint identities]\label{thm:AdjointFormal-mv}
Let $F:\TT^{d}\to\HH^{s\times t}$ with $F(\btheta)=Z(\btheta)+W(\btheta)\J$, $Z,W\in L^{1}(\TT^{d};\CC_{\I}^{\,s\times t})$, and let $\bn\in\NN^{d}$. Then
\begin{equation}\label{eq:adjointnessCrit}
 \bigl(T_{\bn}^{(L)}(F)\bigr)^{*}=T_{\bn}^{(R)}(F^{*}),\qquad
 \bigl(T_{\bn}^{(R)}(F)\bigr)^{*}=T_{\bn}^{(L)}(F^{*}),\qquad
 \bigl(T_{\bn}^{(S_L,S_R)}(F)\bigr)^{*}=T_{\bn}^{(S_R,S_L)}(F^{*}).
\end{equation}
\end{theorem}

\MovedProofMarker{p09}{0}{proof:p09}{app:toeplitz-proofs}

We now generalize the Hermitian character proved in \cite{lin_hermitian_2025} for unilevel scalar right Toeplitz matrices. As in the basic case, the Hermitian character is completely determined by the Fourier coefficient behavior of the generating functions. In particular, the following result holds true.

\begin{theorem}[Hermitian character for single-axis quaternion Toeplitz matrices]\label{thm:selfadjoint-single-axis-mv}
Assume $s=t$ and let $F(\btheta)=Z(\btheta)+W(\btheta)\J$ with
$Z,W\in L^{1}(\TT^{d};\CC_{\I}^{\,s\times s})$.
Fix an ordered partition $S_L\dot\cup S_R=[d]$.
For $\bn\in\NN^{d}$ set
\[
T_{\bn}^{(S_L,S_R)}(F)
=\bigl[\widehat F^{(S_L,S_R)}(\alpha-\beta)\bigr]_{\alpha,\beta\in\Lambda_{\bn}}.
\]
Then the following are equivalent:
\begin{enumerate}[label=(\roman*),leftmargin=*]
\item $T_{\bn}^{(S_L,S_R)}(F)$ is Hermitian for all $\bn$;
\item $Z(\btheta)$ is essentially Hermitian on $\TT^{d}$ and
      $W(-\btheta)=-\,W(\btheta)^{\mathsf T}$ a.e.\ on $\TT^{d}$.
\end{enumerate}
In particular, the claim does not depend on the choice of the ordered partition.
\end{theorem}

\MovedProofMarker{p10}{2}{proof:p10}{app:toeplitz-proofs}

\subsection{Symplectic embedding of Toeplitz matrices and spectral/singular value distributions of Toeplitz matrix sequences}
\label{subsec:symp-embed-Toeplitz}

The aim of the present section is to derive spectral/singular value distributions for block multilevel Toeplitz matrix sequences in the quaternion setting. We start by defining the proper embeddings.

\subsubsection{Left/right and sandwich embeddings}

Throughout this section, let $F=Z+W\J\in L^{1}(\TT^{d};\HH^{s\times t})$ with $Z,W:\TT^{d}\to\CC_{\I}^{\,s\times t}$ slice-valued.
Let $\PhiSymp:\HH^{s\times t}\to\CC^{(2s)\times(2t)}$ be the block symplectic embedding on $s\times t$ matrices
\[
\PhiSymp(Z+W\J)=
\begin{bmatrix}
Z & W \\[2pt]
-\overline{W} & \overline{Z}
\end{bmatrix},
\qquad Z,W\in\CC_{\I}^{\,s\times t}.
\]
If $A\in\HH^{N_{\bn}s\times N_{\bn}t}$ is an $s\times t$-block matrix, we apply $\PhiSymp$ $s\times t$-blockwise, meaning
\[
\PhiSymp(A):=\bigl[\PhiSymp(A_{\alpha\beta})\bigr]_{\alpha,\beta\in {\Lambda_{\bn}}} \in\CC^{(2sN_{\bn})\times(2tN_{\bn})},
\]
{with \(A_{\alpha,\beta} \in \HH^{s \times t}\)}. We will use the following identity
\begin{equation}\label{eq:Phi-cov-used}
\PhiSymp(\lambda A\rho)
=\operatorname{diag}(\lambda,\overline{\lambda})\,\PhiSymp(A)\,\operatorname{diag}(\rho,\overline{\rho})
\quad\bigl(\lambda\in\CC_{\I}^{\,s\times s},\ \rho\in\CC_{\I}^{\,t\times t}\bigr),
\end{equation}
where $\operatorname{diag}(\lambda,\overline{\lambda})$ and $\operatorname{diag}(\rho,\overline{\rho})$ are $2\times2$ block-diagonal matrices with blocks of sizes $s\times s$ and $t\times t$, respectively. When $\lambda,\rho\in\CC_{\I}$ are scalars we read them as $\lambda I_{s}$ and $\rho I_{t}$.

For a complex matrix symbol $G:\TT^{d}\to\CC^{(2s)\times(2t)}$, we write
\[
\widehat G(\bm k)=(2\pi)^{-d}\!\int_{\TT^{d}} G(\boldsymbol\theta)\,e^{-\I\langle \bm k,\boldsymbol\theta\rangle}\,d\boldsymbol\theta,
\qquad
T_{\bn}(G):=\bigl[\widehat G(\alpha-\beta)\bigr]_{\alpha,\beta\in\Lambda_{\bn}} .
\]

\begin{theorem}[Left/right embedding]\label{thm:unified-LR-rect}
With the decomposition $F(\btheta)=Z(\btheta)+W(\btheta)\J$ from Lemma~\ref{lem:Phi-sandwich-functional}, define the complex symbols
\[
G_{\mathrm L}(\btheta)=
\begin{pmatrix}
Z(\btheta) & W(\btheta)\\[2pt]
-\overline{W(-\btheta)} & \overline{Z(-\btheta)}
\end{pmatrix},
\qquad
G_{\mathrm R}(\btheta)=
\begin{pmatrix}
Z(\btheta) & W(-\btheta)\\[2pt]
-\overline{W(\btheta)} & \overline{Z(-\btheta)}
\end{pmatrix}.
\]
Then, for every $\bn\in\NN^{d}$,
\[
\PhiSymp\!\bigl(T_{\bn}^{(\mathrm L)}(F)\bigr)=T_{\bn}(G_{\mathrm L}),
\qquad
\PhiSymp\!\bigl(T_{\bn}^{(\mathrm R)}(F)\bigr)=T_{\bn}(G_{\mathrm R}).
\]
\end{theorem}

\MovedProofMarker{p11}{0}{proof:p11}{app:toeplitz-proofs}

\medskip

We now treat the two-sided sandwich case using the involution $r_{S_L}:\TT^{d}\to\TT^{d}$ from Proposition~\ref{prop:Toeplitz-right-reduction}.
Note that $r_{S_L}\circ(-\mathrm{id})$ flips the $S_{R}$-coordinates and that $-\,r_{S_L}(-\btheta)=r_{S_L}(\btheta)$.

\begin{theorem}[Sandwich embedding]\label{thm:sandwich+perm-rect-r}
Let $F(\btheta)=Z(\btheta)+W(\btheta)\J$ with $Z,W:\TT^{d}\to\CC_{\I}^{\,s\times t}$ and fix an ordered partition $S_L\dot\cup S_R=[d]$.
Define the complex symbol
\[
G_{S_L,S_R}(\btheta)=
\begin{pmatrix}
Z(\btheta) & W\!\bigl(r_{S_L}(-\btheta)\bigr)\\[2pt]
-\overline{W\!\bigl(r_{S_L}(\btheta)\bigr)} & \overline{Z(-\btheta)}
\end{pmatrix}.
\]
Then, for every $\bn\in\NN^{d}$,
\begin{equation}\label{eq:sandwich-embed-r}
\PhiSymp\!\bigl(T_{\bn}^{(S_L,S_R)}(F)\bigr)=T_{\bn}\!\bigl(G_{S_L,S_R}\bigr).
\end{equation}
\val{In the purely left case $S_{L}=[d]$, $S_{R}=\varnothing$, one recovers}
\[
G_{\mathrm L}(\btheta)=
\begin{pmatrix}
Z(\btheta) & W(\btheta)\\[2pt]
-\overline{W(-\btheta)} & \overline{Z(-\btheta)}
\end{pmatrix}.
\]
\val{In the purely right case $S_{L}=\varnothing$, $S_{R}=[d]$, one obtains}
\[
G_{\mathrm R}(\btheta)=
\begin{pmatrix}
Z(\btheta) & W(-\btheta)\\[2pt]
-\overline{W(\btheta)} & \overline{Z(-\btheta)}
\end{pmatrix}.
\]
\end{theorem}

\MovedProofMarker{p12}{0}{proof:p12}{app:toeplitz-proofs}

\subsubsection{Singular value and spectral distribution via the embedding approach}
\label{subsubsec:sv-ev-embed}

In this part, we use the embedding identities proved in the previous sections to transfer spectral and singular value
distribution results known for complex block Toeplitz sequences to quaternion Toeplitz
sequences. Throughout, write $\bn=(n_{1},\dots,n_{d})\in\NN^{d}$ and
$N_{\bn}=\prod_{\ell=1}^{d}n_{\ell}$.
For $\tau\in\{L,R,(S_{L},S_{R})\}$ we \val{define} the complex symbol
\[
G_{\tau}(\btheta):=
\begin{cases}
G_{\mathrm L}(\btheta), & \tau=L,\\[2pt]
G_{\mathrm R}(\btheta), & \tau=R,\\[2pt]
G_{S_{L},S_{R}}(\btheta), & \tau=(S_{L},S_{R}),
\end{cases}
\]
where $G_{\mathrm L},G_{\mathrm R},G_{S_{L},S_{R}}$ are those in
Theorems~\ref{thm:unified-LR-rect} and \ref{thm:sandwich+perm-rect-r}.

\begin{theorem}[Singular value distribution {and Hermitian spectral distribution}]\label{thm:SV-qToeplitz-1axis-embed-mv}
Let $F(\btheta)=Z(\btheta)+W(\btheta)\J$ with
$Z,W\in L^{1}(\TT^{d};\CC_{\I}^{\,s\times t})$. For each
$\tau\in\{L,R,(S_{L},S_{R})\}$,
\begin{equation}\label{eq:sv-limit-mv}
\bigl\{\,T_{\bn}^{(\tau)}(F)\,\bigr\}_{\bn}\ \sim_{\sigma}\ G_{\tau}.
\end{equation}
If, in addition, $s=t$ and $T_{\bn}^{(\tau)}(F)$ is Hermitian for all $\bn$,
then also
\[
\bigl\{\,T_{\bn}^{(\tau)}(F)\,\bigr\}_{\bn}\ \sim_{\lambda}\ G_{\tau}.
\]
\end{theorem}

\MovedProofMarker{p13}{0}{proof:p13}{app:toeplitz-proofs}
\val{The next result uses the labeling-free spectral essential range from Definition~\ref{def:complex-spectral-essential-range}.}

Under mild topological assumptions on the spectral essential range of the embedding
symbols, \val{collected in Theorem~\ref{thm:complex-tilli}}, we can go beyond the Hermitian case and describe the (canonical) eigenvalue
distribution of non-Hermitian Toeplitz sequences generated by bounded functions.

\begin{theorem}[Spectral distribution for Tilli-class symbols]\label{thm:EV-qToeplitz-1axis-embed-mv}
Assume $s=t$ and $F(\btheta)=Z(\btheta)+W(\btheta)\J$
with $Z,W\in L^{\infty}(\TT^{d};\CC_{\I}^{\,s\times s})$. For
$\tau\in\{L,R,(S_{L},S_{R})\}$, consider the spectral essential range $R(G_\tau)$.
Suppose that $G_\tau$ belongs to the Tilli class, i.e.,
\[
\qquad \mathrm{int}\,R(G_\tau)=\varnothing
\ \ \text{and}\ \ \CC\setminus R(G_\tau)\ \text{is connected}.
\]
Then
\begin{equation}\label{eq:ev-limit-mv}
\bigl\{\,T_{\bn}^{(\tau)}(F)\,\bigr\}_{\bn}\ \sim_{\lambda}\ G_{\tau}.
\end{equation}
\end{theorem}

\MovedProofMarker{p14}{0}{proof:p14}{app:toeplitz-proofs}

%-----------------------------------------------------------------------
\subsection{Schatten--norm estimates}
\label{sec:schatten-1axis-embed-mv}
%-----------------------------------------------------------------------

In this section we derive Schatten bounds for quaternionic multilevel Toeplitz
matrices via the complex symplectic embedding. Let
\(F(\btheta)=Z(\btheta)+W(\btheta)\J\) with
\(Z,W\in L^{p}(\TT^{d};\CC_{\I}^{\,s\times t})\) and \(1\le p\le\infty\).
We use the embedded complex symbols from
Theorems~\ref{thm:unified-LR-rect} and \ref{thm:sandwich+perm-rect-r}, denoted
\(G_{\tau}(\btheta)\) with \(\tau\in\{L,R,(S_{L},S_{R})\}\).
\val{By Corollary~\ref{cor:s_p}, the symplectic embedding satisfies}
\[
\|\PhiSymp(X)\|_{p}=2^{1/p}\,\|X\|_{p}
\quad\text{(with the convention \(2^{1/\infty}=1\)),}
\]
\val{and integrating this identity gives}
\begin{equation}\label{eq:Lp-phi-isometry}
\|\PhiSymp\!\circ F\|_{L^{p}(\TT^{d};\CC^{(2s)\times(2t)})}
=2^{1/p}\,\|F\|_{L^{p}(\TT^{d};\HH^{s\times t})}.
\end{equation}

%-----------------------------------------------------------------------
\val{Lemma~\ref{lem:Tilli-Sp-mv} in Appendix~\ref{app:complex-results} states the complex Toeplitz Schatten estimate used below.}

%-----------------------------------------------------------------------
\begin{theorem}[Quaternionic matrix-valued Toeplitz--Schatten bound]\label{thm:q-schatten-coarse-mv}
Let \(1\le p\le\infty\) and \(F\in L^{p}(\TT^{d};\HH^{s\times t})\).
For any \(\tau\in\{L,R,(S_{L},S_{R})\}\) and \(\bn\in\NN^{d}\), with
\(N_{\bn}:=\prod_{\ell=1}^{d} n_{\ell}\),
\begin{equation}\label{eq:Schatten-main-mv-coarse}
\bigl\|T^{(\tau)}_{\bn}(F)\bigr\|_{p}
\ \le\ {4}\left(\frac{N_{\bn}}{(2\pi)^{d}}\right)^{\!1/p}\,\|F\|_{L^{p}(\TT^{d};\HH^{s\times t})}.
\end{equation}
\end{theorem}
\val{The constant $4$ is not sharp. The corresponding complex bound in Lemma~\ref{lem:Tilli-Sp-mv} has constant $1$, and the additional factor arises from the embedding. We do not optimize this factor because the a.c.s.\ application does not require sharpness.}
\MovedProofMarker{p15}{0}{proof:p15}{app:toeplitz-proofs}

%====================================================================
\subsection{Spectral localization of Hermitian matrices via the $2s\times 2s$ complex symbols}
\label{subsec:embed_bounds-mv}
%====================================================================
Lin, Pan, and Ng \cite{lin_hermitian_2025} prove strong spectral localization for
unilevel scalar Hermitian right-quaternion Toeplitz matrices and provide positivity
criteria. We extend these results to the multilevel and block cases via the symplectic
embedding and the complex Toeplitz theory.

\begin{theorem}[Spectral localization of Hermitian Toeplitz matrices]\label{thm:embed_bounds-mv}
\val{Let $F\in L^{1}(\TT^{d};\HH^{s\times s})$, fix
$\tau\in\{L,R,(S_L,S_R)\}$, and let
$G_{\tau}\in L^{1}(\TT^d;\CC^{2s\times2s})$
be the corresponding complex symbol defined in
Theorems~\ref{thm:unified-LR-rect} and~\ref{thm:sandwich+perm-rect-r}.
Assume that $T_{\bn}^{(\tau)}(F)$ is Hermitian for every $\bn\in\NN^{d}$.}
Set
\[
m_{\tau}:=\operatorname*{ess\,inf}_{\boldsymbol\theta\in\TT^{d}}
          \lambda_{\min}\!\bigl(G_{\tau}(\boldsymbol\theta)\bigr),\qquad
M_{\tau}:=\operatorname*{ess\,sup}_{\boldsymbol\theta\in\TT^{d}}
          \lambda_{\max}\!\bigl(G_{\tau}(\boldsymbol\theta)\bigr).
\]
Then, for every $\bn\in\NN^{d}$,
\begin{equation}\label{eq:hermitian_loc}
\sigma\!\bigl(T^{(\tau)}_{\bn}(F)\bigr)\ \subset\ [\,m_{\tau},\,M_{\tau}\,].
\end{equation}
\end{theorem}

\MovedProofMarker{p16}{0}{proof:p16}{app:toeplitz-proofs}

\val{The multilevel block positivity mechanism used here is the
matrix-valued linear-positive-operator property in
\cite[Sec.~3.1, Prop.~3.1]{serra_noncommutative_2002}.
For related unilevel block and preconditioned block Toeplitz spectral
analyses, see \cite{serra_spectral_1999,SIMAX99}.}

\val{For Hermitian matrices $A$ and $B$ of the same size, $A\succeq B$ means that
$A-B$ is positive semidefinite.}

\begin{corollary}[Positive definiteness criterion]\label{cor:HPD}
\val{Let $F\in L^{1}(\TT^{d};\HH^{s\times s})$, fix
$\tau\in\{L,R,(S_L,S_R)\}$, and let
$G_{\tau}\in L^{1}(\TT^d;\CC^{2s\times2s})$
be the corresponding complex symbol defined in
Theorems~\ref{thm:unified-LR-rect} and~\ref{thm:sandwich+perm-rect-r}.
Assume that $T_{\bn}^{(\tau)}(F)$ is Hermitian for every $\bn\in\NN^d$, and set
\[
m_{\tau}:=\operatorname*{ess\,inf}_{\boldsymbol\theta\in\TT^{d}}
\lambda_{\min}\!\bigl(G_{\tau}(\boldsymbol\theta)\bigr).
\]
Then:}
\begin{enumerate}[label=(\alph*),leftmargin=*]
\item If $m_{\tau}>0$, then $T_{\bn}^{(\tau)}(F)$ is Hermitian positive definite
      for every $\bn$, with $\lambda_{\min}\!\bigl(T_{\bn}^{(\tau)}(F)\bigr)\ge m_{\tau}$.
\item Conversely, if there exists $c>0$ such that
      $T_{\bn}^{(\tau)}(F)\succeq c\,I$ for all $\bn$, then $m_{\tau}\ge c$.
\end{enumerate}
\end{corollary}

\MovedProofMarker{p17}{0}{proof:p17}{app:toeplitz-proofs}

\subsection{Toeplitz a.c.s. approximation via circulant matrix sequences}
\label{sec:acs-circulant}

\val{In the present section}, we recover the classical circulant approximation property for Toeplitz sequences in the quaternion setting and provide an alternative derivation of their singular value distribution using a.c.s. theory. In fact, the spectral analysis of Toeplitz sequences through circulant approximation is the ``textbook'' approach to this problem for two reasons: this approach is generally considered elegant and highlights in a clear way the topological structure of the space we are studying; moreover, it has a direct link with the classical preconditioning theory, where the observation that circulant matrices approximate Toeplitz ones created a huge ecosystem of fast Toeplitz preconditioning strategies (see e.g. \cite{Ng-Book} and references there reported). Numerically speaking, we expect the same link also in quaternion problems (see, for example, the fast Strang preconditioner analyzed in \cite{lin_hermitian_2025}). Barbarino and Garoni make a remark in \cite{barbarino_glt_2025}, where the easy and ``rigid'' spectral properties of circulant matrices allow for a simpler analysis of the asymptotic spectral distribution of normal sequences of the GLT class, for which Toeplitz sequences are a subclass. \\
Given these premises, first, we follow the lines of \cite{PanCirculant2024} and decompose multilevel $s\times t$ circulant matrices under the $\I$-axis quaternion discrete Fourier Transform, and we extract the sparsity $X$-pattern composed of blocks of size $s\times t$ or $2s\times 2t$. Up to block permutations, this pattern provides a block-diagonalization. Then, we describe the spectrum of multilevel block quaternion circulant matrices, we show that every Toeplitz matrix of any family admits a proper circulant approximating class of sequences, and we provide an alternative approach to find the singular value or spectral distribution of Toeplitz sequences. \\
We recall the basic notation and tools for the Pan-Ng factorization. For $n\in\NN$, let $A_n\in\{0,1\}^{n\times n}$ be the permutation matrix that maps the canonical basis in $\HH^n$ according to the law $e_s\mapsto e_{(-s)\bmod n}$ and set $A_{\bm n}:=A_{n_1}\otimes\cdots\otimes A_{n_d}$. The QDFT on the one-dimensional $\I$-axis is defined as in the complex case (see \cite{Ell_2014}). More precisely we have
\[
[F_{\I,n}]_{uv}:=n^{-1/2}\exp(-2\pi\I\,uv/n),\qquad u,v=0,\dots,n-1,
\]
with the multilevel transform being $F_{\I}^{(\bm n)}:=F_{\I,n_1}\otimes\cdots\otimes F_{\I,n_d}\in\CC_\I^{N_{\bm n}\times N_{\bm n}}$.
Given the block sizes $s,t\in\NN$ and since, for the time being, we define the DFT in the slice $\CC_\I$, we use the following compact notation
\[
U_L:=F_{\I}^{(\bm n)}\otimes I_s,\qquad U_R:=F_{\I}^{(\bm n)}\otimes I_t.
\]
In this section, most of the time and without loss of generality, we work with right polynomials and right generating functions. Given a right-polynomial, in the following we introduce $d$-level block circulants:

\subsubsection*{Multilevel block quaternion circulants generated by trigonometric polynomials}
We start by recalling the canonical construction of multilevel block circulant matrices in the quaternion setting.
First, given $\bm n=(n_{1},\ldots,n_{d})\in\NN^{d}$ and the multi-index set $\Lambda_{\bm n}$, the elementary periodic multilevel shifts are defined by
\[
P_{\bm n}^{(\brho)}:=P_{n_{1}}^{(\rho_{1})}\otimes\cdots\otimes P_{n_{d}}^{(\rho_{d})},
\qquad [P_{n_k}^{(\rho_k)}]_{ij}:=\mathbf{1}_{\{i\equiv j+\rho_k\ (\mathrm{mod}\ n_k)\}}.
\]
We now define multilevel block circulant matrices generated by matrix-valued right-polynomials. Most of the time, we use the definition based on periodic shifts, due to its algebraic convenience in our derivations, but we also recall the entrywise definition.

\begin{definition}[$d$-level $s\times t$ quaternion block-circulant generated by a right polynomial]
\label{def:bcirc-poly}
Given a right-polynomial $p_{\mathrm{R}}$ with block coefficients $\{p_{\mathrm R}(\brho)\}_{\brho\in \mathcal F}\subset\HH^{s\times t}$, we define the associated
$d$-level $s \times t$ block-circulant matrix as
\[
C_{\bm n}\bigl(p_{\mathrm{R}}\bigr)
:=\sum_{\brho\in \mathcal F}\bigl(P_{\bm n}^{(\brho)}\otimes p_{\mathrm R}(\brho)\bigr)\ \in\ \HH^{(N_{\bm n}s)\times(N_{\bm n}t)},
\qquad N_{\bm n}=\prod_{\ell=1}^{d} n_{\ell}.
\]
Equivalently, we can describe the entries explicitly
\[
\bigl[C_{\bm n}\bigl(p_{\mathrm{R}}\bigr)\bigr]_{\alpha,\beta}
=\sum_{\brho\in \mathcal F} p_{\mathrm R}(\brho)\,\mathbf{1}_{\{\alpha\equiv \beta+\brho\ (\mathrm{mod}\ \bm n)\}},
\]
\end{definition}
To find a suitable block-diagonalization of multilevel block-circulant matrices, we follow the Pan-Ng splitting approach verbatim, in particular the results found in \cite{PanCirculant2024}. The approach in the multilevel and block setting is completely analogous. Further, this approach is purely quaternion algebraic and is suitable for our a.c.s.\ approximation formalism. We note that a full-block diagonalization of unilevel-block circulants and a Fourier-based interpretation of diagonal blocks is also found by Zheng and Ni in \cite{zheng_block_2024} by moving the problem from quaternion to non-associative octonions.   \\
First, write each polynomial coefficient in Cartesian form
\(p_{\mathrm R}(\brho)=p_{\mathrm R}(\brho)^{(0)}+p_{\mathrm R}(\brho)^{(1)}\I+p_{\mathrm R}(\brho)^{(2)}\J+p_{\mathrm R}(\brho)^{(3)}(\I\J)\)
with \(p_{\mathrm R}(\brho)^{(\ell)}\in\RR^{s\times t}\). By expanding the shift representation, we obtain
\[
\begin{aligned}
C_{\bm n}\bigl(p_{\mathrm{R}}\bigr)
&=\sum_{\brho\in \mathcal F}\bigl(P_{\bm n}^{(\brho)}\otimes p_{\mathrm R}(\brho)^{(0)}\bigr)
\;+\;\Bigl(\sum_{\brho\in \mathcal F}P_{\bm n}^{(\brho)}\otimes p_{\mathrm R}(\brho)^{(1)}\Bigr)\,\I\\
&\quad+\;\Bigl(\sum_{\brho\in \mathcal F}P_{\bm n}^{(\brho)}\otimes p_{\mathrm R}(\brho)^{(2)}\Bigr)\,\J
\;+\;\Bigl(\sum_{\brho\in \mathcal F}P_{\bm n}^{(\brho)}\otimes p_{\mathrm R}(\brho)^{(3)}\Bigr)\,(\I\J).
\end{aligned}
\]
In particular, we observe that \(C_{\bm n}\bigl(p_{\mathrm{R}}\bigr)\) is a right-linear combination of four
real multilevel block-circulants. The idea behind our target diagonalization is the observation that applying linearly the QDFT to the circulant components in the complex slice \(\CC_{\I}\) and to the orthogonal circulant components produces different sparsity patterns, and when we sum all of these the sparsity results in a multilevel version of the ``X shape'' decomposition found in \cite{PanCirculant2024}. \\

Next, we need the following elementary algebraic identities of QDFTs.

\begin{lemma}[QDFT identities, 1D and multilevel]\label{lem:qfft-id-detailed}
\val{We consider the fixed imaginary unit \(\I\in\operatorname{Im}\HH\) and identify \(\CC_{\I}\cong\CC\)}.
For \(n\in\NN\) let \(F_{\I,n}\in\CC_{\I}^{n\times n}\) be
\([F_{\I,n}]_{uv}=n^{-1/2}\exp(-2\pi \I\,uv/n)\), \(u,v=0,\dots,n-1\).
Then
\[
F_{\I,n}^{*}F_{\I,n}=I_{n},\qquad F_{\I,n}^{T}=F_{\I,n},\qquad (F_{\I,n})^{2}=A_{n},
\]
where \(A_{n}\) is the reversal permutation (first/last, second/second-last, \dots).
For \(\bn=(n_{1},\dots,n_{d})\) and \(F_{\I}^{(\bn)}:=\bigotimes_{\ell=1}^{d}F_{\I,n_{\ell}}\),
\[
\bigl(F_{\I}^{(\bn)}\bigr)^{*}F_{\I}^{(\bn)}=I_{N},\quad
\bigl(F_{\I}^{(\bn)}\bigr)^{T}=F_{\I}^{(\bn)},\quad
\bigl(F_{\I}^{(\bn)}\bigr)^{2}=A_{\bn}:=\bigotimes_{\ell=1}^{d}A_{n_{\ell}},
\]
with \(N=\prod_{\ell=1}^{d}n_{\ell}\).
\end{lemma}

As a consequence, we also obtain QDFT flips for multilevel block transforms.

\begin{lemma}[\val{QDFT flip identities}]\label{lem:flip-detailed}
\val{Let $N_{\bn}:=\prod_{\ell=1}^{d}n_\ell$. Then}
\begin{equation}\label{eq:flip-detailed}
\val{F_{\I}^{(\bn)}\,(\J I_{N_{\bn}})\,\bigl(F_{\I}^{(\bn)}\bigr)^{*}
=A_{\bn}\,(\J I_{N_{\bn}}),}
\end{equation}
\val{and}
\begin{equation}\label{eq:flip-detailed-ij}
\val{F_{\I}^{(\bn)}\,((\I\J)I_{N_{\bn}})\,\bigl(F_{\I}^{(\bn)}\bigr)^{*}
=A_{\bn}\,((\I\J)I_{N_{\bn}}).}
\end{equation}
\end{lemma}

\val{Proposition~\ref{prop:complex-bcirc-diagonalization} in Appendix~\ref{app:complex-results} states the complex/real DFT diagonalization used below.}
To obtain a block-diagonal sparsity pattern, we need to use permutations. The following ordering description is the multilevel generalization of the one used by \cite{PanCirculant2024}:

\begin{definition}\label{def:Fix}
\val{Let $d\in\NN$, let $\bm n=(n_{1},\ldots,n_{d})\in\NN^{d}$, and let
$\Lambda_{\bm n}$ be the index set defined in \eqref{eq:Lambda}. If
$\bk=(k_{1},\ldots,k_{d})\in\Lambda_{\bm n}$, set
$-\bk:=((-k_{1})\bmod n_{1},\ldots,(-k_{d})\bmod n_{d})$.}
Define the fixed-point set
\[
\mathrm{Fix}:=\{\bk\in\Lambda_{\bm n}:\ \bk\equiv -\bk\ (\mathrm{mod}\ \bm n)\}
=\bigl\{\bk\in\Lambda_{\bm n}:\ 2k_{\ell}\equiv 0\ (\mathrm{mod}\ n_{\ell})\ \forall\,\ell\bigr\}.
\]
\end{definition}

\begin{definition}\label{def:pair}
We denote the standard lexicographic order on $\Lambda_{\bm n}$ by $\prec$, and set
\[
\mathcal K:=\{\bk\in\Lambda_{\bm n}\setminus\mathrm{Fix}:\ \bk\prec -\bk\}.
\]
\val{Set $f:=\#\mathrm{Fix}$ and $q:=\#\mathcal K$, and enumerate the two
sets in lexicographic order as}
\[
\val{\mathrm{Fix}=\{\bk^{(1)},\ldots,\bk^{(f)}\},
\qquad
\mathcal K=\{\bk_1,\ldots,\bk_q\}.}
\]
\val{Define the complete ordered list}
\[
\val{\mathcal L
=\bigl(\bk^{(1)},\ldots,\bk^{(f)},
\bk_1,-\bk_1,\ldots,\bk_q,-\bk_q\bigr)
=:(\boldsymbol\ell_1,\ldots,\boldsymbol\ell_{N_{\bn}}).}
\]
\val{If $e_{\bk}$ denotes the canonical basis vector indexed by
$\bk\in\Lambda_{\bn}$ in the original lexicographic ordering, let $P$ be the
unique permutation matrix satisfying}
\[
\val{P e_{\boldsymbol\ell_j}=e_j,
\qquad j=1,\ldots,N_{\bn}.}
\]
\val{Thus $P$ maps the original lexicographic ordering to $\mathcal L$.
For the multilevel $s\times t$ block case, define
$\Pi_L:=P\otimes I_s$ and $\Pi_R:=P\otimes I_t$.}
\end{definition}
By construction, we have the following identities:
\begin{lemma}\label{lem:PAP}
Let $A_{\bm n}=\bigotimes_{\ell=1}^{d}A_{n_{\ell}}$ be the multilevel exchange permutation as in Lemma~\ref{lem:qfft-id-detailed}.
With $P$ from Definition~\ref{def:pair} and $\mathcal K$ as above,
\[
P\,A_{\bm n}\,P^{*}
= \val{I_{\#\mathrm{Fix}}}\ \oplus\ \bigoplus_{\bk\in\mathcal K}\begin{bmatrix}0&1\\[2pt]1&0\end{bmatrix}.
\]
Consequently,
\[
\Pi_{R}\,(A_{\bm n}\otimes I_{t})\,\Pi_{R}^{*}
= \val{I_{\#\mathrm{Fix}}}\otimes I_{t}\ \oplus\ \bigoplus_{\bk\in\mathcal K}
\begin{bmatrix}0&I_{t}\\[2pt]I_{t}&0\end{bmatrix}.
\]
\end{lemma}

Given these premises, we can derive the canonical QDFT decomposition for multilevel block-circulant matrices:

\begin{proposition}[Canonical QDFT, multilevel rectangular block case]\label{prop:canon-X}
Let $U_{L}:=F_{\I}^{(\bm n)}\otimes I_{s}$ and $U_{R}:=F_{\I}^{(\bm n)}\otimes I_{t}$.
Define, for $\ell=0,1,2,3$,
\[
\Lambda_{\ell}
:=\Bdiag\bigl(\widehat S_{\ell}(\bk)\bigr)_{\bk\in\Lambda_{\bm n}},
\]
where
\[
\widehat S_{\ell}(\bk)=\sum_{\brho\in \mathcal F}p_{\mathrm R}(\brho)^{(\ell)}\,
e^{-2\pi \I\sum_{\ell'=1}^{d} k_{\ell'}\rho_{\ell'}/n_{\ell'}}.
\]
Then
\begin{equation}\label{eq:canonical-form}
U_{L}\,C_{\bm n}\bigl(p_{\mathrm{R}}\bigr)\,U_{R}^{*}
=\Lambda_{0}+\Lambda_{1}\,\I+\Lambda_{2}\,(A_{\bm n}\otimes I_{t})\,\J
+\Lambda_{3}\,(A_{\bm n}\otimes I_{t})\,(\I\J).
\end{equation}
\val{Let $P$ be the permutation defined in Definition~\ref{def:pair}, and set
$\Pi_L:=P\otimes I_s$ and $\Pi_R:=P\otimes I_t$. Applying $\Pi_L$ to the
block rows and $\Pi_R^*$ to the block columns gives}
\[
\Pi_{L}U_{L}\,C_{\bm n}\bigl(p_{\mathrm{R}}\bigr)\,U_{R}^{*}\Pi_{R}^{*}
=\bigoplus_{\bk\in\mathrm{Fix}} \bigl( D_{1}(\bk) {+D_2(\bk)} \bigr)\ \ \oplus\
\val{\bigoplus_{\bk\in\mathcal K}}\begin{bmatrix} D_{1}(\bk) & D_{2}(\bk)\\[2pt] D_{2}(-\bk) & D_{1}(-\bk)\end{bmatrix},
\]
where
\[
D_{1}(\bk):=\widehat S_{0}(\bk)+\widehat S_{1}(\bk)\,\I,\qquad
D_{2}(\bk):=\bigl(\widehat S_{2}(\bk)+\widehat S_{3}(\bk)\,\I\bigr)\,\J.
\]
\end{proposition}

\MovedProofMarker{p18}{1}{proof:p18}{app:toeplitz-proofs}

\subsubsection{The spectrum of multilevel block-circulant matrices and the a.c.s.\ circulant approximation of Toeplitz sequences}
\label{subsubsec:embed-fibers-poly}

\val{In the current section}, we discuss the details of the canonical spectrum and singular values of multilevel block circulant matrices. Then we use these details to find the asymptotic spectral/singular value distribution of the corresponding matrix sequence and prove that every quaternion Toeplitz sequence has a multilevel right block-circulant a.c.s. \\
Let $p_{\mathrm{R}}(\btheta)=\sum_{\brho\in \mathcal F}p_{\mathrm R}(\brho)\,e^{-\I\langle \brho,\btheta\rangle}$
be a $s\times t$ trigonometric polynomial with finite $\mathcal F\subset\ZZ^{d}$\val{. We then construct} the multilevel block-circulant $C_{\bm n}\bigl(p_{\mathrm{R}}\bigr)$
generated by $\{p_{\mathrm R}(\brho)\}_{\brho\in \mathcal F}$ as in Definition~\ref{def:bcirc-poly}
and consider $U_{R}$, $U_{L}$, and $A_{\bm n}$ used in the preceding proposition (Proposition \ref{prop:canon-X}). We need to highlight the fiber property of the block diagonal we obtain with the factorization in Proposition \ref{prop:canon-X}, meaning that their symplectic embeddings are uniform samples of the spectral and singular value symbol of the corresponding \val{circulant} sequence $\{C_{\bm n}\bigl(p_{\mathrm{R}}\bigr)\}_{\bm n}$. \\
Consider the uniform sample set $\Theta$ in $\val{\TT^d}$ with elements
$\btheta_{\bk}= 2\pi\bigl(\bk/\bn)$ for $\bk\in\Lambda_{\bm n}$ and note that, by periodicity $\ -\btheta_{\bk}=\btheta_{-\bk}\ (\mathrm{mod}\ 2\pi)$. \\

\val{Write}
\[
\val{p_{\mathrm R}(\brho)=z_{\brho}+w_{\brho}\J,}
\]
\val{where}
\[
\val{z_{\brho}:=p_{\mathrm R}(\brho)^{(0)}+p_{\mathrm R}(\brho)^{(1)}\I,
\qquad
w_{\brho}:=p_{\mathrm R}(\brho)^{(2)}+p_{\mathrm R}(\brho)^{(3)}\I,}
\]
\val{and $p_{\mathrm R}(\brho)^{(\ell)}\in\RR^{s\times t}$. Define}
\begin{equation}\label{eq:split}
\val{Z(\btheta):=\sum_{\brho\in \mathcal F}z_{\brho}e^{-\I\langle\brho,\btheta\rangle},
\qquad
W(\btheta):=\sum_{\brho\in \mathcal F}w_{\brho}e^{+\I\langle\brho,\btheta\rangle}.}
\end{equation}
\val{The identity $\J e^{-\I t}=e^{+\I t}\J$ implies}
\[
\val{p_{\mathrm R}(\btheta)=Z(\btheta)+W(\btheta)\J,
\qquad
W(-\btheta)=\sum_{\brho\in\mathcal F}w_{\brho}e^{-\I\langle\brho,\btheta\rangle}.}
\]
\val{Thus $W(-\btheta)$ has the same negative Fourier phase as $Z(\btheta)$,
which is the form occurring in the embedded right symbol. For any right symbol
$H=Z_H+W_H\J$, we write}
\begin{equation}\label{eq:GR-bracket}
\val{G_{\mathrm R}[H](\btheta):=
\begin{bmatrix}
Z_H(\btheta)&W_H(-\btheta)\\[2pt]
-\overline{W_H(\btheta)}&\overline{Z_H(-\btheta)}
\end{bmatrix}.}
\end{equation}
For our aims, we need the symplectic embedding in block form $\PhiSymp_1:\HH^{s\times t}\to\CC^{(2s)\times(2t)}$. \\
To obtain the samplings of canonical eigenvalues and singular values of circulant matrices, we make use of the following claim, which is a consequence of Proposition \ref{prop:canon-X}.

\begin{proposition}\label{prop:qdftr-fibers-poly}
Let $p_{\mathrm R}(\btheta)=\sum_{\brho\in \mathcal F}p_{\mathrm R}(\brho)\,e^{-\I\langle \brho,\btheta\rangle}$ be an $s\times t$ trigonometric polynomial with finite $\mathcal F\subset\ZZ^{d}$ and coefficients $p_{\mathrm R}(\brho)\in\HH^{s\times t}$\val{, and let $Z$ and $W$ be defined by \eqref{eq:split}}. Let $C_{\bm n}(p_{\mathrm R})$ be the associated $d$-level $s\times t$ block-circulant. \\
Then, from the QDFT canonical form of $C_{\bm n}(p_{\mathrm R})$ written in the notation
\begin{equation}\label{eq:X-fiber}
\Pi_{L}\,U_{L}\,C_{\bm n}(p_{\mathrm R})\,U_{R}^{*}\,\Pi_{R}^{*}
=\bigoplus_{\bk\in\mathrm{Fix}}\bigl(Z(\btheta_{\bk})+W(-\btheta_{\bk})\J\bigr)\ \ \oplus\
\val{\bigoplus_{\bk\in\mathcal K}}
\begin{bmatrix}
Z(\btheta_{\bk}) & W(-\btheta_{\bk})\J\\[2pt]
W(-\btheta_{-\bk})\J & Z(\btheta_{-\bk})
\end{bmatrix},
\end{equation}
we obtain the embedding of the block diagonals:
\begin{equation}\label{eq:Phi-paired}
\PhiSymp\!\left(
\begin{bmatrix}
Z(\btheta_{\bk}) & W(-\btheta_{\bk})\J\\[2pt]
W(-\btheta_{-\bk})\J & Z(\btheta_{-\bk})
\end{bmatrix}
\right)
=
\begin{bmatrix}
Z(\btheta_{\bk}) & 0 & 0 & W(-\btheta_{\bk})\\[2pt]
0 & \overline{Z(\btheta_{\bk})} & -\overline{W(-\btheta_{\bk})} & 0\\[2pt]
0 & W(-\btheta_{-\bk}) & Z(\btheta_{-\bk}) & 0\\[2pt]
-\overline{W(-\btheta_{-\bk})} & 0 & 0 & \overline{Z(\btheta_{-\bk})}
\end{bmatrix},
\end{equation}
when \val{$\bk\in\mathcal K$}, while for $\bk\in\mathrm{Fix}$ we have
\begin{equation}\label{eq:Phi-fixed}
\PhiSymp\bigl(Z(\btheta_{\bk})+W(-\btheta_{\bk})\J\bigr)
=\begin{bmatrix}
Z(\btheta_{\bk}) & W(-\btheta_{\bk})\\[2pt]
-\overline{W(-\btheta_{\bk})} & \overline{Z(\btheta_{\bk})}
\end{bmatrix}.
\end{equation}
\end{proposition}

\MovedProofMarker{p19}{0}{proof:p19}{app:toeplitz-proofs}
As a corollary, we deduce the asymptotic singular value and canonical spectral distribution of multilevel $s \times t$ circulant matrix sequences as reported in the subsequent result.

\begin{corollary}[Asymptotic spectral/singular value distributions for right-polynomial quaternion circulants]\label{cor:asymp-circ}
\val{Let}
\[
p_{\mathrm R}(\btheta)=\sum_{\brho\in \mathcal F}p_{\mathrm R}(\brho)\,e^{-\I\langle \brho,\btheta\rangle},
\qquad
\mathcal F\subset\ZZ^{d}\ \text{finite},\ \ p_{\mathrm R}(\brho)\in\HH^{s\times t}.
\]
\val{Let $Z$ and $W$ be defined by \eqref{eq:split}, and let
$G_{\mathrm R}[p_{\mathrm R}]$ be the associated complex symbol from
\eqref{eq:GR-bracket}.}
Let $C_{\bm n}(p_{\mathrm R})$ be the corresponding $d$-level $s\times t$ block-circulant. Then
\[
\{C_{\bm n}(p_{\mathrm R})\}_{\bm n}\ \sim_{\sigma}\ \val{G_{\mathrm R}[p_{\mathrm R}]}.
\]
If $s=t$, then also
\[
\{C_{\bm n}(p_{\mathrm R})\}_{\bm n}\ \sim_{\lambda}\ \val{G_{\mathrm R}[p_{\mathrm R}]}.
\]
\end{corollary}

\MovedProofMarker{p20}{0}{proof:p20}{app:toeplitz-proofs}

\begin{theorem}[Right Toeplitz admits an a.c.s.\ of circulants]\label{thm:right-acs-circ}
Let $d\in\NN$, $s,t\in\NN$, and
$F\in L^{1}(\TT^{d};\HH^{s\times t})$. Consider the right Toeplitz family
\[
\{T_{\bm n}^{(R)}(F)\}_{\bm n\in\NN^d}.
\]
\val{Then there exist finite sets $\mathcal S_m\subset\ZZ^{d}$ and coefficients
$A_{m,\bk}\in\HH^{s\times t}$ defining right circulants
\[
B_{\bm n,m}
:=
\sum_{\bk\in\mathcal S_m}
P_{\bm n}^{(-\bk)}\otimes A_{m,\bk}
\]
such that, for every scalar-indexed sequence
$\bm n=\bm n(n)\to\infty$,
\[
\bigl\{\{B_{\bm n(n),m}\}_n\bigr\}_{m\in\NN}
\]
is an a.c.s.\ for
$\{T_{\bm n(n)}^{(R)}(F)\}_n$. The external approximation index
$m\in\NN$ remains scalar.}
\end{theorem}

\MovedProofMarker{p21}{0}{proof:p21}{app:toeplitz-proofs}
\val{The following argument is an alternative a.c.s.-based proof of the singular-value
assertion in Theorem~\ref{thm:SV-qToeplitz-1axis-embed-mv}. The Hermitian
eigenvalue assertion follows from the embedding proof in
Appendix~\ref{app:toeplitz-proofs}.}
\MovedProofMarker{p22}{0}{proof:p22}{app:toeplitz-proofs}

\section{Numerical Results}\label{sec:apps}

In this section, we present numerical {tests focused on showing the results of Theorem~\ref{thm:SV-qToeplitz-1axis-embed-mv} regarding the singular value distribution of the Toeplitz quaternion sequences and the spectral distribution in the Hermitian setting empirically. The experimental setting is standard in asymptotic spectral analysis: indeed we compare the non-decreasing sorted eigenvalues of the Toeplitz matrix against a fine uniform sampling of a monotone non-decreasing rearrangement of the symbol (see \cite{garoni_multilevel_2018} for the definition of monotone rearrangement we use). Owing to Theorem~\ref{thm:SV-qToeplitz-1axis-embed-mv}, we expect an evident match between the two plots as the size of the matrix is large. We further remark that, as observed in \cite{Bogoya2016} and \cite{barbarino_uniform_2025}, this experimental setup coincides with verifying numerically the $L^1$ convergence of quantile functions.\\
More specifically, our experiments focus on showing quantile convergence for \val{two-level} block quaternionic Toeplitz matrices generated by square $2\times2$ \val{two-level} block generating functions. This choice makes the test general, \val{computationally tractable}, and spectrally informative, if applied, for example, to dual-channel 2D systems in signal analysis problems.
\val{The agreement is already good for moderate matrix sizes, although the
statements in Theorem~\ref{thm:SV-qToeplitz-1axis-embed-mv} are asymptotic.}
}

% [Revision, Referee 1 point 60] The "Computational Setup" block (hardware/OS/package list)
% has been removed: it is irrelevant to these deterministic spectral-distribution plots (no
% timings are reported). The generating functions are given explicitly below and the figures
% are reproducible from the accompanying code.

\subsection{\val{Analytical construction of the embedded symbols}}
\label{subsec:worked-symbols}

\val{We first compute the complex slice components of the four generating
functions used below. Let $\btheta=(\theta_1,\theta_2)$ and write}
\[
\val{F(\btheta)=Z(\btheta)+W(\btheta)\J,
\qquad
Z(\btheta),W(\btheta)\in\CC_{\I}^{2\times2}.}
\]
\val{For the four kernel classes used in the experiments, the general embedded
symbol formula gives}
\[
\val{G_L(\btheta)=
\begin{bmatrix}
Z(\theta_1,\theta_2)&W(\theta_1,\theta_2)\\[2pt]
-\overline{W(-\theta_1,-\theta_2)}&\overline{Z(-\theta_1,-\theta_2)}
\end{bmatrix},}
\]
\[
\val{G_R(\btheta)=
\begin{bmatrix}
Z(\theta_1,\theta_2)&W(-\theta_1,-\theta_2)\\[2pt]
-\overline{W(\theta_1,\theta_2)}&\overline{Z(-\theta_1,-\theta_2)}
\end{bmatrix},}
\]
\[
\val{G_{S_{12}}(\btheta)=
\begin{bmatrix}
Z(\theta_1,\theta_2)&W(\theta_1,-\theta_2)\\[2pt]
-\overline{W(-\theta_1,\theta_2)}&\overline{Z(-\theta_1,-\theta_2)}
\end{bmatrix},}
\]
\[
\val{G_{S_{21}}(\btheta)=
\begin{bmatrix}
Z(\theta_1,\theta_2)&W(-\theta_1,\theta_2)\\[2pt]
-\overline{W(\theta_1,-\theta_2)}&\overline{Z(-\theta_1,-\theta_2)}
\end{bmatrix}.}
\]
\val{Indeed, the two sandwich cases follow from}
\[
\val{r_{\{1\}}(\theta_1,\theta_2)=(-\theta_1,\theta_2),
\qquad
r_{\{2\}}(\theta_1,\theta_2)=(\theta_1,-\theta_2).}
\]

\paragraph{\val{Continuous Hermitian symbol.}}
\val{Set}
\[
\val{a(\btheta):=2+\cos\theta_1+\cos\theta_2.}
\]
\val{The slice decomposition of the continuous Hermitian generating function is}
\[
\val{Z_{\mathrm{cH}}(\btheta)=
\begin{bmatrix}
a(\btheta)&\frac12\sin\theta_1\,\I\\[3pt]
-\frac12\sin\theta_1\,\I&a(\btheta)
\end{bmatrix}.}
\]
\[
\val{
W_{\mathrm{cH}}(\btheta)=
\begin{bmatrix}
0&\frac12\sin\theta_2+\frac14\sin(\theta_1-\theta_2)\I\\[3pt]
\frac12\sin\theta_2+\frac14\sin(\theta_1-\theta_2)\I&0
\end{bmatrix}.}
\]
\val{Substitution of $Z_{\mathrm{cH}}$ and $W_{\mathrm{cH}}$ into the four
formulas above gives the embedded symbols for all four kernel classes.}

\paragraph{\val{Continuous non-Hermitian symbol.}}
\val{The two quaternion exponentials that contribute to the second slice satisfy}
\[
\val{e^{\J\theta_2}=\cos\theta_2+\sin\theta_2\,\J,
\qquad
e^{\K\theta_2}=\cos\theta_2+(\I\sin\theta_2)\J.}
\]
\val{Consequently,}
\[
\val{Z_{\mathrm{cN}}(\btheta)=
\begin{bmatrix}
e^{\I\theta_1}+\cos\theta_2&0\\[3pt]
(\cos\theta_1-\sin\theta_2)\I&1+\frac12(\cos\theta_1-\cos\theta_2)
\end{bmatrix},}
\]
\[
\val{W_{\mathrm{cN}}(\btheta)=
\begin{bmatrix}
\sin\theta_2&(\cos\theta_2+\sin\theta_1)\I\\[3pt]
0&\frac12(\sin\theta_1-\I\sin\theta_2)
\end{bmatrix}.}
\]

\paragraph{\val{$L^1$ Hermitian symbol.}}
\val{Let $s_j:=\operatorname{sgn}(\theta_j)$ for $j=1,2$. Since
$\J+\K=(1+\I)\J$, we obtain}
\[
\val{Z_{1\mathrm H}(\btheta)=
\begin{bmatrix}
2&s_1\I\\
-s_1\I&2
\end{bmatrix},
\qquad
W_{1\mathrm H}(\btheta)=
\begin{bmatrix}
0&s_2(1+\I)\\
s_2(1+\I)&0
\end{bmatrix}.}
\]

\paragraph{\val{$L^1$ non-Hermitian symbol.}}
\val{Define}
\[
\val{s_j:=\operatorname{sgn}(\theta_j),
\qquad
h(\btheta):=\operatorname{sgn}\!\left(\sin\theta_1+\frac14\cos\theta_2\right),
\qquad
g(\btheta):=\operatorname{sgn}\!\left(\cos\theta_1-\sin\theta_2\right),}
\]
\[
\val{c_j:=\cos\!\left(\frac{\pi s_j}{2}\right),
\qquad
u_j:=\sin\!\left(\frac{\pi s_j}{2}\right),
\qquad j=1,2.}
\]
\val{The corresponding slice functions are}
\[
\val{Z_{1\mathrm N}(\btheta)=
\begin{bmatrix}
e^{\I\theta_1}+s_1&0\\
g(\btheta)\I&1+\frac14(c_1-c_2)
\end{bmatrix},
\qquad
W_{1\mathrm N}(\btheta)=
\begin{bmatrix}
s_2&h(\btheta)\I\\
0&\frac14(u_1-\I u_2)
\end{bmatrix}.}
\]
\val{Away from the measure-zero sets defined by the equations $\theta_j=0$, $j=1,2$, one has
$c_1=c_2=0$ and $u_j=s_j$, $j=1,2$. Thus the lower-right entries simplify almost
everywhere to the quantities}
\[
\val{(Z_{1\mathrm N})_{22}=1,
\qquad
(W_{1\mathrm N})_{22}=\frac14(s_1-\I s_2).}
\]

\subsection{Continuous symbols}\label{subsec:quant-cont}
Tests use matrix sizes $(n,n)$ with $n=2,8,32,64$ for both Hermitian and non-Hermitian cases.
\paragraph{Continuous Hermitian.}
We consider the following smooth symmetric {function}:
\[
F(\theta_1,\theta_2)=
\begin{bmatrix}
2+\cos\theta_1+\cos\theta_2 &
\tfrac12\bigl(\sin\theta_1\,\I+\sin\theta_2\,\J\bigr)
+\tfrac14\sin(\theta_1-\theta_2)\,\K\\[4pt]
\tfrac12\bigl(\sin\theta_2\,\J-\sin\theta_1\,\I\bigr)
+\tfrac14\sin(\theta_1-\theta_2)\,\K &
2+\cos\theta_1+\cos\theta_2
\end{bmatrix}.
\]
\begin{figure}[H]
  \centering
  \includegraphics[width=\linewidth]{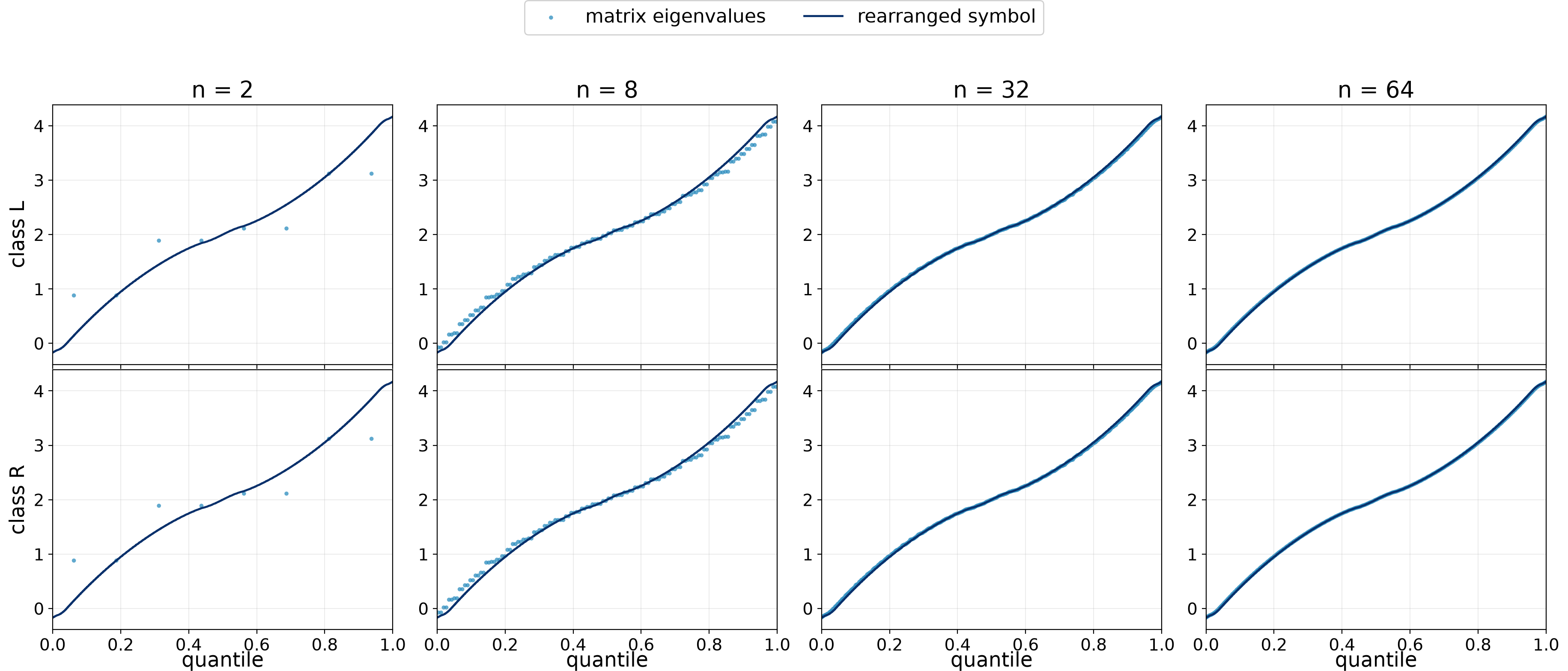}
  \caption{$2\times2$ continuous Hermitian: spectral rearranged symbol (dark blue line) compared with rearranged eigenvalues (blue points). Rows: the one-sided kernel classes $\mathrm{L}$ and $\mathrm{R}$. Columns: the sizes $n=2,8,32,64$.}
\label{fig:quant-2d-2x2-herm}
\end{figure}

\begin{figure}[H]
  \centering
  \includegraphics[width=\linewidth]{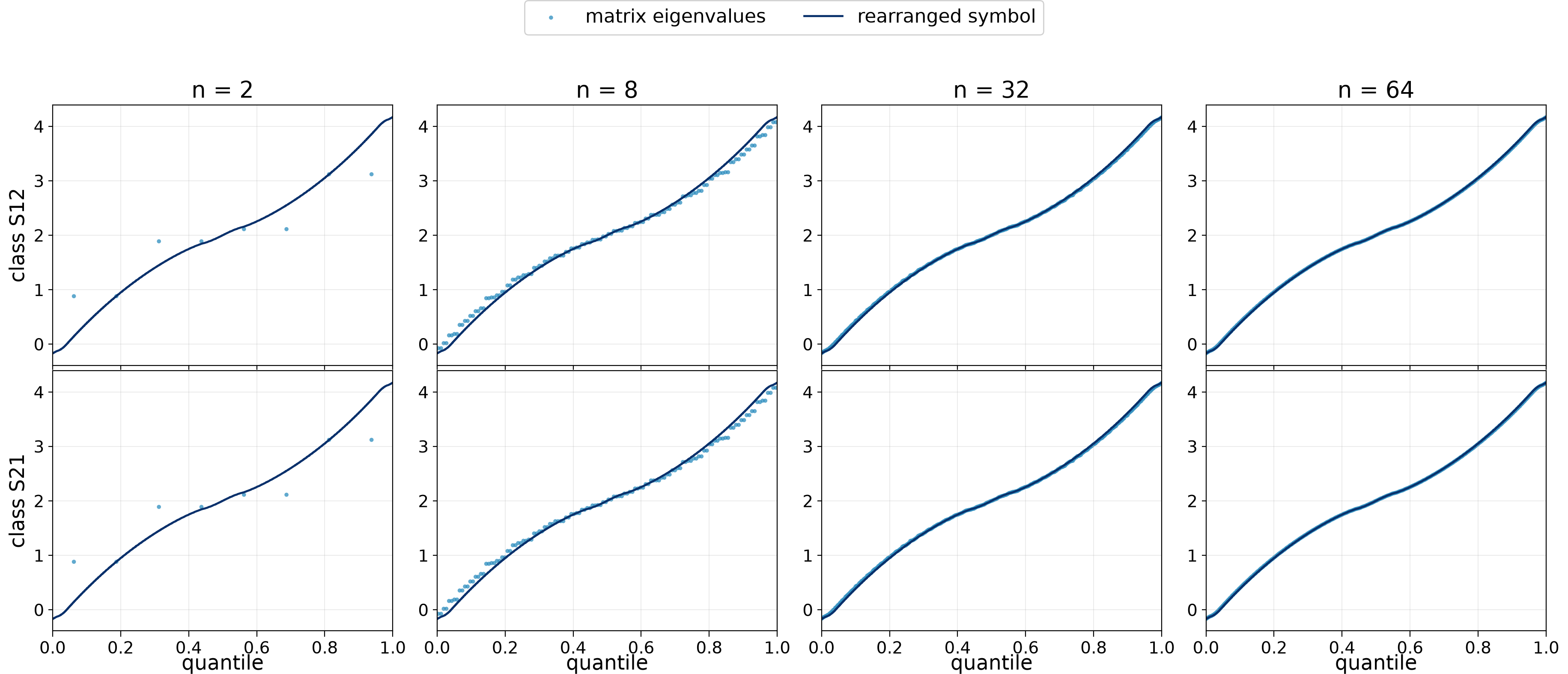}
  \caption{$2\times2$ continuous Hermitian: spectral rearranged symbol (dark blue line) compared with rearranged eigenvalues (blue points). Rows: the sandwich kernel classes $S_{12}$ and $S_{21}$. Columns: the sizes $n=2,8,32,64$.}
\label{fig:quant-2d-2x2-herm-sw}
\end{figure}

\paragraph{Continuous non-Hermitian.}
Consider the following non-symmetric smooth function:
\[
F(\theta_1,\theta_2)=
\begin{bmatrix}
e^{\I\theta_1}+e^{\J\theta_2} &
(\cos\theta_2+\sin\theta_1)\,\K\\[4pt]
(\cos\theta_1-\sin\theta_2)\,\I &
1+\tfrac12\bigl(e^{\J\theta_1}-e^{\K\theta_2}\bigr)
\end{bmatrix}.
\]
\begin{figure}[H]
  \centering
  \includegraphics[width=\linewidth]{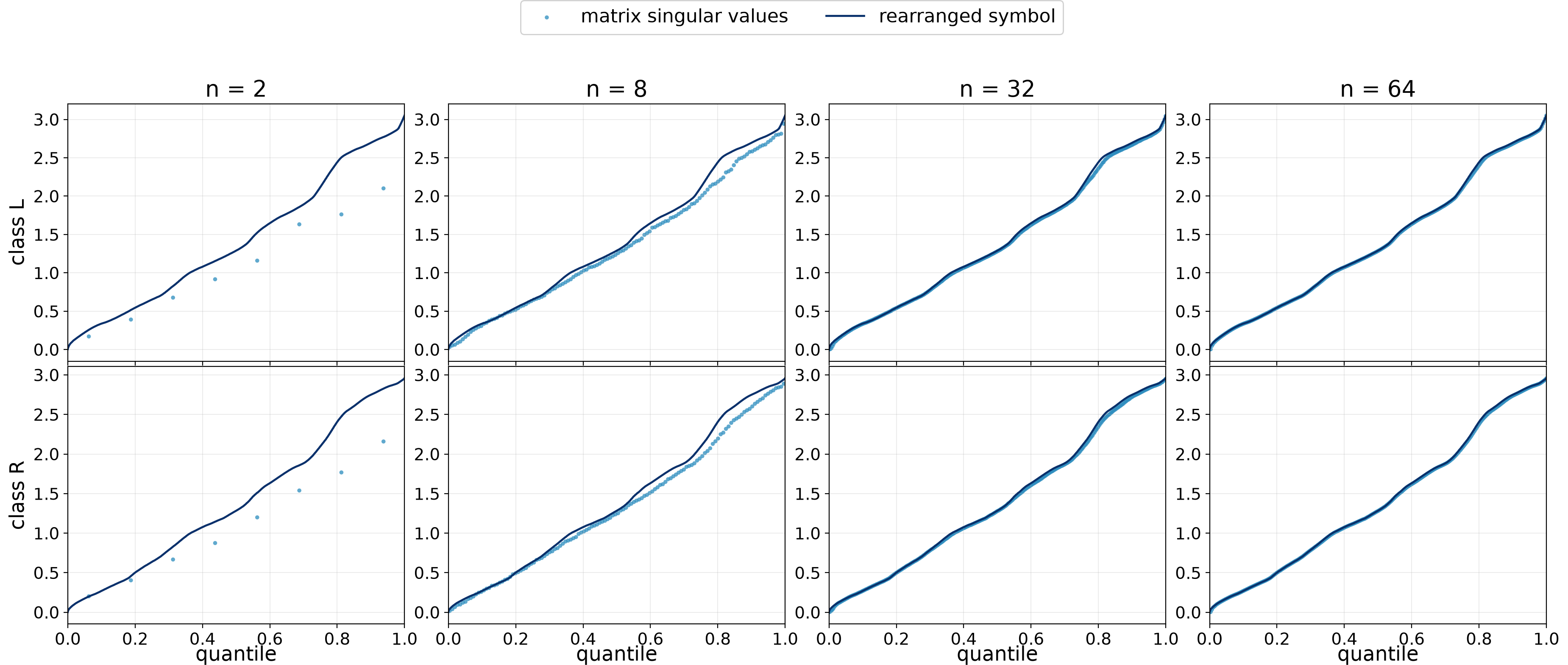}
  \caption{$2\times2$ continuous non-Hermitian: singular value rearranged symbol (dark blue line) compared with rearranged singular values (blue points). Rows: the one-sided kernel classes $\mathrm{L}$ and $\mathrm{R}$. Columns: the sizes $n=2,8,32,64$.}
\label{fig:quant-2d-2x2-nonherm}
\end{figure}

\begin{figure}[H]
  \centering
  \includegraphics[width=\linewidth]{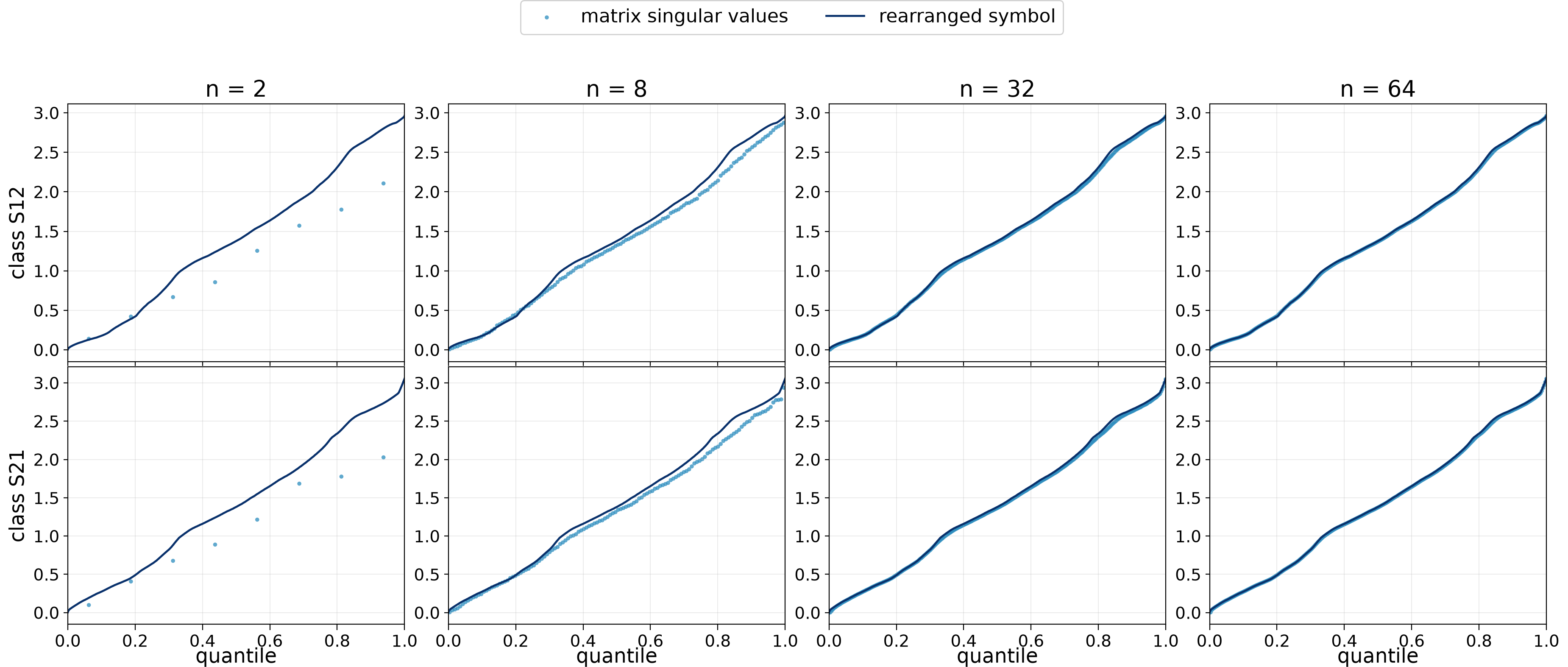}
  \caption{$2\times2$ continuous non-Hermitian: singular value rearranged symbol (dark blue line) compared with rearranged singular values (blue points). Rows: the sandwich kernel classes $S_{12}$ and $S_{21}$. Columns: the sizes $n=2,8,32,64$.}
\label{fig:quant-2d-2x2-nonherm-sw}
\end{figure}

\subsection{$L^1$ symbols}\label{subsec:L1-cont}
Tests use matrix sizes $(n,n)$ with $n=2,8,32,64$ for both Hermitian and non-Hermitian cases.
In the $L^{1}$ examples below, discontinuities are modeled using the sign
function $\operatorname{sgn}(x)$, equal to $+1$, $0$, $-1$ for $x>0$, $x=0$, $x<0$.

\paragraph{$L^1$ Hermitian.}
A \val{piecewise-constant Hermitian} symbol:
\[
F(\theta_1,\theta_2)=
\begin{bmatrix}
2 & \operatorname{sgn}(\theta_1)\,\I+\operatorname{sgn}(\theta_2)(\J+\K)\\[4pt]
-\operatorname{sgn}(\theta_1)\,\I+\operatorname{sgn}(\theta_2)(\J+\K) & 2
\end{bmatrix}.
\]
\begin{figure}[H]
  \centering
  \includegraphics[width=\linewidth]{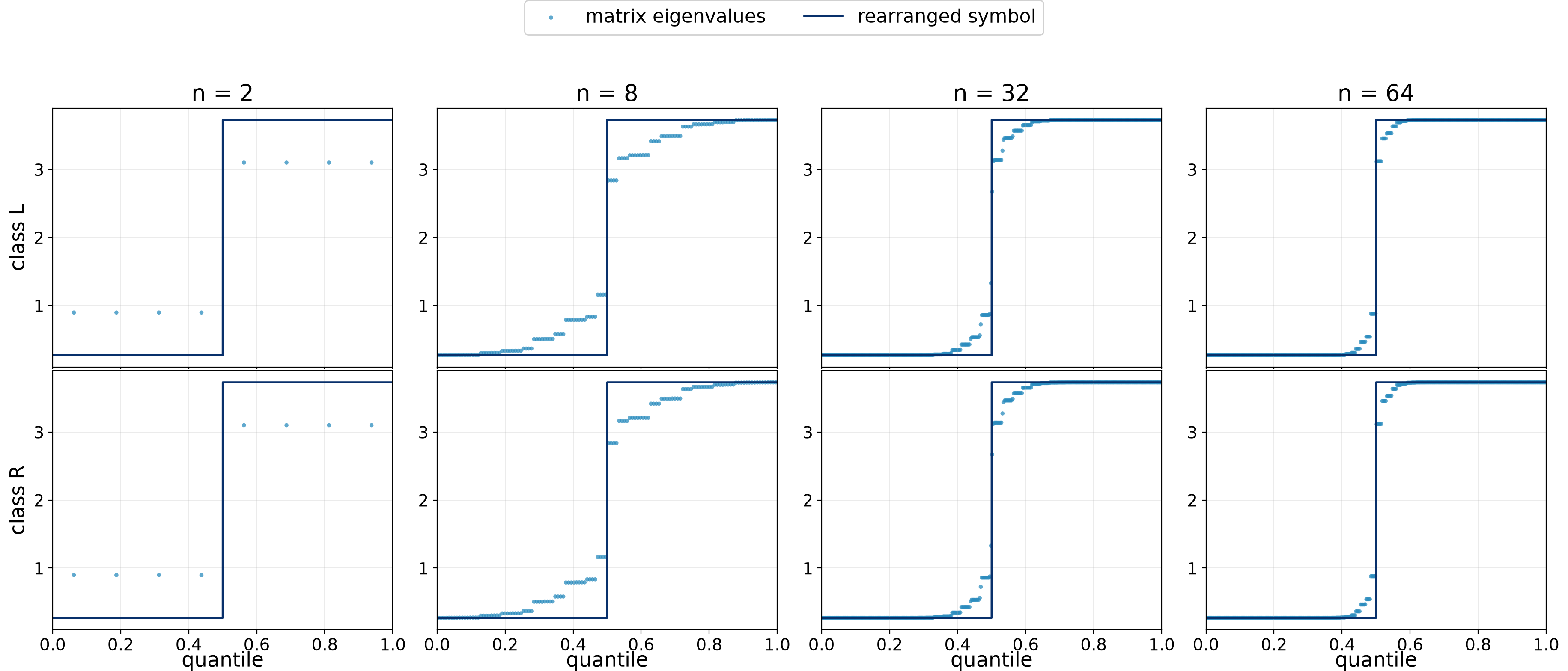}
  \caption{$L^1$ $2\times2$ Hermitian: spectral rearranged symbol (dark blue line) compared with rearranged eigenvalues (blue points). Rows: the one-sided kernel classes $\mathrm{L}$ and $\mathrm{R}$. Columns: the sizes $n=2,8,32,64$.}
\label{fig:quant-2d-2x2-L1-herm}
\end{figure}

\begin{figure}[H]
  \centering
  \includegraphics[width=\linewidth]{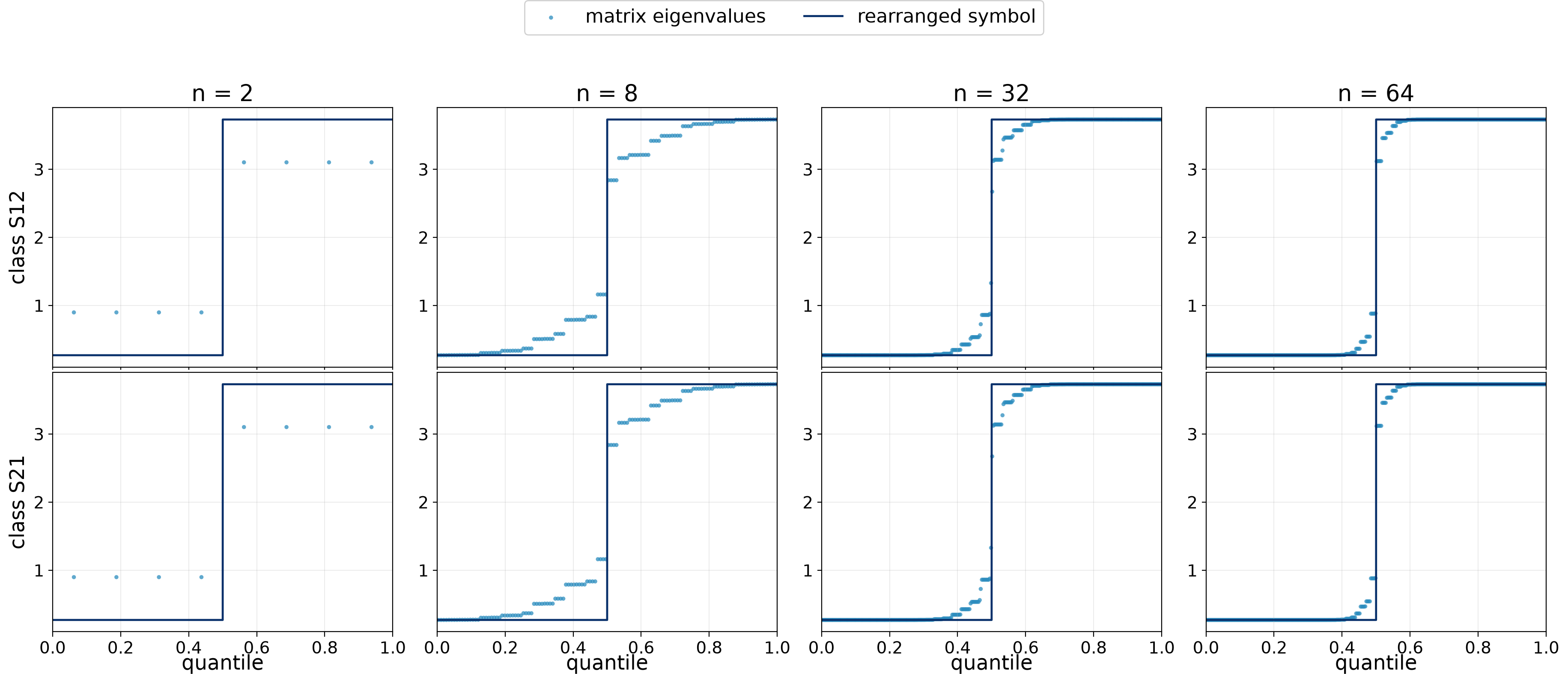}
  \caption{$L^1$ $2\times2$ Hermitian: spectral rearranged symbol (dark blue line) compared with rearranged eigenvalues (blue points). Rows: the sandwich kernel classes $S_{12}$ and $S_{21}$. Columns: the sizes $n=2,8,32,64$.}
\label{fig:quant-2d-2x2-L1-herm-sw}
\end{figure}

\paragraph{$L^1$ non-Hermitian.}
\val{As a final example, consider the non-symmetric function expressed as}
\[
F(\theta_1,\theta_2)=
\begin{bmatrix}
e^{\,\I\theta_1}
+ \operatorname{sgn}(\theta_1)
+ \J\,\operatorname{sgn}(\theta_2)
&
\K\,\operatorname{sgn}\!\bigl(\sin\theta_1+\tfrac14\cos\theta_2\bigr)
\\[6pt]
\I\,\operatorname{sgn}\!\bigl(\cos\theta_1-\sin\theta_2\bigr)
&
1+\tfrac14\!\left(
e^{\,\J\,\operatorname{sgn}(\theta_1)\pi/2}
-
e^{\,\K\,\operatorname{sgn}(\theta_2)\pi/2}
\right)
\end{bmatrix}.
\]

\begin{figure}[H]
  \centering
  \includegraphics[width=\linewidth]{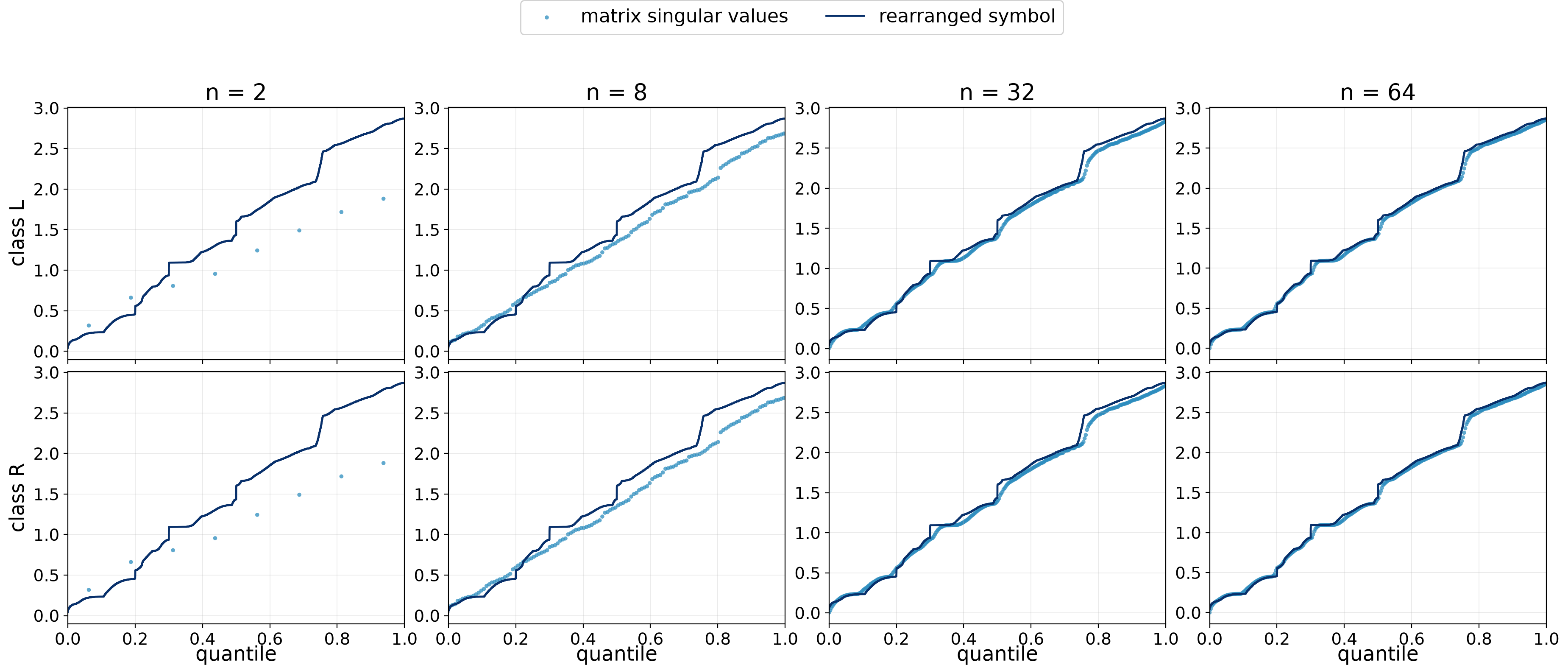}
  \caption{$L^1$ $2\times2$ non-Hermitian: singular value rearranged symbol (dark blue line) compared with rearranged singular values (blue points). Rows: the one-sided kernel classes $\mathrm{L}$ and $\mathrm{R}$. Columns: the sizes $n=2,8,32,64$.}
\label{fig:quant-2d-2x2-L1-nonherm-bis}
\end{figure}

\begin{figure}[H]
  \centering
  \includegraphics[width=\linewidth]{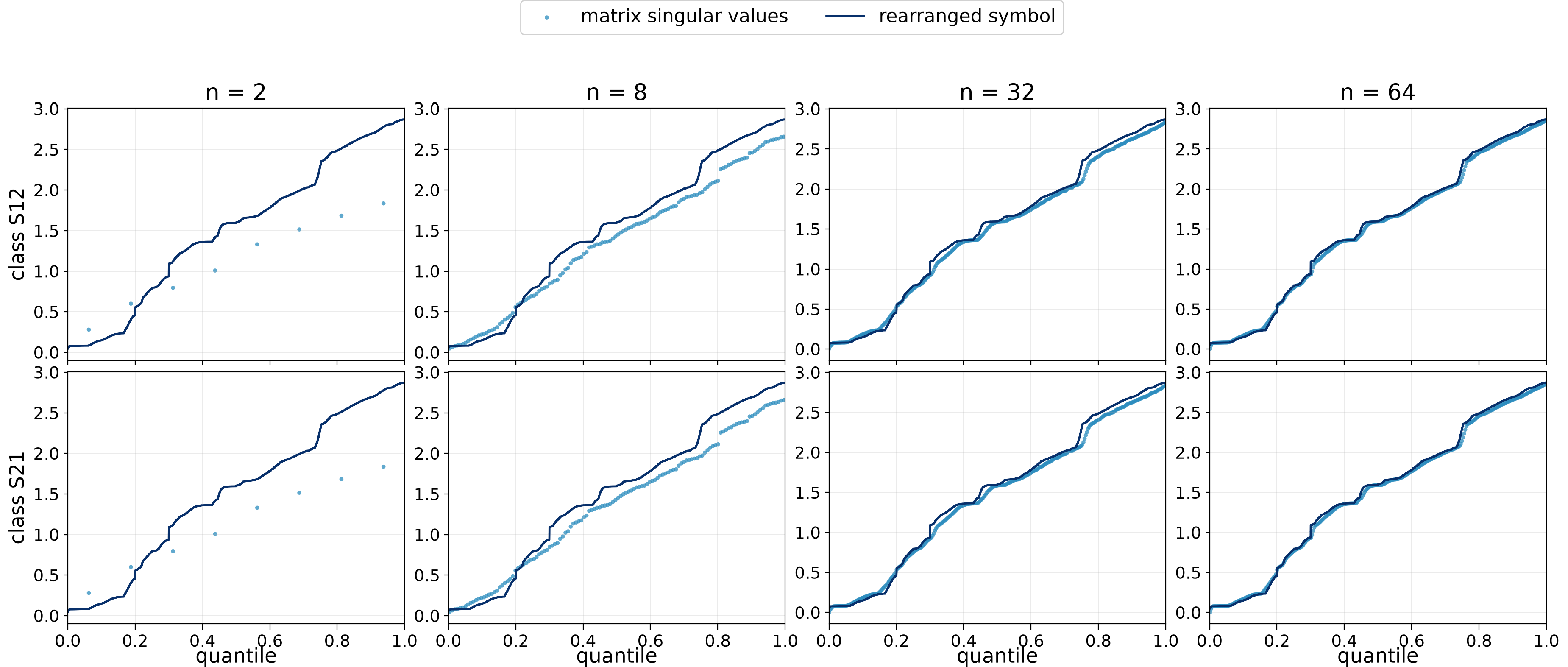}
  \caption{$L^1$ $2\times2$ non-Hermitian: singular value rearranged symbol (dark blue line) compared with rearranged singular values (blue points). Rows: the sandwich kernel classes $S_{12}$ and $S_{21}$. Columns: the sizes $n=2,8,32,64$.}
\label{fig:quant-2d-2x2-L1-nonherm-bis-sw}
\end{figure}

\noindent
\val{Across all four regimes (continuous and $L^1$, Hermitian and
non-Hermitian), the empirical quantile curves approach the corresponding
eigenvalue or singular-value quantiles of the embedded symbols as $\bn$
increases.}

As a final remark, we observe that the agreement between the empirical distributions (of eigenvalues and singular values) and the symbols deduced theoretically is remarkably precise even for small sizes of the quaternion matrices. This fact, which is not covered by our theoretical investigations, should be studied carefully, having in mind quantitative estimates of the convergence speed.

\begin{comment}
\noindent Across all regimes (scalar/block, uni/bivariate, Hermitian/non-Hermitian),
the empirical quantile curves approach their symbol counterparts monotonically in $\mathbf{n}$,
confirming the predicted limit laws.
\end{comment}

%====================================================
\FloatBarrier
\section{Conclusions and a plan for future investigations}\label{sec:conclusion}

We have presented an exhaustive treatment of Weyl eigenvalue and singular value distributions for single-axis quaternion block multilevel Toeplitz matrix sequences generated by $s\times t$ quaternion matrix-valued $d$-variate Lebesgue integrable generating functions: this study solves in the positive a large part of the program given in {\bf RP 4}, items 1, 2, in \cite{visionary}. Furthermore, we have proved localization results in the quaternion Hermitian setting and we have generalized to the specific quaternion case the notion of approximating class of sequences: in particular, we have shown that block multilevel circulant matrix sequences are good approximation candidates for block multilevel Toeplitz matrix sequences in the a.c.s. sense, to be used for preconditioning purposes.
A wide class of visualizations and numerical tests have been included. As future steps we mention
\begin{itemize}
\item fast preconditioning strategies in Krylov methods, when dealing with large quaternion linear systems stemming from modern applications;
\item the definition and study of the GLT $r$-block $d$ level $*$ algebras class in the quaternion setting, $r,d\ge 1$, having in mind variable coefficient problems;
\item quantitative estimates and asymptotic expansions for the spectra of quaternion matrix sequences;
\item extremal spectral and singular value behavior, mimicking the tools and the results in \cite{extr1,extr2,serra_spectral_1999,SIMAX99,extr3};
\item a general analysis beyond the single-axis case;
\item the use of the present powerful machinery in applications and related numerical algorithms.
\end{itemize}

Regarding the first item, we observe that a notion of linear positive operator can be given in the quaternion setting using the results of the present work. Hence we have the tools for considering the challenge of a Korovkin approximation \cite{Koro-Book} for the Frobenius optimal preconditioners \cite{Koro1,Koro2} in fast algebras, following and adapting the technique developed in \cite{Koro1,Koro2,benedetto_optimal_2000}; see also the Research Project {\bf RP 6} in \cite{visionary}. On a different side, also the topological barriers to the superlinear convergence in the multilevel setting \cite{nega1,nega2} represent an interesting challenge worth investigating.

The second item is very ambitious and requires further study and understanding. Starting from \cite{tau-elliptic}, it has been clear that the use of diagonal plus Toeplitz matrix sequences can be very useful for designing fast algorithms in the case of approximated partial differential equations with variable coefficients, when local numerical schemes are used \cite{garoni_Toeplitz_2017,garoni_multilevel_2018,GLTblock2}: a big challenge is the quaternion or tensor extensions; see also the Research Project {\bf RP 4}, items 3, 4, in \cite{visionary}.

With reference to the third item, we recall that linear in time matrix-less algorithms have been developed in \cite{matrixless1,matrixless2,matrixless3} for the computation of all the eigenvalues in complex Toeplitz setting, based on proper asymptotic expansions (refer to \cite{matrixless3} and references therein). A challenge is to obtain conditions under which a similar type of asymptotic expansions hold in the quaternion case, so allowing the extension of matrix-less algorithms.

Finally, we are convinced that the picture would be complete following the directions indicated in the last three items, having in mind concrete applications and numerical techniques for attacking them.

% ---------------------------------------------------------------------
\section*{Acknowledgements}
Valerio Loi and Stefano Serra-Capizzano are members of the research group GNCS (Gruppo Nazionale per il Calcolo Scientifico) of INdAM (Istituto Nazionale di Alta Matematica).

\long\def\RevisionBibliography{%

}
\appendix
\section[Complex results used in the paper]{\val{Complex results used in the paper}}\label{app:complex-results}

\noindent\val{The present appendix states the complex block multilevel Toeplitz,
approximating-class, and circulant results needed for the embedding arguments
in the main text. It uses the notation and normalization adopted in this
paper. We cite the standard results; the
product-compatible rectangular extension in \(\mathrm{ACS}3\) is proved
directly in Appendix~\ref{app:preliminary-proofs}.}

\subsection*{\val{Notation and Toeplitz distributions}}

\val{Let $\bn=(n_1,\ldots,n_d)\in\NN^d$ and set}
\[
\val{\Lambda_{\bn}:=\prod_{\ell=1}^{d}\{0,\ldots,n_\ell-1\},
\qquad
N_{\bn}:=\prod_{\ell=1}^{d}n_\ell.}
\]
\val{Throughout this appendix, $\bn\to\infty$ means $\min_{1\le\ell\le d}n_\ell\to\infty$. For $G\in L^1(\TT^d;\CC^{a\times b})$, we use}
\[
\val{\widehat G(\bk):=\frac{1}{(2\pi)^d}\int_{\TT^d}
G(\btheta)e^{-\I\langle\bk,\btheta\rangle}\,d\btheta,
\qquad
T_{\bn}(G):=\bigl[\widehat G(\alpha-\beta)\bigr]_{\alpha,\beta\in\Lambda_{\bn}}.}
\]

\begin{theorem}[\val{Block multilevel Toeplitz singular-value distribution}]\label{thm:complex-toeplitz-sv}
\val{Let $G\in L^1(\TT^d;\CC^{a\times b})$. Then}
\[
\val{\{T_{\bn}(G)\}_{\bn}\sim_\sigma(G,\TT^d).}
\]
\val{This is the rectangular block multilevel singular-value main result of \cite{TilliNota}; see also \cite{tyrtyshnikov1998}.}
\end{theorem}

\begin{theorem}[\val{Hermitian complex block multilevel Toeplitz eigenvalue distribution}]\label{thm:complex-toeplitz-hermitian}
\val{Let $G\in L^1(\TT^d;\CC^{a\times a})$ and assume that $G(\btheta)=G(\btheta)^*$ for a.e. $\btheta\in\TT^d$. Then}
\[
\val{\{T_{\bn}(G)\}_{\bn}\sim_\lambda(G,\TT^d).}
\]
\val{See \cite{TilliNota,tyrtyshnikov1998}.}
\end{theorem}

\begin{definition}[\val{Spectral essential range}]\label{def:complex-spectral-essential-range}
\val{Let $G:\TT^d\to\CC^{a\times a}$ be measurable. Its spectral essential range is}
\[
\val{R(G):=\left\{z\in\CC:
\left|\left\{\btheta\in\TT^d:
\dist\!\bigl(z,\sigma(G(\btheta))\bigr)<\varepsilon
\right\}\right|>0
\text{ for every }\varepsilon>0\right\}.}
\]
\val{Definition \ref{def:complex-spectral-essential-range} is independent of any measurable labeling of the eigenvalues.}
\end{definition}

\begin{remark}\label{rem:spectral-essential-range-branches}
\begingroup
\color{red}
Let \(\lambda_1,\ldots,\lambda_a\) be any Borel enumeration of the
eigenvalues of \(G\), provided by
Lemma~\ref{lem:borel-eigenvalue-enumeration}. Then
\[
R(G)
=
\bigcup_{j=1}^{a}
\operatorname{ess\,ran}\lambda_j.
\]

Indeed, if
\(z\in\operatorname{ess\,ran}\lambda_j\) for some \(j\), then
\[
\left|
\left\{
\btheta:
\dist\bigl(z,\sigma(G(\btheta))\bigr)<\varepsilon
\right\}
\right|
\ge
\left|
\left\{
\btheta:
|z-\lambda_j(\btheta)|<\varepsilon
\right\}
\right|
>0
\]
for every \(\varepsilon>0\), so \(z\in R(G)\).

Conversely, suppose \(z\in R(G)\). For every \(k\ge1\),
\[
\left\{
\btheta:
\dist\bigl(z,\sigma(G(\btheta))\bigr)<\frac1k
\right\}
=
\bigcup_{j=1}^{a}
\left\{
\btheta:
|z-\lambda_j(\btheta)|<\frac1k
\right\}
\]
has positive measure. Hence, for each \(k\), at least one set on the
right has positive measure. Since there are only finitely many indices,
there exists a fixed index \(j\) that occurs for infinitely many \(k\).
Along this subsequence, the radii tend to zero, and monotonicity with
respect to the radius yields
\[
\left|
\left\{
\btheta:
|z-\lambda_j(\btheta)|<\varepsilon
\right\}
\right|>0
\qquad\text{for every }\varepsilon>0.
\]
Thus \(z\in\operatorname{ess\,ran}\lambda_j\).
\endgroup
\end{remark}

\begin{theorem}[\val{Complex block multilevel Tilli theorem}]\label{thm:complex-tilli}
\val{Let $G\in L^\infty(\TT^d;\CC^{a\times a})$. If}
\[
\val{\operatorname{int}R(G)=\varnothing,
\qquad
\CC\setminus R(G)\text{ is connected},}
\]
\val{then}
\[
\val{\{T_{\bn}(G)\}_{\bn}\sim_\lambda(G,\TT^d).}
\]
\val{See \cite{Donatelli2012}.}
\end{theorem}

\subsection*{\val{Schatten estimates and spectral localization}}

\begin{lemma}[Complex matrix-valued Toeplitz Schatten bound]\label{lem:Tilli-Sp-mv}
Let \(G\in L^{p}(\TT^{d};\CC^{a\times b})\) with \(1\le p\le\infty\) and
\(\bn\in\NN^{d}\), \(N_{\bn}:=\prod_{\ell=1}^{d}n_{\ell}\).
Then
\begin{equation}\label{eq:Tilli-Sp-mv}
\bigl\|T_{\bn}(G)\bigr\|_{p}
\ \le\
\left(\frac{N_{\bn}}{(2\pi)^{d}}\right)^{\!1/p}\ \|G\|_{L^{p}(\TT^{d};\CC^{a\times b})}.
\end{equation}
Moreover, the constant \(1\) is optimal.
\end{lemma}
\val{The estimate is \cite[Cor.~3.5]{serra_noncommutative_2002};
the nonsquare extension is stated in
\cite[Sec.~3.3]{serra_noncommutative_2002}.}

\begin{theorem}[\val{Spectral localization of Hermitian complex block Toeplitz matrices}]\label{thm:complex-toeplitz-localization}
\val{Let $G\in L^1(\TT^d;\CC^{a\times a})$ and assume that $G(\btheta)=G(\btheta)^*$ for a.e. $\btheta$. Set}
\[
\val{m_G:=\operatorname*{ess\,inf}_{\btheta\in\TT^d}\lambda_{\min}(G(\btheta)),
\qquad
M_G:=\operatorname*{ess\,sup}_{\btheta\in\TT^d}\lambda_{\max}(G(\btheta)).}
\]
\val{Then, for every $\bn\in\NN^d$,}
\[
\val{\sigma(T_{\bn}(G))\subseteq[m_G,M_G].}
\]
\val{See \cite{serra_spectral_1999,serra-capizzano_unitarily_2002}.}
\end{theorem}

\subsection*{\val{Complex approximating classes of sequences}}

\begin{definition}[\val{Complex rectangular a.c.s.}]\label{def:complex-rect-acs}
\val{Let $A_n\in\CC^{d_n\times e_n}$ and $r_n:=d_n\wedge e_n\to\infty$. A class of matrix sequences $\bigl\{\{B_{n,m}\}_n\bigr\}_{m\in\NN}$ is an approximating class of sequences for $\{A_n\}_n$ if, for every fixed $m$, there exist $R_{n,m},N_{n,m}\in\CC^{d_n\times e_n}$ and $n_m$ such that, for all $n\ge n_m$,}
\[
\val{A_n=B_{n,m}+R_{n,m}+N_{n,m},
\qquad
\frac{\operatorname{rank}(R_{n,m})}{r_n}\le c(m),
\qquad
\|N_{n,m}\|\le\omega(m),}
\]
\val{where $c(m)\to0$ and $\omega(m)\to0$ as $m\to\infty$.}
\end{definition}
\val{This is \cite[Def.~2.6]{rectangular_glt_2022}.}

\begin{theorem}[\val{Complex a.c.s.\ properties used in the paper}]
\label{thm:complex-acs-properties}
\begingroup
\color{red}

Let \(A_n\in\CC^{d_n\times e_n}\) and set
\[
r_n:=d_n\wedge e_n\longrightarrow\infty.
\]
Then the complex rectangular a.c.s.\ calculus has the following
properties.

\begin{enumerate}

\item[\textup{ACS1.}]
\(\{A_n\}_n\sim_{\sigma}F\) if and only if there exist matrix
sequences
\[
\{B_{n,m}\}_n\sim_{\sigma}F_m
\]
such that
\[
\{B_{n,m}\}_n
\xrightarrow{\mathrm{a.c.s.}}
\{A_n\}_n,
\qquad
F_m\longrightarrow F
\quad\text{in measure}.
\]

\item[\textup{ACS2.}]
If every \(A_n\) is Hermitian and square, then
\(\{A_n\}_n\sim_{\lambda}f\) if and only if there exist Hermitian
matrix sequences
\[
\{B_{n,m}\}_n\sim_{\lambda}f_m
\]
such that
\[
\{B_{n,m}\}_n
\xrightarrow{\mathrm{a.c.s.}}
\{A_n\}_n,
\qquad
f_m\longrightarrow f
\quad\text{in measure}.
\]

\item[\textup{ACS3.}]
Let
\[
A_n,B_{n,m}\in\CC^{d_n\times e_n},
\qquad
A'_n,B'_{n,m}\in\CC^{d'_n\times e'_n},
\]
and assume
\[
\{B_{n,m}\}_n
\xrightarrow{\mathrm{a.c.s.}}
\{A_n\}_n,
\qquad
\{B'_{n,m}\}_n
\xrightarrow{\mathrm{a.c.s.}}
\{A'_n\}_n.
\]
Then
\begin{align*}
&\{B_{n,m}^{*}\}_n
\xrightarrow{\mathrm{a.c.s.}}
\{A_n^{*}\}_n,
\\
&\{\alpha B_{n,m}+\beta B'_{n,m}\}_n
\xrightarrow{\mathrm{a.c.s.}}
\{\alpha A_n+\beta A'_n\}_n
\quad
\text{if }(d'_n,e'_n)=(d_n,e_n),
\\
&\{B_{n,m}B'_{n,m}\}_n
\xrightarrow{\mathrm{a.c.s.}}
\{A_nA'_n\}_n
\quad
\substack{
\text{if }d'_n=e_n,\
\{A_n\}_n,\{A'_n\}_n\text{ are s.u., and}\\[-1pt]
\operatorname{PC}(d,e,e'),
}
\\
&\{B_{n,m}C_n\}_n
\xrightarrow{\mathrm{a.c.s.}}
\{A_nC_n\}_n
\quad
\substack{
\text{if }C_n\in\CC^{e_n\times f_n},\
\{C_n\}_n\text{ is s.u., and}\\[-1pt]
\operatorname{PC}(d,e,f).
}
\end{align*}
Here \(\alpha,\beta\in\CC\), and each product a.c.s.\ is normalized
by the minimum dimension of the product output.

\item[\textup{ACS4.}]
If
\[
A_n=[A_{n,ij}]_{i,j=1}^{s},
\qquad
B_n^{(m)}=[B_{n,ij}^{(m)}]_{i,j=1}^{s},
\]
have a fixed compatible block partition and
\[
\{B_{n,ij}^{(m)}\}_n
\xrightarrow{\mathrm{a.c.s.}}
\{A_{n,ij}\}_n
\qquad
\text{for every }i,j,
\]
then
\[
\{B_n^{(m)}\}_n
\xrightarrow{\mathrm{a.c.s.}}
\{A_n\}_n.
\]

\item[\textup{ACS5.}]
Let \(p\in[1,\infty]\). If, for every \(m\), there exists \(n_m\)
such that, for all \(n\ge n_m\),
\[
\|A_n-B_{n,m}\|_{S_p}
\le
\varepsilon(m,n)\,r_n^{1/p},
\]
where \(r_n^{1/\infty}:=1\) and
\[
\lim_{m\to\infty}\limsup_{n\to\infty}\varepsilon(m,n)=0,
\]
then
\[
\{B_{n,m}\}_n
\xrightarrow{\mathrm{a.c.s.}}
\{A_n\}_n.
\]

\end{enumerate}
\endgroup
\end{theorem}

\noindent
\val{\(\mathrm{ACS}1\), \(\mathrm{ACS}2\), \(\mathrm{ACS}4\), and
\(\mathrm{ACS}5\) are part of the standard complex a.c.s.\ calculus; see
\cite{capizzano_distribution_2001} and
\cite[Chapter~5]{garoni_Toeplitz_2017}. The classical common-size
product proof appears in \cite[Prop.~5.5, pp.~87-88]{garoni_Toeplitz_2017}.
The rectangular normalization follows \cite[Def.~2.6]{rectangular_glt_2022},
and fixed-block rectangular GLT multiplication is given in
\cite[Thm.~4.5]{rectangular_glt_2022}. The dimension-compatible
varying-aspect-ratio product assertions in \(\mathrm{ACS}3\) and
\(\mathrm{ACS}_{\HH}3\) are proved simultaneously below.}

\par\smallskip
\noindent
\hyperref[proof:p25]{%
\val{\emph{The product assertions in
\(\mathrm{ACS}3\) and \(\mathrm{ACS}_{\HH}3\) are proved in
Appendix~\ref*{app:preliminary-proofs}.}}}
\par\smallskip

\subsection*{\val{Block-circulant diagonalization and shift comparison}}

\begin{proposition}[\val{DFT diagonalization of multilevel block circulants}]\label{prop:complex-bcirc-diagonalization}
\val{Let $H:\Lambda_{\bn}\to\CC^{s\times t}$ and}
\[
\val{S:=\sum_{\brho\in\Lambda_{\bn}}P_{\bn}^{(\brho)}\otimes H(\brho),
\qquad
U_L:=F_{\I}^{(\bn)}\otimes I_s,
\qquad
U_R:=F_{\I}^{(\bn)}\otimes I_t.}
\]
\val{For a block family $(X_{\bk})_{\bk\in\Lambda_{\bn}}$, define}
\[
\val{\Bdiag(X_{\bk})_{\bk\in\Lambda_{\bn}}
:=\operatorname{diag}_{\bk\in\Lambda_{\bn}}X_{\bk}.}
\]
\val{Then}
\[
\val{U_LSU_R^*=\Bdiag(\widehat H(\bk))_{\bk\in\Lambda_{\bn}},
\qquad
\widehat H(\bk):=\sum_{\brho\in\Lambda_{\bn}}H(\brho)
e^{-2\pi\I\sum_{\ell=1}^{d}k_\ell\rho_\ell/n_\ell}.}
\]
\val{See \cite{Ng-Book,Golub_VanLoan_2013}.}
\end{proposition}

\begin{lemma}[\val{Toeplitz-circulant shift comparison}]
\label{lem:complex-shift-rank}
\val{For $q\in\ZZ$, define
\[
[J_n^{(q)}]_{ij}:=\mathbf 1_{\{i=j+q\}},
\qquad i,j=0,\ldots,n-1,
\]
and, for $\bk=(k_1,\ldots,k_d)\in\ZZ^d$, set
\[
J_{\bn}^{(\bk)}
:=
\bigotimes_{\ell=1}^{d}J_{n_\ell}^{(k_\ell)}.
\]
Then}
\[
\val{\operatorname{rank}\!\left(J_{\bn}^{(\bk)}-P_{\bn}^{(\bk)}\right)
\le 2\sum_{\ell=1}^{d}|k_\ell|\frac{N_{\bn}}{n_\ell}.}
\]
\val{The estimate follows by applying the one-dimensional boundary-rank estimate in each tensor factor and using a telescoping decomposition; see also \cite{garoni_multilevel_2018}.}
\end{lemma}

\long\def\WorkedAppendix{%
\section[Worked embedded-symbol computations]{\val{Worked embedded-symbol computations}}\label{app:worked}
\noindent\val{We evaluate the four embedded-symbol formulas from
Subsection~\ref{subsec:worked-symbols} at the representative interior point
$\btheta_0=(\pi/3,\pi/4)$ to check them numerically.}

\val{The point $\btheta_0$ is chosen away from the discontinuities of the
$L^1$ symbols. For each generating function, we evaluate the slice functions
$Z,W$ from Subsection~\ref{subsec:worked-symbols} at $\btheta_0$ and at the
reflected arguments occurring in $G_L,G_R,G_{S_{12}},G_{S_{21}}$. In particular,}
\[
  \val{F(\btheta_0)=Z(\btheta_0)+W(\btheta_0)\J,
  \qquad Z(\btheta_0),W(\btheta_0)\in\CC_\I^{\,s\times t}.}
\]
\val{Substituting these values into the four formulas in
Subsection~\ref{subsec:worked-symbols} gives the corresponding $4\times4$
complex matrices. The eigenvalues of $G_\tau(\btheta_0)$ sample the Weyl measure
for Hermitian symbols, while its singular values do so for non-Hermitian symbols.}

\paragraph{Worked example: the continuous Hermitian symbol.}
\val{For the continuous Hermitian slice decomposition in
Subsection~\ref{subsec:worked-symbols}, evaluation at $\btheta_0$ gives}
\[
  Z=\begin{bmatrix} 3.2071 & 0.4330\,\I \\ -0.4330\,\I & 3.2071 \end{bmatrix},\qquad
  W=\begin{bmatrix} 0 & 0.3536+0.0647\,\I \\ 0.3536+0.0647\,\I & 0 \end{bmatrix},
\]
and the four embedded $4\times4$ symbols are
\[
G_L=\begin{bmatrix}
3.2071 & 0.4330\,\I & 0 & 0.3536+0.0647\,\I\\
-0.4330\,\I & 3.2071 & 0.3536+0.0647\,\I & 0\\
0 & 0.3536-0.0647\,\I & 3.2071 & 0.4330\,\I\\
0.3536-0.0647\,\I & 0 & -0.4330\,\I & 3.2071
\end{bmatrix},
\]
\[
G_R=\begin{bmatrix}
3.2071 & 0.4330\,\I & 0 & -0.3536-0.0647\,\I\\
-0.4330\,\I & 3.2071 & -0.3536-0.0647\,\I & 0\\
0 & -0.3536+0.0647\,\I & 3.2071 & 0.4330\,\I\\
-0.3536+0.0647\,\I & 0 & -0.4330\,\I & 3.2071
\end{bmatrix}.
\]
\[
G_{S_{12}}=\begin{bmatrix}
3.2071 & 0.4330\,\I & 0 & -0.3536+0.2415\,\I\\
-0.4330\,\I & 3.2071 & -0.3536+0.2415\,\I & 0\\
0 & -0.3536-0.2415\,\I & 3.2071 & 0.4330\,\I\\
-0.3536-0.2415\,\I & 0 & -0.4330\,\I & 3.2071
\end{bmatrix},
\]
\[
G_{S_{21}}=\begin{bmatrix}
3.2071 & 0.4330\,\I & 0 & 0.3536-0.2415\,\I\\
-0.4330\,\I & 3.2071 & 0.3536-0.2415\,\I & 0\\
0 & 0.3536+0.2415\,\I & 3.2071 & 0.4330\,\I\\
0.3536+0.2415\,\I & 0 & -0.4330\,\I & 3.2071
\end{bmatrix}.
\]
\val{The embedded Hermitian symbols have eigenvalues} $\{2.6444,\,2.6444,\,3.7699,\,3.7699\}$ for $\tau\in\{L,R\}$ and $\{2.5982,\,2.5982,\,3.8161,\,3.8161\}$ for $\tau\in\{S_{12},S_{21}\}$; these are then sampled in non-decreasing order by the quantile of the corresponding quaternion Weyl distribution and reproduce the plot in Figure~\ref{fig:quant-2d-2x2-herm}.

\paragraph{The remaining symbols.}
\val{Evaluating the other three function-level slice decompositions from
Subsection~\ref{subsec:worked-symbols} at $\btheta_0$ gives}
\[
\text{(cont.\ non-Herm.)}\quad
Z=\begin{bmatrix} 1.2071+0.8660\,\I & 0\\ -0.2071\,\I & 0.8964 \end{bmatrix},\;
W=\begin{bmatrix} 0.7071 & 1.5731\,\I\\ 0 & 0.4330-0.3536\,\I \end{bmatrix},
\]
\[
\text{($L^1$ Herm.)}\quad
Z=\begin{bmatrix} 2 & \I\\ -\I & 2 \end{bmatrix},\;
W=\begin{bmatrix} 0 & 1+\I\\ 1+\I & 0 \end{bmatrix};
\]
\[
\text{($L^1$ non-Herm.)}\quad
Z=\begin{bmatrix} 1.5000+0.8660\,\I & 0\\ -\I & 1 \end{bmatrix},\;
W=\begin{bmatrix} 1 & \I\\ 0 & 0.2500-0.2500\,\I \end{bmatrix}.
\]
\val{The four $G_\tau$ are then obtained from the reflected-argument formulas
in Subsection~\ref{subsec:worked-symbols}. Their spectra (eigenvalues for the
Hermitian symbols and singular values otherwise) are collected in
Table~\ref{tab:worked}.}

\begin{table}[H]
\centering
\begin{tabular}{lll}
\toprule
Symbol & Class $\tau$ & Eigenvalues / singular values at $\btheta_0$\\
\midrule
Continuous Hermitian & $L,\,R$ & $2.6444,\ 2.6444,\ 3.7699,\ 3.7699$\\
                     & $S_{12},\,S_{21}$ & $2.5982,\ 2.5982,\ 3.8161,\ 3.8161$\\
\midrule
Continuous non-Hermitian & $L$ & $2.8029,\ 1.6003,\ 1.0385,\ 0.3881$\\
                     & $S_{12}$ & $2.5117,\ 2.0536,\ 1.0360,\ 0.2180$\\
                     & $S_{21}$ & $2.5820,\ 1.7950,\ 1.3230,\ 0.0840$\\
                     & $R$ & $2.6603,\ 1.7932,\ 1.1078,\ 0.3560$\\
\midrule
$L^1$ Hermitian & $L,\,R,\,S_{12},\,S_{21}$ & $0.2679,\ 0.2679,\ 3.7321,\ 3.7321$\\
\midrule
$L^1$ non-Hermitian & $L,\,R$ & $2.6670,\ 1.9417,\ 1.0932,\ 0.4147$\\
                     & $S_{12},\,S_{21}$ & $2.6204,\ 1.8711,\ 1.3568,\ 0.2034$\\
\bottomrule
\end{tabular}
\caption{\val{Embedded-symbol spectra at $\btheta_0=(\pi/3,\pi/4)$ for the four generating functions of Section~\ref{sec:apps}.} }
\label{tab:worked}
\end{table}
}

% =====================================================================
% Proof appendices.
% =====================================================================
\section[Proofs of preliminaries and a.c.s. results]{\val{Proofs of preliminary and a.c.s. results}}\label{app:preliminary-proofs}
\noindent\val{The proofs of the preliminary and approximating-class results
appear here in the order of their statements in the main text.}

\subsection*{\val{Preliminary tools}}

\phantomsection\label{proof:p23}
\begingroup
\hypersetup{linkcolor=red}
\noindent
\val{\textbf{Proof of Proposition~\ref{prop:phi_dup_sv}}}
\par\smallskip
\endgroup

\begin{MovedProofContext}{p23}
\begin{proof}
\val{By Theorem~\ref{thm:q_svd} and Lemma~\ref{lem:phi_star},
\(\PhiSymp(A)\) is unitarily equivalent to \(\PhiSymp(D)\), where
\[
  D=\begin{bmatrix}\Sigma&0\\[1pt]0&0\end{bmatrix}.
\]
Since \(D\) is real, \(\PhiSymp(D)\) is, up to fixed row and column
permutations, equal to \(D\oplus D\). Hence every nonzero singular value of
\(A\) occurs twice among the singular values of \(\PhiSymp(A)\), while all
the remaining singular values are zero.}
\end{proof}
\end{MovedProofContext}

\begin{lemma}[\val{Borel eigenvalue enumerations}]
\label{lem:borel-eigenvalue-enumeration}
\begingroup
\color{red}
For every \(m\in\NN\), there exist Borel maps
\[
\mu_1,\ldots,\mu_m:\CC^{m\times m}\longrightarrow\CC
\]
such that, for every \(B\in\CC^{m\times m}\), their values are the
eigenvalues of \(B\), counted with algebraic multiplicity.

Moreover, for every \(r\in\NN\), there exist Borel maps
\[
\lambda_1^+,\ldots,\lambda_r^+:
\HH^{r\times r}\longrightarrow
\{z\in\CC:\operatorname{Im}z\ge0\}
\]
whose values are the canonical eigenvalues.
\endgroup
\end{lemma}

\begin{proof}
\begingroup
\color{red}
By \cite[Cor.~2]{azoff_borel_1974}, for each \(m\) there is a Borel map
\(U_m\) assigning to every \(B\in\CC^{m\times m}\) a unitary matrix such
that
\[
T_m(B):=U_m(B)^*BU_m(B)
\]
is upper triangular. Defining
\[
\mu_j(B):=\bigl(T_m(B)\bigr)_{jj},
\qquad j=1,\ldots,m,
\]
gives Borel maps whose values are the eigenvalues of \(B\).

Now let \(A\in\HH^{r\times r}\), set \(B:=\PhiSymp_1(A)\), and define
\[
\rho(z):=\operatorname{Re}z+\I|\operatorname{Im}z|.
\]
Let
\[
\eta_1(A)\preceq\cdots\preceq\eta_{2r}(A)
\]
be the lexicographically ordered values
\(\rho(\mu_1(B)),\ldots,\rho(\mu_{2r}(B))\). Since finite
lexicographic sorting is Borel, the maps \(\eta_j\) are Borel.

The next symmetry argument follows from
the paired Jordan structure; see
\cite[Thm.~3.2]{loring_factorization_2012}. The exact conjugate-pairing
and real-parity formulation can be found in
\cite[Lem.~3.2]{han_yu_sun_jordan_2019}. For
\[
J_r:=\begin{bmatrix}0&I_r\\-I_r&0\end{bmatrix},
\]
the block form of \(B\) gives \(BJ_r=J_r\overline B\). Hence the
conjugate-linear map
\[
T(v):=J_r\overline v
\]
satisfies \(T^2=-I\) and
\[
(B-\overline\lambda I)^kT=T(B-\lambda I)^k,
\qquad k\ge1.
\]
Thus \(T\) maps the generalized eigenspace of \(\lambda\) bijectively
onto that of \(\overline\lambda\), so conjugate eigenvalues have equal
algebraic multiplicity. For real \(\lambda\), the generalized eigenspace
has even dimension $d$ in the complex field: in a basis, the restriction of \(T\) is
\(v\mapsto C\overline v\), so \(C\overline C=-I_d\) and
\(|\det C|^2=(-1)^d=1\).

Therefore every distinct value in the folded list occurs with even
multiplicity. Since each preceding constant block also has even length,
the entries in the even positions retain exactly half of every block.
Consequently,
\[
\lambda_j^+(A):=\eta_{2j}(A),
\qquad j=1,\ldots,r,
\]
is a Borel enumeration of the canonical eigenvalues.
\endgroup
\end{proof}

\phantomsection\label{proof:p01}
\begingroup
\hypersetup{linkcolor=red}
\noindent
\val{\textbf{Proof of Lemma~\ref{lem:meas-spectral-functional}}}
\par\smallskip
\endgroup

\begin{MovedProofContext}{p01}
\begin{proof}
\val{Set
\[
  F:=\PhiSymp\circ f:D\longrightarrow\CC^{2r\times 2r},
\]
which is measurable. By \cite[Lemma~2.5]{GLTblock2}, the maps
\(x\mapsto\|F(x)\|_{p}\) are measurable. Hence, by
Corollary~\ref{cor:s_p},
\[
  \|f(x)\|_{p}
  =2^{-1/p}\|F(x)\|_{p},
\]
where \(2^{-1/\infty}=1\). Therefore,
\(x\mapsto\|f(x)\|_{p}\) is measurable.}

\val{The ordered singular values of a complex matrix depend continuously on
its entries. Moreover, Proposition~\ref{prop:phi_dup_sv} gives
\[
  \sigma_{j}(f(x))
  =\sigma_{2j-1}(F(x))
  =\sigma_{2j}(F(x)),
  \qquad j=1,\ldots,r.
\]
Thus the singular value map of \(f\) is measurable.}

\val{By \cite[Lemma~2.5 and Remark~2.6]{GLTblock2},
continuous symmetric functions of the eigenvalues of the measurable complex
matrix field \(F\) are measurable and independent of the ordering of those
eigenvalues. By Theorem~\ref{thm:canonical},
the eigenvalues of \(F(x)\) occur in conjugate pairs, and the canonical
eigenvalues of \(f(x)\) are obtained by selecting Borel-measurable
representatives in the closed upper half-plane. Since \(g\) is
continuous and symmetric, its value is independent of the ordering of the
selected eigenvalues, and the stated map is measurable.}
\end{proof}
\end{MovedProofContext}

\phantomsection\label{proof:p02}
\begingroup\hypersetup{linkcolor=red}\noindent\val{\textbf{Proof of Lemma~\ref{lem:Phi-sandwich-functional}}}\par\smallskip\endgroup
\begin{MovedProofContext}{p02}
\begin{proof}
Use \(\J\alpha=\overline{\alpha}\,\J\) for all \(\alpha\in\CC_{\I}\) and the identities
\(\J e^{-\I\langle\bm m,\btheta\rangle}=e^{+\I\langle\bm m,\btheta\rangle}\J\),
\(e^{-\I\langle\bm m,\btheta\rangle}\J=\J e^{+\I\langle\bm m,\btheta\rangle}\),
to rewrite \[
  L_{S_{L}}(\btheta,\bm m)\,F(\btheta)\,R_{S_{R}}(\btheta,\bm m)
  = Z(\btheta)\,e^{-\I\langle\bm m,\btheta\rangle}
    + W(\btheta)\,e^{-\I\langle\tilde{\bm m},\btheta\rangle}\,\J.
\]
Then apply \(\PhiSymp\) and integrate entrywise.
\end{proof}
\end{MovedProofContext}

\phantomsection\label{proof:p24}
\begingroup\hypersetup{linkcolor=red}\noindent\val{\textbf{Proof of Proposition~\ref{prop:density-sandwich-functional}}}\par\smallskip\endgroup
\begin{MovedProofContext}{p24}
\begin{proof}
\val{Write \(F=Z+W\J\), where
\(Z,W\in L^{p}(\TT^{d};\CC_{\I}^{m\times n})\). By the density of complex
matrix-valued trigonometric polynomials, choose sequences
\[
P_N(\btheta)=\sum_{\bm m\in\mathcal F_N}Z_{\bm m,N}
e^{-\I\langle\bm m,\btheta\rangle},
\qquad
Q_N(\btheta)=\sum_{\bm m\in\mathcal F_N}W_{\bm m,N}
e^{-\I\langle\bm m,\btheta\rangle},
\]
converging to \(Z\) and \(W\), respectively, in \(L^p\). We may take the
finite support \(\mathcal F_N\) invariant under the involution
\(\bm m\mapsto\widetilde{\bm m}\) from
Lemma~\ref{lem:Phi-sandwich-functional}, setting missing coefficients to
zero. Define
\[
A_{\bm m,N}:=Z_{\bm m,N}+W_{\widetilde{\bm m},N}\J.
\]
Using \(\J e^{-\I b}=e^{\I b}\J\), we obtain
\[
\sum_{\bm m\in\mathcal F_N}
L_{S_L}(\btheta,\bm m)A_{\bm m,N}R_{S_R}(\btheta,\bm m)
=P_N(\btheta)+Q_N(\btheta)\J.
\]
Thus each pair of complex trigonometric approximants recombines as a
sandwich trigonometric polynomial. The slice decomposition is a continuous
real-linear isomorphism, so convergence of \(P_N\) and \(Q_N\) gives convergence
to \(F\) in \(L^p(\TT^d;\HH^{m\times n})\).}
\end{proof}
\end{MovedProofContext}

\phantomsection\label{proof:p03}
\begingroup\hypersetup{linkcolor=red}\noindent\val{\textbf{Proof of Proposition~\ref{prop:reduction-right}}}\par\smallskip\endgroup
\begin{MovedProofContext}{p03}
\begin{proof}
Fix $\bm m\in\ZZ^{d}$ and write $A_{\bm m}=Z_{\bm m}+W_{\bm m}\J$ with
$Z_{\bm m},W_{\bm m}\in\CC_{\I}^{m\times n}$. Set
$a=\langle\bm m,\btheta\rangle_{S_{L}}$ and $b=\langle\bm m,\btheta\rangle_{S_{R}}$.
Since $e^{-\I a},e^{-\I b}\in\CC_{\I}$ commute with $Z_{\bm m},W_{\bm m}$ and
$\J\alpha=\overline{\alpha}\,\J$ for all $\alpha\in\CC_{\I}$, we compute
\[
L_{S_{L}}A_{\bm m}R_{S_{R}}
=e^{-\I a}(Z_{\bm m}+W_{\bm m}\J)e^{-\I b}
=Z_{\bm m}e^{-\I(a+b)}+W_{\bm m}\,(e^{-\I a}\J)e^{-\I b}
=Z_{\bm m}e^{-\I(a+b)}+W_{\bm m}\,e^{-\I(a-b)}\J.
\]
Because $a+b=\langle\bm m,\btheta\rangle$ and $a-b=\langle\tilde{\bm m},\btheta\rangle$, this yields \eqref{eq:poly-reduction}.

Now let $F=Z+W\J\in L^{1}(\TT^{d};\HH^{m\times n})$ with $Z,W:\TT^{d}\to\CC_{\I}^{m\times n}$.
Since the decomposition $F\mapsto(Z,W)$ is given by a bounded linear map
$\HH^{m\times n}\to\CC_{\I}^{m\times n}\times\CC_{\I}^{m\times n}$, we infer $Z,W\in L^{1}$.
Applying the previous algebra pointwise and integrating entrywise,
\[
\widehat{F}^{(S_{L},S_{R})}(\bm m)
=\frac{1}{(2\pi)^{d}}\!\int_{\TT^{d}}
\bigl(Z(\btheta)e^{-\I\langle\bm m,\btheta\rangle}
+W(\btheta)e^{-\I\langle\tilde{\bm m},\btheta\rangle}\J\bigr)\,d\btheta
=\widehat{Z}(\bm m)+\widehat{W}(\tilde{\bm m})\,\J,
\]
which is \eqref{eq:coeff-reduction}.

For \eqref{eq:left-right}, by definition
$\widehat{F}_{\mathrm{L}}(\bm m)=\frac{1}{(2\pi)^{d}}\!\int_{\TT^{d}}e^{-\I\langle\bm m,\btheta\rangle}F(\btheta)\,d\btheta$.
Taking quaternionic conjugates, using $\overline{AB}=\overline{B}\,\overline{A}$ and
$\overline{e^{-\I t}}=e^{\I t}$,
\[
\overline{\widehat{F}_{\mathrm{L}}(\bm m)}
=\frac{1}{(2\pi)^{d}}\!\int_{\TT^{d}}\overline{F(\btheta)}\,e^{\I\langle\bm m,\btheta\rangle}\,d\btheta
=\frac{1}{(2\pi)^{d}}\!\int_{\TT^{d}}\overline{F(\btheta)}\,e^{-\I\langle-\bm m,\btheta\rangle}\,d\btheta
=\widehat{\,\overline{F}\,}_{\mathrm{R}}(-\bm m),
\]
and conjugating both sides gives the claim.
\end{proof}
\end{MovedProofContext}

\phantomsection\label{proof:p04}
\begingroup\hypersetup{linkcolor=red}\noindent\val{\textbf{Proof of Lemma~\ref{lem:norm-rank-phi}}}\par\smallskip\endgroup
\begin{MovedProofContext}{p04}
\begin{proof}
The singular values of \(\PhiSymp(X)\) are those of \(X\), each duplicated (Proposition~\ref{prop:phi_dup_sv}); the equalities follow, using also Corollary~\ref{cor:s_p} and \(\|\cdot\|=\|\cdot\|_{\infty}\).
\end{proof}
\end{MovedProofContext}
\subsection*{\val{Approximating classes of sequences}}

\phantomsection\label{proof:p05}
\begingroup\hypersetup{linkcolor=red}\noindent\val{\textbf{Proof of Theorem~\ref{thm:rect-embedding}}}\par\smallskip\endgroup
\begin{MovedProofContext}{p05}
\begin{proof}
We use Lemma~\ref{lem:norm-rank-phi}.

\medskip
\noindent\emph{(a)\(\Rightarrow\)(b).}
Fix \(m\) and \(n\ge n_m\).
From \eqref{eq:rect-acs-decomp},
\(A_{n}=B_{n,m}+R_{n,m}+N_{n,m}\) with
\(\operatorname{rank}_{\HH}(R_{n,m})/r_n\le c(m)\) and \(\|N_{n,m}\|\le\omega(m)\).
Set
\[
C_{n,m}:=\PhiSymp(B_{n,m}),\quad
S_{n,m}:=\PhiSymp(R_{n,m}),\quad
T_{n,m}:=\PhiSymp(N_{n,m}).
\]
Then
\[
\begin{aligned}
\PhiSymp(A_n)&=C_{n,m}+S_{n,m}+T_{n,m},\\
\frac{\operatorname{rank}_{\CC}(S_{n,m})}{2r_n}
&=\frac{2\,\operatorname{rank}_{\HH}(R_{n,m})}{2r_n}
\le c(m),\\
\|T_{n,m}\|&=\|N_{n,m}\|\le\omega(m).
\end{aligned}
\]
so \(\{C_{n,m}\}_{n,m}\) is a complex rectangular a.c.s.\ with the same \(n_m\), \(c(m)\), and \(\omega(m)\).

\medskip
\noindent\emph{(b)\(\Rightarrow\)(a).}
\val{Set
\[
\mathsf J:=\begin{bmatrix}0&1\\[-1pt]-1&0\end{bmatrix},
\qquad
\mathsf J_k:=I_k\otimes\mathsf J.
\]
For each \(n\), define
\[
\begin{aligned}
\tau_{d_n,e_n}
&\colon \CC^{2d_n\times 2e_n}
\longrightarrow \CC^{2d_n\times 2e_n},
&\qquad
\tau_{d_n,e_n}(X)
&:=-\mathsf J_{d_n}\,\overline X\,\mathsf J_{e_n},
\\
\Pi_{d_n,e_n}
&\colon \CC^{2d_n\times 2e_n}
\longrightarrow \CC^{2d_n\times 2e_n},
&\qquad
\Pi_{d_n,e_n}(X)
&:=\tfrac12\bigl(X+\tau_{d_n,e_n}(X)\bigr).
\end{aligned}
\]
The map \(\tau_{d_n,e_n}\) is conjugate-linear, whereas
\(\Pi_{d_n,e_n}\) is real-linear. Since \(\mathsf J_k\) is real
unitary and \(\mathsf J_k^2=-I_{2k}\),
\[
\tau_{d_n,e_n}^{2}=\mathrm{id},\qquad
\|\tau_{d_n,e_n}(X)\|=\|X\|,\qquad
\operatorname{rank}_{\CC}\!\bigl(\tau_{d_n,e_n}(X)\bigr)
=\operatorname{rank}_{\CC}(X).
\]
To identify the fixed-point space, write \(X=(X_{ij})_{i,j}\) as a
\(d_n\times e_n\) block matrix with \(X_{ij}\in\CC^{2\times2}\). If
\[
X_{ij}=\begin{bmatrix}z&w\\ u&v\end{bmatrix},
\]
then
\[
-\mathsf J\,\overline{X_{ij}}\,\mathsf J
=
\begin{bmatrix}
\overline v&-\overline u\\
-\overline w&\overline z
\end{bmatrix}.
\]
Hence \(\tau_{d_n,e_n}(X)=X\) if and only if every block has the form
\[
X_{ij}
=
\begin{bmatrix}
z&w\\
-\overline w&\overline z
\end{bmatrix}
=
\PhiSymp(z+w\J).
\]
Therefore
\[
\mathrm{Fix}(\tau_{d_n,e_n})
=
\PhiSymp\!\left(\HH^{d_n\times e_n}\right).
\]
Since \(\tau_{d_n,e_n}^{2}=\mathrm{id}\), the map
\(\Pi_{d_n,e_n}\) is the real-linear projector onto this fixed-point
space, and
\[
\|\Pi_{d_n,e_n}(X)\|
\le
\tfrac12\bigl(\|X\|+\|\tau_{d_n,e_n}(X)\|\bigr)
=
\|X\|.
\]}
\val{For \(n\ge n_m\), assumption \textup{(b)} gives
\begin{equation}\label{eq:phi-acs}
\PhiSymp(A_n)
=
\PhiSymp(B_{n,m})+S_{n,m}+T_{n,m},\qquad
\frac{\operatorname{rank}_{\CC}(S_{n,m})}{2r_n}\le c(m),\ \
\|T_{n,m}\|\le\omega(m).
\end{equation}}
\val{Since \(\PhiSymp(B_{n,m})\in\mathrm{Ran}(\PhiSymp)\), projecting
\eqref{eq:phi-acs} onto \(\mathrm{Ran}(\PhiSymp)\) gives
\[
\PhiSymp(A_n)
=
\PhiSymp(B_{n,m})
+
\Pi_{d_n,e_n}(S_{n,m})
+
\Pi_{d_n,e_n}(T_{n,m}).
\]
}
with
\[
\begin{aligned}
\|\Pi_{d_n,e_n}(T_{n,m})\|&\le\omega(m),\\
\operatorname{rank}_{\CC}\!\bigl(\Pi_{d_n,e_n}(S_{n,m})\bigr)
&\le \operatorname{rank}_{\CC}(S_{n,m})
+\operatorname{rank}_{\CC}\!\bigl(\tau_{d_n,e_n}(S_{n,m})\bigr)\\
&\le 2\,\operatorname{rank}_{\CC}(S_{n,m}).
\end{aligned}
\]
\val{
Since the projected error terms lie in \(\mathrm{Ran}(\PhiSymp)\), set
\[
R_{n,m}
:=
\PhiSymp^{-1}\!\bigl(\Pi_{d_n,e_n}(S_{n,m})\bigr),
\qquad
N_{n,m}
:=
\PhiSymp^{-1}\!\bigl(\Pi_{d_n,e_n}(T_{n,m})\bigr).
\]
}
Then \(A_n=B_{n,m}+R_{n,m}+N_{n,m}\) and
\[
\frac{\operatorname{rank}_{\HH}(R_{n,m})}{r_n}
=\frac{\operatorname{rank}_{\CC}(\Pi_{d_n,e_n}(S_{n,m}))}{2r_n}\le 2\,c(m),
\qquad
\|N_{n,m}\|=\|\Pi_{d_n,e_n}(T_{n,m})\|\le\omega(m).
\]
Thus \(\{\{B_{n,m}\}_{n}: \ m\}\) is a quaternionic a.c.s.\ with the same \(n_m\), the same \(\omega(m)\), and rank control at most \(2c(m)\).
\end{proof}
\end{MovedProofContext}

\phantomsection\label{proof:p25}
\begingroup
\hypersetup{linkcolor=red}
\noindent
\val{\textbf{Proof of the rectangular product axioms
\(\mathrm{ACS}3\) and \(\mathrm{ACS}_{\HH}3\).}}
\par\smallskip
\endgroup

\begin{MovedProofContext}{p25}
\begin{proof}
\begingroup
\color{red}
Let \(\mathbb F\in\{\CC,\HH\}\), and put
\[
r_n:=d_n\wedge e_n,\qquad
s_n:=e_n\wedge f_n,\qquad
t_n:=d_n\wedge f_n.
\]
By \(\operatorname{PC}(d,e,f)\), there is \(\Gamma>0\) such that
\[
r_n+s_n\le \Gamma t_n
\]
for all sufficiently large \(n\).

After replacing the controls of the two a.c.s.\ decompositions by
common majorants, denote them by \(c(m),\omega(m)\to0\). Likewise, let
\(\rho(M)\to0\) dominate the sparse-unboundedness tails of the two
target sequences. The proof of
\cite[Prop.~5.5, pp.~87-88]{garoni_Toeplitz_2017}, based on the
singular-value truncation in
\cite[Prop.~5.3]{garoni_Toeplitz_2017}, uses only SVD truncation and
the elementary rank and norm inequalities. The argument therefore extends
to composable rectangular matrices over \(\mathbb F\) (using
Theorem~\ref{thm:q_svd} when \(\mathbb F=\HH\)) and only the rank
normalization changes. Indeed, for every \(M>0\), we obtain
\[
A_nC_n-B_{n,m}D_{n,m}=E_{n,m}+F_{n,m},
\]
with
\[
\operatorname{rank}_{\mathbb F}(E_{n,m})
\le
(2c(m)+\rho(M))r_n+(c(m)+\rho(M))s_n,
\qquad
\|F_{n,m}\|\le 2M\omega(m)+\omega(m)^2.
\]
Therefore
\[
\frac{\operatorname{rank}_{\mathbb F}(E_{n,m})}{t_n}
\le
\Gamma\bigl(2c(m)+\rho(M)\bigr).
\]
Choosing
\[
M_m:=\bigl(\omega(m)+m^{-2}\bigr)^{-1/2}
\]
gives \(M_m\to\infty\) and \(M_m\omega(m)\to0\). Hence, we deduce
\[
\{B_{n,m}D_{n,m}\}_n
\xrightarrow{\mathrm{a.c.s.}}
\{A_nC_n\}_n.
\]
Taking \(C_n=A'_n\), \(D_{n,m}=B'_{n,m}\), and \(f_n=e'_n\) yields
the two-approximant product assertion.

For fixed right multiplication, write
\[
A_n-B_{n,m}=R_{n,m}+N_{n,m},
\qquad
C_n=\widehat C_{n,M}+\widetilde C_{n,M},
\]
where
\[
\operatorname{rank}_{\mathbb F}(R_{n,m})\le c(m)r_n,
\quad
\|N_{n,m}\|\le\omega(m),
\quad
\operatorname{rank}_{\mathbb F}(\widehat C_{n,M})\le\rho(M)s_n,
\quad
\|\widetilde C_{n,M}\|\le M.
\]
Then
\[
A_nC_n-B_{n,m}C_n
=
\underbrace{R_{n,m}C_n+N_{n,m}\widehat C_{n,M}}_{E_{n,m}}
+
\underbrace{N_{n,m}\widetilde C_{n,M}}_{F_{n,m}},
\]
and consequently
\[
\frac{\operatorname{rank}_{\mathbb F}(E_{n,m})}{t_n}
\le
\Gamma\bigl(c(m)+\rho(M)\bigr),
\qquad
\|F_{n,m}\|\le M\omega(m).
\]
Taking \(M=M_m\) proves the fixed-multiplier assertion. Finally,
\(\mathbb F=\CC\) gives \(\mathrm{ACS}3\), while \(\mathbb F=\HH\)
gives \(\mathrm{ACS}_{\HH}3\).
\endgroup
\end{proof}
\end{MovedProofContext}

\section[Proofs of Toeplitz and circulant results]{\val{Proofs of Toeplitz and circulant results}}\label{app:toeplitz-proofs}
\noindent\val{The proofs of the Toeplitz, distribution, localization, QDFT,
and circulant results appear here in the order of their statements in the main
text.}
\subsection*{\val{Toeplitz structural identities}}

\phantomsection\label{proof:p06}
\begingroup\hypersetup{linkcolor=red}\noindent\val{\textbf{Proof of Proposition~\ref{prop:Toeplitz-right-reduction}}}\par\smallskip\endgroup
\begin{MovedProofContext}{p06}
\begin{proof}
Fix $\bm k\in\mathbb{Z}^{d}$. By equation \eqref{eq:coeff-reduction}, we write
\[
\widehat F^{(S_{L},S_{R})}(\bm k)=\widehat Z(\bm k)+\widehat W(\tilde{\bm k})\,\J,
\]
where $\tilde k_{j}=k_{j}$ for $j\in S_{L}$ and $\tilde k_{j}=-k_{j}$ for $j\in S_{R}$.
Given $H=Z+(W\circ r_{S_{L}})\J$, applying \eqref{eq:coeff-reduction} with the right kernel
($S_{L}=\varnothing,S_{R}=[d]$) yields
\[
\widehat H^{(R)}(\bm k)=\widehat Z(\bm k)+\widehat{\,W\circ r_{S_{L}}\,}(-\bm k)\,\J,
\]
and since $\langle-\bm k,r_{S_{L}}(\btheta)\rangle=\langle\tilde{\bm k},\btheta\rangle$ and
Lebesgue measure on $\TT^{d}$ is invariant under $r_{S_{L}}$,
\[
\widehat{\,W\circ r_{S_{L}}\,}(-\bm k)=\widehat W(\tilde{\bm k}).
\]
Thus $\widehat H^{(R)}(\bm k)=\widehat F^{(S_{L},S_{R})}(\bm k)$ for all $\bm k$, which implies
\[
\bigl[T_{\bn}^{(S_{L},S_{R})}(F)\bigr]_{\alpha,\beta}
=\widehat F^{(S_{L},S_{R})}(\alpha-\beta)
=\widehat H^{(R)}(\alpha-\beta)
=\bigl[T_{\bn}^{(R)}(H)\bigr]_{\alpha,\beta},
\]
proving \eqref{eq:Toeplitz-reduction}. Taking $S_{L}=[d]$ gives \eqref{eq:left-right-reduction}.

For the fixed-size criterion, $T_{\bn}^{(L)}(F)=T_{\bn}^{(R)}(F)$ iff
$\widehat F^{(L)}(\bm k)=\widehat F^{(R)}(\bm k)$ for all difference indices
$\bm k=\alpha-\beta$ that occur in $T_{\bn}$, i.e., for all $\bm k\in\Delta(\bn)$; namely
\[
\widehat Z(\bm k)+\widehat W(\bm k)\J=\widehat Z(\bm k)+\widehat W(-\bm k)\J
\quad\text{for all } \bm k\in\Delta(\bn),
\]
which is equivalent to \eqref{eq:evenness-criterion-box}. \val{The equality
$T_{\bn}^{(L)}(F)=T_{\bn}^{(R)}(F)$ for every $\bn\in\NN^d$ holds if and only if
the above identity holds for all $\bm k\in\ZZ^d$,}
yielding \eqref{eq:evenness-criterion-global}, which is equivalent to $W(\btheta)=W(-\btheta)$ a.e.
\end{proof}
\end{MovedProofContext}

\phantomsection\label{proof:p07}
\begingroup\hypersetup{linkcolor=red}\noindent\val{\textbf{Proof of Lemma~\ref{lem:Fourier-linearity}}}\par\smallskip\endgroup
\begin{MovedProofContext}{p07}
\begin{proof}
Since the kernels
$e^{-\I\langle\bm k,\boldsymbol\theta\rangle_{S_{L}}}$ and
$e^{-\I\langle\bm k,\boldsymbol\theta\rangle_{S_{R}}}$ take values in $\CC_{\I}$,
they commute with $a,b$.
Using only associativity in $\HH$, commutation with $\CC_{\I}$-valued kernels, and the
$\mathbb{R}$-linearity of the Lebesgue integral, we get
\[
\begin{aligned}
\widehat{\,aF+bG\,}^{(S_{L},S_{R})}(\bm k)
&=(2\pi)^{-d}\!\int
e^{-\I\langle\bm k,\boldsymbol\theta\rangle_{S_{L}}}
\,(aF+bG)\,
e^{-\I\langle\bm k,\boldsymbol\theta\rangle_{S_{R}}} \, d\boldsymbol\theta \\
&=(2\pi)^{-d}\!\int
\bigl(a\,e^{-\I\langle\bm k,\boldsymbol\theta\rangle_{S_{L}}}F e^{-\I\langle\bm k,\boldsymbol\theta\rangle_{S_{R}}}
+ b\,e^{-\I\langle\bm k,\boldsymbol\theta\rangle_{S_{L}}}G e^{-\I\langle\bm k,\boldsymbol\theta\rangle_{S_{R}}}\bigr)
\, d\boldsymbol\theta \\
&= a\,\widehat{F}^{(S_{L},S_{R})}(\bm k)
+ b\,\widehat{G}^{(S_{L},S_{R})}(\bm k),
\end{aligned}
\]
which proves \eqref{eq:Fourier-CI-linear}.
\end{proof}
\end{MovedProofContext}

\phantomsection\label{proof:p08}
\begingroup\hypersetup{linkcolor=red}\noindent\val{\textbf{Proof of Proposition~\ref{prop:Toeplitz-linear}}}\par\smallskip\endgroup
\begin{MovedProofContext}{p08}
\begin{proof}
By definition, each block entry of $T_{\bn}^{(S_{L},S_{R})}(H)$ equals
$\widehat{H}^{(S_{L},S_{R})}(\bm k)$ with $\bm k=\alpha-\beta$.
Applying Lemma \ref{lem:Fourier-linearity} entrywise with
$H=aF+bG$ gives \eqref{eq:Toeplitz-CI-linear}.
\end{proof}
\end{MovedProofContext}

\phantomsection\label{proof:p09}
\begingroup\hypersetup{linkcolor=red}\noindent\val{\textbf{Proof of Theorem~\ref{thm:AdjointFormal-mv}}}\par\smallskip\endgroup
\begin{MovedProofContext}{p09}
\begin{proof}
By the $*$-closure of integrals,
\[
\bigl(\widehat F^{(L)}(-\bm k)\bigr)^{*}
=\frac{1}{(2\pi)^{d}}\!\int_{\TT^{d}}\bigl(F(\btheta)\bigr)^{*}\,e^{-\I\langle \bm k,\btheta \rangle}\,d\btheta
=\widehat{(F^{*})}^{(R)}(\bm k),
\]
and similarly
\[
\bigl(\widehat F^{(R)}(-\bm k)\bigr)^{*}
=
\widehat{(F^{*})}^{(L)}(\bm k).
\]
\val{For the two-sided case,}
\[
\val{
\bigl(\widehat F^{(S_L,S_R)}(-\bm k)\bigr)^{*}
=
\widehat{(F^{*})}^{(S_R,S_L)}(\bm k).}
\]
Hence, for $\alpha,\beta\in\Lambda_{\bn}$,
\[
\bigl(T_{\boldsymbol{n}}^{(\tau)}(F)\bigr)^{*}_{\alpha\beta}
=
\bigl(\widehat F^{(\tau)}(\beta-\alpha)\bigr)^{*}
=
\widehat{(F^{*})}^{(\tau')}(\alpha-\beta)
=
\bigl(T_{\bn}^{(\tau')}(F^{*})\bigr)_{\alpha\beta},
\]
where
\[
\val{
(\tau,\tau')
\in
\left\{
(L,R),\,
(R,L),\,
\bigl((S_L,S_R),(S_R,S_L)\bigr)
\right\}.}
\]
This proves the claim.
\qedhere
\end{proof}
\end{MovedProofContext}

\phantomsection\label{proof:p10}
\begingroup\hypersetup{linkcolor=red}\noindent\val{\textbf{Proof of Theorem~\ref{thm:selfadjoint-single-axis-mv}}}\par\smallskip\endgroup
\begin{MovedProofContext}{p10}
\begin{proof}
For a fixed $\bn$, stating that $T_{\bn}^{(S_L,S_R)}(F)$ is Hermitian is equivalent to
\begin{equation}\label{eq:HermCrit-general}
  \widehat F^{(S_L,S_R)}(\bm k)=\bigl(\widehat F^{(S_L,S_R)}(-\bm k)\bigr)^{*}
  \qquad\text{for all }\bm k\in\Delta(\bn),
\end{equation}
where $\Delta(\bn)=\{\alpha-\beta:\ \alpha,\beta\in\Lambda_{\bn}\}$.
Since the statement holds for all $\bn$, \eqref{eq:HermCrit-general} is equivalent to the same identity
for all $\bm k\in\ZZ^{d}$.

\val{Let $\bm k_L$ and $\bm k_R$ be the coordinate projections defined before
\eqref{eq:mv-Sandwich-coeffs}.} By \eqref{eq:mv-Sandwich-coeffs},
\[
\widehat F^{(S_L,S_R)}(\bm k)
=\widehat Z(\bm k_L+\bm k_R)+\widehat W(\bm k_L-\bm k_R)\,\J.
\]
Set
\[
\bm p:=\bm k_L+\bm k_R=\bm k,
\qquad
\bm q:=\bm k_L-\bm k_R.
\]
\val{The maps
\(\bm k\mapsto\bm p\) and
\(\bm k\mapsto\bm q\)
are bijections of \(\ZZ^d\), where the latter only changes coordinate signs.
For each fixed \(\bm k\), equation~\eqref{eq:HermCrit-general} becomes}
\[
\widehat Z(\bm p)+\widehat W(\bm q)\,\J
=\Bigl(\widehat Z(-\bm p)+\widehat W(-\bm q)\,\J\Bigr)^{*}
=\widehat Z(-\bm p)^{*}+(\widehat W(-\bm q)\,\J)^{*}.
\]
Using $\J^{*}=-\J$ and $\J A=\overline{A}\,\J$ for $A\in\CC_{\I}^{\,s\times s}$, together with
$\overline{A^{*}}=A^{\mathsf T}$, we obtain
\[
(\widehat W(-\bm q)\,\J)^{*}
=\J^{*}\widehat W(-\bm q)^{*}
=-\J\,\widehat W(-\bm q)^{*}
=-\,\overline{\widehat W(-\bm q)^{*}}\,\J
=-\,\bigl(\widehat W(-\bm q)\bigr)^{\mathsf T}\,\J.
\]
Hence
\[
\widehat Z(\bm p)+\widehat W(\bm q)\,\J
=\widehat Z(-\bm p)^{*}-\bigl(\widehat W(-\bm q)\bigr)^{\mathsf T}\,\J.
\]
Since \(1\) and \(\J\) are linearly independent over the slice
\(\CC_{\I}\), \val{and since both maps above are bijective,}
this is equivalent to
\begin{equation}\label{eq:Fourier-conditions}
\widehat Z(\bm p)=\widehat Z(-\bm p)^{*}
\qquad\text{and}\qquad
\widehat W(\bm q)=-\,\bigl(\widehat W(-\bm q)\bigr)^{\mathsf T}
\quad\text{for all }\bm p,\bm q\in\ZZ^{d}.
\end{equation}
By standard properties of complex Fourier series, \eqref{eq:Fourier-conditions} is equivalent to
$Z(\btheta)^{*}=Z(\btheta)$ and $W(-\btheta)=-\,W(\btheta)^{\mathsf T}$ a.e.\ on $\TT^{d}$,
which proves (i)$\Rightarrow$(ii). Conversely, if (ii) holds, then \eqref{eq:Fourier-conditions} holds
and thus \eqref{eq:HermCrit-general} follows for every $\bn$, so
$T_{\bn}^{(S_L,S_R)}(F)$ is Hermitian for all $\bn$.

\val{Because both maps above are bijections for every ordered partition,
the conclusion does not depend on the partition.}
\end{proof}
\end{MovedProofContext}
\subsection*{\val{Embeddings and distribution theorems}}

\phantomsection\label{proof:p11}
\begingroup\hypersetup{linkcolor=red}\noindent\val{\textbf{Proof of Theorem~\ref{thm:unified-LR-rect}}}\par\smallskip\endgroup
\begin{MovedProofContext}{p11}
\begin{proof}
Let $\bm k=\alpha-\beta\in\ZZ^{d}$. Using \eqref{eq:Phi-cov-used} with the left multiplier
$\lambda(\btheta)=e^{-\I\langle\bm k,\btheta\rangle}I_{s}$ for the left case (respectively the right multiplier $\rho(\btheta)=e^{-\I\langle\bm k,\btheta\rangle}I_{t}$ for the right case), we obtain
\[
\PhiSymp\!\bigl(\widehat F^{(\mathrm L)}(\bm k)\bigr)
=\frac{1}{(2\pi)^{d}}\!\int_{\TT^{d}}\!
\operatorname{diag}\!\bigl(e^{-\I\langle\bm k,\btheta\rangle}I_{s},\,e^{+\I\langle\bm k,\btheta\rangle}I_{s}\bigr)\,
\PhiSymp\!\bigl(F(\btheta)\bigr)\,d\btheta,
\]
and
\[
\PhiSymp\!\bigl(\widehat F^{(\mathrm R)}(\bm k)\bigr)
=\frac{1}{(2\pi)^{d}}\!\int_{\TT^{d}}\!
\PhiSymp\!\bigl(F(\btheta)\bigr)\,
\operatorname{diag}\!\bigl(e^{-\I\langle\bm k,\btheta\rangle}I_{t},\,e^{+\I\langle\bm k,\btheta\rangle}I_{t}\bigr)\,d\btheta.
\]
Since $\PhiSymp(F)=\bigl[\begin{smallmatrix} Z & W\\ -\overline{W} & \overline{Z} \end{smallmatrix}\bigr]$, a direct multiplication of the integrands yields the phase factors
$Z\,e^{-\I\langle\bm k,\btheta\rangle}$ and $W\,e^{-\I\langle\bm k,\btheta\rangle}$ in the left case, and
$Z\,e^{-\I\langle\bm k,\btheta\rangle}$ and $W\,e^{+\I\langle\bm k,\btheta\rangle}$ in the right case; the lower row carries the corresponding conjugate phases. Therefore
\[
\PhiSymp\!\bigl(\widehat F^{(\mathrm L)}(\bm k)\bigr)=
\begin{pmatrix}
\widehat Z(\bm k) & \widehat W(\bm k)\\[2pt]
-\overline{\widehat W(\bm k)} & \overline{\widehat Z(\bm k)}
\end{pmatrix},
\qquad
\PhiSymp\!\bigl(\widehat F^{(\mathrm R)}(\bm k)\bigr)=
\begin{pmatrix}
\widehat Z(\bm k) & \widehat W(-\bm k)\\[2pt]
-\overline{\widehat W(-\bm k)} & \overline{\widehat Z(\bm k)}
\end{pmatrix}.
\]
These are precisely the $\bm k$-th Fourier coefficients of $G_{\mathrm L}$ and $G_{\mathrm R}$, respectively. Setting $\bm k=\alpha-\beta$ gives the Toeplitz identities.
\end{proof}
\end{MovedProofContext}

\phantomsection\label{proof:p12}
\begingroup\hypersetup{linkcolor=red}\noindent\val{\textbf{Proof of Theorem~\ref{thm:sandwich+perm-rect-r}}}\par\smallskip\endgroup
\begin{MovedProofContext}{p12}
\begin{proof}
Let $\bm k=\alpha-\beta\in\ZZ^{d}$. \val{Let $\bm k_L$ and $\bm k_R$ be the
coordinate projections defined before \eqref{eq:mv-Sandwich-coeffs}.}
By the definition of the sandwich Fourier coefficients,
\[
\widehat F^{(S_{L},S_{R})}(\bm k)
=\frac{1}{(2\pi)^{d}}\!\int_{\TT^{d}} e^{-\I\langle \bm k_{L},\btheta\rangle}\,(Z+W\J)\,e^{-\I\langle \bm k_{R},\btheta\rangle}\,d\btheta.
\]
\val{Applying the covariance identity \eqref{eq:Phi-cov-used} with}
\[
\lambda(\btheta)=e^{-\I\langle \bm k_{L},\btheta\rangle}I_s,
\qquad
\rho(\btheta)=e^{-\I\langle \bm k_{R},\btheta\rangle}I_t,
\]
we obtain
\[
\PhiSymp\!\bigl(\widehat F^{(S_{L},S_{R})}(\bm k)\bigr)
=\frac{1}{(2\pi)^{d}}\!\int_{\TT^{d}}
\begin{bmatrix}
e^{-\I\langle \bm k_{L},\btheta\rangle} I_{s} & 0\\[1pt] 0 & e^{+\I\langle \bm k_{L},\btheta\rangle} I_{s}
\end{bmatrix}
\begin{bmatrix}
Z & W\\[1pt] -\overline W & \overline Z
\end{bmatrix}
\begin{bmatrix}
e^{-\I\langle \bm k_{R},\btheta\rangle} I_{t} & 0\\[1pt] 0 & e^{+\I\langle \bm k_{R},\btheta\rangle} I_{t}
\end{bmatrix}
\,d\btheta.
\]
Multiplying the integrand, we find the four entries
\[
Z\,e^{-\I\langle \bm k_{L}+\bm k_{R},\btheta\rangle},\quad
W\,e^{-\I\langle \bm k_{L}-\bm k_{R},\btheta\rangle},\quad
-\overline W\,e^{+\I\langle \bm k_{L}-\bm k_{R},\btheta\rangle},\quad
\overline Z\,e^{+\I\langle \bm k_{L}+\bm k_{R},\btheta\rangle}.
\]
Hence
\[
\PhiSymp\!\bigl(\widehat F^{(S_{L},S_{R})}(\bm k)\bigr)=
\begin{pmatrix}
\widehat Z(\bm k_{L}+\bm k_{R}) & \widehat W(\bm k_{L}-\bm k_{R})\\[2pt]
-\overline{\widehat W(\bm k_{L}-\bm k_{R})} & \overline{\widehat Z(\bm k_{L}+\bm k_{R})}
\end{pmatrix}.
\]
By standard changes of variables on $\TT^{d}$, we have
\[
\widehat{H\circ(-\mathrm{id})}(\bm k)=\widehat H(-\bm k),
\qquad
\widehat{H\circ r_{S_L}}(\bm k)=\widehat H\!\bigl(r_{S_L}(\bm k)\bigr),
\]
and since $r_{S_L}(-\bm k)=(\bm k_{L},-\bm k_{R})$, we deduce
\[
\widehat{W\circ r_{S_L}\circ(-\mathrm{id})}(\bm k)
=\widehat W\!\bigl(r_{S_L}(-\bm k)\bigr)
=\widehat W(\bm k_{L}-\bm k_{R}),
\qquad
\widehat{\overline{Z(-\cdot)}}(\bm k)=\overline{\widehat Z(\bm k)}.
\]
Therefore the last displayed matrix equals $\widehat{G_{S_L,S_R}}(\bm k)$.
Setting $\bm k=\alpha-\beta$ and assembling by blocks yields
\(
\PhiSymp\!\bigl(T_{\bn}^{(S_L,S_R)}(F)\bigr)=T_{\bn}\!\bigl(G_{S_L,S_R}\bigr),
\)
which is \eqref{eq:sandwich-embed-r}.
\end{proof}
\end{MovedProofContext}

\phantomsection\label{proof:p13}
\begingroup\hypersetup{linkcolor=red}\noindent\val{\textbf{Proof of Theorem~\ref{thm:SV-qToeplitz-1axis-embed-mv}}}\par\smallskip\endgroup
\begin{MovedProofContext}{p13}
\begin{proof}
By the embedding identities established above,
\[
\PhiSymp\!\bigl(T_{\bn}^{(L)}(F)\bigr)=T_{\bn}(G_{\mathrm L}),\qquad
\PhiSymp\!\bigl(T_{\bn}^{(R)}(F)\bigr)=T_{\bn}(G_{\mathrm R}),\qquad
\PhiSymp\!\bigl(T_{\bn}^{(S_{L},S_{R})}(F)\bigr)=T_{\bn}(G_{S_{L},S_{R}}).
\]
Hence it suffices to analyze $T_{\bn}(G_{\tau})$ on the complex side.
\val{Theorem~\ref{thm:complex-toeplitz-sv} gives
$\{T_{\bn}(G_{\tau})\}_{\bn}\sim_{\sigma}G_{\tau}$.}
By Corollary~\ref{cor:s_p}, the singular values of $\PhiSymp(A)$ are those of $A$,
each duplicated, so the quaternion-side normalization matches the complex one,
giving \eqref{eq:sv-limit-mv}.

If $s=t$ and every $T_{\bn}^{(\tau)}(F)$ is Hermitian, then
$\PhiSymp\!\bigl(T_{\bn}^{(\tau)}(F)\bigr)$ is Hermitian as well. \val{Theorem~\ref{thm:complex-toeplitz-hermitian} gives
$\{T_{\bn}(G_{\tau})\}_{\bn}\sim_{\lambda}G_{\tau}$.} The embedding argument
transfers the claim back to the quaternion sequence.
\end{proof}
\end{MovedProofContext}

\phantomsection\label{proof:p14}
\begingroup\hypersetup{linkcolor=red}\noindent\val{\textbf{Proof of Theorem~\ref{thm:EV-qToeplitz-1axis-embed-mv}}}\par\smallskip\endgroup
\begin{MovedProofContext}{p14}
\begin{proof}
By the embedding identities,
$\PhiSymp\!\bigl(T_{\bn}^{(\tau)}(F)\bigr)=T_{\bn}(G_\tau)$.
Since $G_\tau\in L^\infty(\TT^d;\CC^{(2s)\times(2s)})$ and it belongs to the complex Tilli class,
\val{Theorem~\ref{thm:complex-tilli} leads to
$\{T_{\bn}(G_\tau)\}_{\bn}\sim_{\lambda}G_\tau$.}
The eigenvalue distribution for the quaternion sequence follows from the embedding.
\end{proof}
\end{MovedProofContext}
\subsection*{\val{Schatten estimates and localization}}

\phantomsection\label{proof:p15}
\begingroup\hypersetup{linkcolor=red}\noindent\val{\textbf{Proof of Theorem~\ref{thm:q-schatten-coarse-mv}}}\par\smallskip\endgroup
\begin{MovedProofContext}{p15}
{
\begin{proof}
We consider a general sandwich-kernel setting, thus consider the embedding symbol
\[
G_{S_{L},S_{R}}(\btheta)
=
\begin{pmatrix}
Z(\btheta) & W\!\bigl(r_{S_{L}}(-\btheta)\bigr)\\[2pt]
-\overline{W\!\bigl(r_{S_{L}}(\btheta)\bigr)} & \overline{Z(-\btheta)}
\end{pmatrix},
\]
where $r_{S_{L}} : \TT^{d} \to \TT^{d}$ is the reflection of left-kernel variables.
Define the row projections
\[
P_{1} := \operatorname{diag}(I_{s},0) \in \CC^{(2s)\times(2s)},
\qquad
P_{2} := \operatorname{diag}(0,I_{s}),
\]
and the respective column projections
\[
E_{1} := \operatorname{diag}(I_{t},0) \in \CC^{(2t)\times(2t)},
\qquad
E_{2} := \operatorname{diag}(0,I_{t}).
\]
Clearly, $\|P_{j}\|_{\infty} = \|E_{j}\|_{\infty} = 1$ for $j=1,2$. \\
Using the projections, we decompose $G_{S_{L},S_{R}}(\btheta)$ into the sum of four terms:
\[
\begin{aligned}\label{eq:symbdec}
G_{S_{L},S_{R}}(\btheta)
&=
P_{1}\,\PhiSymp\!\bigl(F(\btheta)\bigr)\,E_{1}
+ P_{1}\,\PhiSymp\!\bigl(F(r_{S_{L}}(-\btheta))\bigr)\,E_{2}\\
&\quad
+ P_{2}\,\PhiSymp\!\bigl(F(r_{S_{L}}(\btheta))\bigr)\,E_{1}
+ P_{2}\,\PhiSymp\!\bigl(F(-\btheta)\bigr)\,E_{2},
\end{aligned}
\]
where, \val{explicitly},
\[
P_{1}\,\PhiSymp\!\bigl(F(\btheta)\bigr)E_{1}
=
\begin{pmatrix}
Z(\btheta) & 0\\[2pt]
0 & 0
\end{pmatrix},
\quad
P_{1}\,\PhiSymp\!\bigl(F(r_{S_{L}}(-\btheta))\bigr)E_{2}
=
\begin{pmatrix}
0 & W(r_{S_{L}}(-\btheta))\\[2pt]
0 & 0
\end{pmatrix},
\]
\[
P_{2}\,\PhiSymp\!\bigl(F(r_{S_{L}}(\btheta))\bigr)E_{1}
=
\begin{pmatrix}
0 & 0\\[2pt]
-\overline{W(r_{S_{L}}(\btheta))} & 0
\end{pmatrix},
\quad
P_{2}\,\PhiSymp\!\bigl(F(-\btheta)\bigr)E_{2}
=
\begin{pmatrix}
0 & 0\\[2pt]
0 & \overline{Z(-\btheta)}
\end{pmatrix}.
\]
Now for any choice of possible contractions, we have
\[
\|A X B\|_{p} \le \|A\|\,\|X\|_{p}\,\|B\|\le \|X\|_{p}.
\]
Hence, by using \eqref{eq:symbdec}, for any $\btheta$ we deduce
\[
\begin{aligned}
\|G_{S_L,S_R}(\btheta)\|_{p}
&\le \sum_{j=1}^{4}
\bigl\|\PhiSymp\!\bigl(F(\varphi_{j}(\btheta))\bigr)\bigr\|_{p},
\end{aligned}
\]
where $\varphi_j(\btheta)$ is either $\btheta \mapsto \btheta$, $\btheta \mapsto -\btheta$, or $\btheta \mapsto r_{S_L}(\btheta)$. Note that all of these functions \val{are measure-preserving involutions of $\TT^d$}. Hence, passing to the $L^p$ norm, we have
\[
\bigl\|\PhiSymp \circ F \circ \varphi_{j}\bigr\|_{L^{p}}
=
\|\PhiSymp \circ F\|_{L^{p}}.
\]
Summing all the terms we obtain
\[
\|G_{\tau}\|_{L^{p}(\TT^{d})}
\le 4\,\|\PhiSymp \circ F\|_{L^{p}(\TT^{d})}
= 4\cdot 2^{1/p}\,\|F\|_{L^{p}(\TT^{d};\HH^{s\times t})},
\]
by the symplectic $L^{p}$ isometry. \\
Furthermore, we consider the corresponding Toeplitz matrix $T_{\bn}^{(S_L,S_R)}(F)$. By assumption, we have
\[
\PhiSymp\!\bigl(T_{\bn}^{(S_L,S_R)}(F)\bigr)
= T_{\bn}(G_{\tau}).
\]
By the Schatten scaling under $\PhiSymp$,
\[
\bigl\|T_{\bn}^{(\tau)}(F)\bigr\|_{p}
= 2^{-1/p} \bigl\|T_{\bn}(G_{\tau})\bigr\|_{p}.
\]
Further, applying the complex Toeplitz bound (Lemma~\ref{lem:Tilli-Sp-mv}) and the
$L^{p}$ bound for $G_{\tau}$, we obtain
\[
\begin{aligned}
\bigl\|T_{\bn}^{(\tau)}(F)\bigr\|_{p}
&\le 2^{-1/p}
\left(\frac{N_{\bn}}{(2\pi)^{d}}\right)^{1/p}
\|G_{\tau}\|_{L^{p}(\TT^{d})}\\[2pt]
&\le 2^{-1/p}
\left(\frac{N_{\bn}}{(2\pi)^{d}}\right)^{1/p}
\cdot 4\cdot 2^{1/p}\,\|F\|_{L^{p}(\TT^{d};\HH^{s\times t})}\\[2pt]
&= 4\left(\frac{N_{\bn}}{(2\pi)^{d}}\right)^{1/p}
\|F\|_{L^{p}(\TT^{d};\HH^{s\times t})},
\end{aligned}
\]
which is exactly the claim we want to prove.
\end{proof}
}
\end{MovedProofContext}

\phantomsection\label{proof:p16}
\begingroup\hypersetup{linkcolor=red}\noindent\val{\textbf{Proof of Theorem~\ref{thm:embed_bounds-mv}}}\par\smallskip\endgroup
\begin{MovedProofContext}{p16}
\begin{proof}
By the embedding identities, we find
\[
\PhiSymp\!\bigl(T_{\bn}^{(\tau)}(F)\bigr)=T_{\bn}(G_{\tau}).
\]
Since $T_{\bn}^{(\tau)}(F)$ is Hermitian, $\PhiSymp\!\bigl(T_{\bn}^{(\tau)}(F)\bigr)$
is Hermitian and hence $T_{\bn}(G_{\tau})$ is Hermitian as well.
\val{Because this holds for every $\bn\in\NN^d$, every $\bm k\in\ZZ^d$ occurs as
a Toeplitz difference index for some $\bn$, and therefore
$\widehat G_{\tau}(\bm k)=\widehat G_{\tau}(-\bm k)^*$ for every $\bm k$.
Uniqueness of Fourier coefficients gives
$G_{\tau}(\btheta)=G_{\tau}(\btheta)^*$ for a.e. $\btheta\in\TT^d$. Hence
Theorem~\ref{thm:complex-toeplitz-localization} gives}
\[
\sigma\!\bigl(T_{\bn}(G_{\tau})\bigr)\ \subset\
[\,m_{\tau},\,M_{\tau}\,].
\]
For Hermitian matrices, $\PhiSymp$ preserves eigenvalues (it duplicates them in the
complex representation but with the same values); thus \eqref{eq:hermitian_loc} follows.
\end{proof}
\end{MovedProofContext}

\phantomsection\label{proof:p17}
\begingroup\hypersetup{linkcolor=red}\noindent\val{\textbf{Proof of Corollary~\ref{cor:HPD}}}\par\smallskip\endgroup
\begin{MovedProofContext}{p17}
\begin{proof}
(a) Apply Theorem~\ref{thm:embed_bounds-mv}. \val{For (b), the Hermitian matrices
$T_{\bn}(G_{\tau})=\PhiSymp\!\bigl(T_{\bn}^{(\tau)}(F)\bigr)$ have all their real
eigenvalues at least $c$. The eigenvalue distribution in
Theorem~\ref{thm:complex-toeplitz-hermitian} therefore implies
$\lambda_{\min}\!\bigl(G_{\tau}(\boldsymbol\theta)\bigr)\ge c$ in measure; hence
$m_{\tau}\ge c$.}
\end{proof}
\end{MovedProofContext}
\subsection*{\val{Circulants and a.c.s. approximation}}

\phantomsection\label{proof:p18}
\begingroup\hypersetup{linkcolor=red}\noindent\val{\textbf{Proof of Proposition~\ref{prop:canon-X}}}\par\smallskip\endgroup
\begin{MovedProofContext}{p18}
\begin{proof}
Write each coefficient of $p_{\mathrm R}$ in Cartesian form
\[
p_{\mathrm R}(\brho)=p_{\mathrm R}(\brho)^{(0)}+p_{\mathrm R}(\brho)^{(1)}\I+p_{\mathrm R}(\brho)^{(2)}\J+p_{\mathrm R}(\brho)^{(3)}(\I\J),
\]
with real blocks $p_{\mathrm R}(\brho)^{(\ell)}\in\RR^{s\times t}$. Expanding the shift representation,
\[
\begin{aligned}
C_{\bm n}\bigl(p_{\mathrm{R}}\bigr)
&=\sum_{\brho\in \mathcal F}\bigl(P_{\bm n}^{(\brho)}\otimes p_{\mathrm R}(\brho)^{(0)}\bigr)
+\Bigl(\sum_{\brho\in \mathcal F}P_{\bm n}^{(\brho)}\otimes p_{\mathrm R}(\brho)^{(1)}\Bigr)\,\I \\
&\qquad+\Bigl(\sum_{\brho\in \mathcal F}P_{\bm n}^{(\brho)}\otimes p_{\mathrm R}(\brho)^{(2)}\Bigr)\,\J
+\Bigl(\sum_{\brho\in \mathcal F}P_{\bm n}^{(\brho)}\otimes p_{\mathrm R}(\brho)^{(3)}\Bigr)\,(\I\J).
\end{aligned}
\]
For $\ell=0,1$, \val{Proposition~\ref{prop:complex-bcirc-diagonalization}} applied componentwise gives
\[
U_{L}\Bigl(\sum_{\brho}P_{\bm n}^{(\brho)}\otimes p_{\mathrm R}(\brho)^{(\ell)}\Bigr)U_{R}^{*}
=\Bdiag\bigl(\widehat S_{\ell}(\bk)\bigr)_{\bk\in\Lambda_{\bm n}}=:\Lambda_{\ell}.
\]
For the $\J$-component we use that $\J\alpha=\overline{\alpha}\,\J$ for $\alpha\in\CC_{\I}$ and the QDFT flip
\val{$F_{\I}^{(\bm n)}(\J I_{N_{\bn}})(F_{\I}^{(\bm n)})^{*}
=A_{\bm n}(\J I_{N_{\bn}})$} (Lemma~\ref{lem:flip-detailed}). As a consequence
\[
\begin{aligned}
&\ \val{U_{L}\Bigl(\sum_{\brho}P_{\bm n}^{(\brho)}\otimes p_{\mathrm R}(\brho)^{(2)}\Bigr)
(\J I_{N_{\bn}t})\,U_{R}^{*}} \\
&=\bigl[U_{L}\Bigl(\sum_{\brho}P_{\bm n}^{(\brho)}\otimes p_{\mathrm R}(\brho)^{(2)}\Bigr)U_{R}^{*}\bigr]\cdot
\val{\bigl(U_{R}(\J I_{N_{\bn}t})U_{R}^{*}\bigr)} \\
&=\val{\Lambda_{2}\cdot\bigl((F_{\I}^{(\bm n)}\otimes I_{t})
((\J I_{N_{\bn}})\otimes I_t)(F_{\I}^{(\bm n)}\otimes I_{t})^{*}\bigr)} \\
&=\val{\Lambda_{2}\cdot\bigl(
(F_{\I}^{(\bm n)}(\J I_{N_{\bn}})(F_{\I}^{(\bm n)})^{*})\otimes I_{t}\bigr)
=\Lambda_{2}\,(A_{\bm n}\otimes I_{t})(\J I_{N_{\bn}t}).}
\end{aligned}
\]
The case $\ell=3$ is identical, and this gives us \eqref{eq:canonical-form}.
Consider now the permutations $P$ of Definition~\ref{def:pair} and set
$\Pi_{L}:=P\otimes I_{s}$ and $\Pi_{R}:=P\otimes I_{t}$. Define
\[
\widetilde\Lambda_{\ell}:=\Pi_{L}\Lambda_{\ell}\Pi_{R}^{*}\qquad(\ell=0,1,2,3).
\]
\val{Applying $\Pi_L$ to the block rows and $\Pi_R^*$ to the block columns in
\eqref{eq:canonical-form}, we obtain}
\[
\Pi_{L}U_{L}\,C_{\bm n}(p_{\mathrm R})\,U_{R}^{*}\Pi_{R}^{*}
=\widetilde\Lambda_{0}+\widetilde\Lambda_{1}\,\I
+\Pi_{L}\Lambda_{2}(A_{\bm n}\otimes I_{t})\Pi_{R}^{*}\,\J
+\Pi_{L}\Lambda_{3}(A_{\bm n}\otimes I_{t})\Pi_{R}^{*}\,(\I\J).
\]
Now, insert the identity $\Pi_{R}^{*}\Pi_{R}=I$ between $\Lambda_{2}$ and $(A_{\bm n}\otimes I_{t})$ so that
\[
\Pi_{L}\Lambda_{2}(A_{\bm n}\otimes I_{t})\Pi_{R}^{*}
=(\Pi_{L}\Lambda_{2}\Pi_{R}^{*})\cdot\bigl(\Pi_{R}(A_{\bm n}\otimes I_{t})\Pi_{R}^{*}\bigr),
\]
and in a similar way, we do the same for $\Lambda_{3}$. By using the identity in block form of Lemma~\ref{lem:PAP}, we have
\[
Z:=\Pi_{R}(A_{\bm n}\otimes I_{t})\Pi_{R}^{*}
=(P\otimes I_{t})(A_{\bm n}\otimes I_{t})(P\otimes I_{t})^{*}
=\val{I_{\#\mathrm{Fix}}}\otimes I_{t}\ \oplus\ \bigoplus_{\bk\in\mathcal K}\begin{bmatrix}0&I_{t}\\[2pt]I_{t}&0\end{bmatrix}.
\]
Therefore, we infer
\begin{equation}\label{eq:diag}
\Pi_{L}U_{L}\,C_{\bm n}(p_{\mathrm R})\,U_{R}^{*}\Pi_{R}^{*}
=\widetilde\Lambda_{0}+\widetilde\Lambda_{1}\,\I
+\bigl(\widetilde\Lambda_{2}Z\bigr)\,\J
+\bigl(\widetilde\Lambda_{3}Z\bigr)\,(\I\J).
\end{equation}
The relation obtained in \eqref{eq:diag} reads as follows.
\begin{itemize}
\item If $\bk\in\mathrm{Fix}$, we have $Z=I_{t}$, and the block diagonals we obtain are equal to
\[
D_{1}(\bk){+ D_{2}(\bk)}:=\widehat S_{0}(\bk)+\widehat S_{1}(\bk)\,\I + {\bigl(\widehat S_{2}(\bk)+\widehat S_{3}(\bk)\,\I\bigr)\,\J}.
\]
\item \val{For each $\bk\in\mathcal K$}, the factor $Z$ is the $t$-block $2\times 2$ exchange matrix, and the corresponding blocks we obtain are
\[
\begin{bmatrix}
D_{1}(\bk) & D_{2}(\bk)\\[2pt]
D_{2}(-\bk) & D_{1}(-\bk)
\end{bmatrix},
\qquad
D_{2}(\bk):=\bigl(\widehat S_{2}(\bk)+\widehat S_{3}(\bk)\,\I\bigr)\,\J.
\]
\end{itemize}
Collecting all the sums in the block-diagonal order given by $P$ gives exactly our claim.
\end{proof}
\end{MovedProofContext}

\phantomsection\label{proof:p19}
\begingroup\hypersetup{linkcolor=red}\noindent\val{\textbf{Proof of Proposition~\ref{prop:qdftr-fibers-poly}}}\par\smallskip\endgroup
\begin{MovedProofContext}{p19}
\begin{proof}
First, we recall the notation of the canonical form in \eqref{eq:X-fiber}. Indeed, by Proposition~\ref{prop:canon-X}, we deduce the canonical QDFT form
\[
\Pi_{L}\,U_{L}\,C_{\bm n}\bigl(p_{\mathrm{R}}\bigr)\,U_{R}^{*}\,\Pi_{R}^{*}
=\bigoplus_{\bk\in\mathrm{Fix}}\!\Bigl(\widehat S_{0}(\bk)+\widehat S_{1}(\bk)\I+\bigl(\widehat S_{2}(\bk)+\widehat S_{3}(\bk)\I\bigr)\J\Bigr)
\]
\[
\oplus\
\val{\bigoplus_{\bk\in\mathcal K}}\!
\begin{bmatrix}
\widehat S_{0}(\bk)+\widehat S_{1}(\bk)\I & \bigl(\widehat S_{2}(\bk)+\widehat S_{3}(\bk)\I\bigr)\J\\[2pt]
\bigl(\widehat S_{2}(-\bk)+\widehat S_{3}(-\bk)\I\bigr)\J & \widehat S_{0}(-\bk)+\widehat S_{1}(-\bk)\I
\end{bmatrix}.
\]
Now, by the definitions of $Z$ and $W$ we have
\[
Z(\btheta_{\bk})=\widehat S_{0}(\bk)+\widehat S_{1}(\bk)\,\I,\qquad
W(-\btheta_{\bk})=\widehat S_{2}(\bk)+\widehat S_{3}(\bk)\,\I.
\]
Thus the fixed blocks are indeed $Z(\btheta_{\bk})+W(-\btheta_{\bk})\J$ and the paired blocks are
\[
\begin{bmatrix}
Z(\btheta_{\bk}) & W(-\btheta_{\bk})\J\\[2pt]
W(-\btheta_{-\bk})\J & Z(\btheta_{-\bk})
\end{bmatrix},
\]
which confirms \eqref{eq:X-fiber}.

For the complex embedding claim, we have to recall the basic identities for every $X\in\CC_{\I}^{s\times t}$
\[
\PhiSymp(X)=\begin{bmatrix}X&0\\[2pt]0&\overline{X}\end{bmatrix},
\qquad
\PhiSymp(X\J)=\begin{bmatrix}0&X\\[2pt]-\overline{X}&0\end{bmatrix},
\]
together with the $\RR$-linearity of $\PhiSymp$ and its blockwise action on real $2\times 2$ partitions.
For a fixed index $\bk\in\mathrm{Fix}$ we obtain
\[
\PhiSymp\bigl(Z(\btheta_{\bk})+\val{W(-\btheta_{\bk})}\J\bigr)
=\PhiSymp\bigl(Z(\btheta_{\bk})\bigr)+\PhiSymp\bigl(W(-\btheta_{\bk})\J\bigr)
=\begin{bmatrix}
Z(\btheta_{\bk}) & W(-\btheta_{\bk})\\[2pt]
-\overline{W(-\btheta_{\bk})} & \overline{Z(\btheta_{\bk})}
\end{bmatrix},
\]
which is \eqref{eq:Phi-fixed}. For the paired blocks, split them as sums of their diagonal and off-diagonal parts,
\[
\begin{bmatrix}
Z(\btheta_{\bk}) & W(-\btheta_{\bk})\J\\[2pt]
W(-\btheta_{-\bk})\J & Z(\btheta_{-\bk})
\end{bmatrix}
=
\begin{bmatrix}
Z(\btheta_{\bk}) & 0\\[2pt] 0 & Z(\btheta_{-\bk})
\end{bmatrix}
\;+\;
\begin{bmatrix}
0 & W(-\btheta_{\bk})\J\\[2pt] W(-\btheta_{-\bk})\J & 0
\end{bmatrix}.
\]
Applying $\PhiSymp$ to each part and summing, using the two identities above, yields exactly the $4\times 4$ complex block in \eqref{eq:Phi-paired}. This completes the proof.
\end{proof}
\end{MovedProofContext}

\phantomsection\label{proof:p20}
\begingroup\hypersetup{linkcolor=red}\noindent\val{\textbf{Proof of Corollary~\ref{cor:asymp-circ}}}\par\smallskip\endgroup
\begin{MovedProofContext}{p20}
\begin{proof}
Let $\btheta_{\bk}=2\pi(\bk/\bn)$ for $\bk\in\Lambda_{\bm n}$ and $\mathrm{Fix}=\{\bk:2\bk\equiv0\pmod{\bn}\}$ the usual fixed point set.
By Proposition~\ref{prop:qdftr-fibers-poly} (see \eqref{eq:X-fiber}), after the application of the QDFT and the permutations, we block-diagonalize $C_{\bm n}(p_{\mathrm R})$ as
\[
\bigl(Z(\btheta_{\bk})+W(-\btheta_{\bk})\J\bigr)\quad(\bk\in\mathrm{Fix}),\qquad
\begin{bmatrix}
Z(\btheta_{\bk}) & W(-\btheta_{\bk})\J\\[2pt]
W(-\btheta_{-\bk})\J & Z(\btheta_{-\bk})
\end{bmatrix}\quad(\{\bk,-\bk\}),
\]
and after the symplectic embedding  we have the $4\times4$ complex block in \eqref{eq:Phi-paired} for pairs and the $2\times2$ block in \eqref{eq:Phi-fixed} for fixed indices.
Now, if we apply for both rows and columns an $s \times t$ block permutation with the law $(1,4,3,2)$ to the $4\times4$ blocks \eqref{eq:Phi-paired}, we transform them into
\[
\diag\!\Bigl(\val{G_{\mathrm R}[p_{\mathrm R}](\btheta_{\bk})},\ \val{G_{\mathrm R}[p_{\mathrm R}](\btheta_{-\bk})}\Bigr),
\]
while for $\bk\in\mathrm{Fix}$ we have the periodic congruence $-\btheta_{\bk}\equiv\btheta_{\bk}\ (\mathrm{mod}\ 2\pi)$, so
\[
\PhiSymp\bigl(Z(\btheta_{\bk})+W(-\btheta_{\bk})\J\bigr)
=\begin{bmatrix}Z(\btheta_{\bk})&W(-\btheta_{\bk})\\[2pt]-\overline{W(-\btheta_{\bk})}&\overline{Z(\btheta_{\bk})}\end{bmatrix}
=\val{G_{\mathrm R}[p_{\mathrm R}](\btheta_{\bk})}.
\]
Hence, for every $\bn$, we deduce
\[
\sigma\!\bigl(\PhiSymp(C_{\bm n}(p_{\mathrm R}))\bigr)
=\bigsqcup_{\bk\in\Lambda_{\bm n}}\sigma\!\bigl(\val{G_{\mathrm R}[p_{\mathrm R}](\btheta_{\bk})}\bigr).
\]
If $s=t$, the left/right QDFT and the permutation used above coincide so the singular value result is obtained by unitary similarity, and thus we can give the same conclusion for eigenvalues:
\[
\lambda\!\bigl(\PhiSymp(C_{\bm n}(p_{\mathrm R}))\bigr)
=\bigsqcup_{\bk\in\Lambda_{\bm n}}\lambda\!\bigl(\val{G_{\mathrm R}[p_{\mathrm R}](\btheta_{\bk})}\bigr).
\]
As a consequence, note that for any continuous compactly supported test function $\psi$, the Riemann sums on $\{\btheta_{\bk}\}_{\bk\in\Lambda_{\bm n}}$ of the left side of \eqref{eq:def-ev-emb} and \eqref{eq:def-sv-emb} applied to $\{C_{\bm n}(p_{\mathrm R})\}_{\bm n}$ are a converging trapezoidal rule that yields
$\{C_{\bm n}(p_{\mathrm R})\}_{\bm n}\sim_{\sigma}\val{G_{\mathrm R}[p_{\mathrm R}]}$ and, when $s=t$, also $\{C_{\bm n}(p_{\mathrm R})\}_{\bm n}\sim_{\lambda}\val{G_{\mathrm R}[p_{\mathrm R}]}$.
\end{proof}
\end{MovedProofContext}

\phantomsection\label{proof:p21}
\begingroup\hypersetup{linkcolor=red}\noindent\val{\textbf{Proof of Theorem~\ref{thm:right-acs-circ}}}\par\smallskip\endgroup
\begin{MovedProofContext}{p21}
\begin{proof}
\val{By Proposition~\ref{prop:density-sandwich-functional}, right
trigonometric polynomials are dense in
$L^{1}(\TT^{d};\HH^{s\times t})$. Hence, there exist finite sets
$\mathcal S_m\subset\ZZ^{d}$ and coefficients
$A_{m,\bk}\in\HH^{s\times t}$ such that
\[
F_m(\btheta)\ =\ \sum_{\bk\in \mathcal S_m}
A_{m,\bk}\,e^{-\I\langle \bk,\btheta\rangle},
\qquad
\|F-F_m\|_{L^{1}(\TT^{d};\HH^{s\times t})}
\xrightarrow[m\to\infty]{}0.
\]}
\val{Fix $m$ and write
$\mathcal S_m\subset\prod_{\ell=1}^{d}
\{-r_{\ell}(m),\dots,r_{\ell}(m)\}$.
For all $\bm n$ with $n_{\ell}\ge r_{\ell}(m)+1$,
\[
T_{\bm n}^{(R)}(F_m)-B_{\bm n,m}
=
\sum_{\bk\in \mathcal S_m}
\bigl(J_{\bm n}^{(-\bk)}-P_{\bm n}^{(-\bk)}\bigr)
\otimes A_{m,\bk}.
\]
For each $\bk=(k_1,\dots,k_d)$, Lemma~\ref{lem:complex-shift-rank} gives
\[
\operatorname{rank}\!\bigl(
J_{\bm n}^{(-\bk)}-P_{\bm n}^{(-\bk)}
\bigr)
\le
\sum_{\ell=1}^{d}
2|k_{\ell}|\,\frac{N_{\bm n}}{n_{\ell}}.
\]
Using
$\operatorname{rank}(X\otimes Y)
\val{=}
\operatorname{rank}(X)\operatorname{rank}(Y)$
and
$\operatorname{rank}(A_{m,\bk})\le\min\{s,t\}$,
with
$r_{\bm n}:=N_{\bm n}\min\{s,t\}$, we obtain
\[
\frac{
\operatorname{rank}_{\HH}\!\Bigl(
T_{\bm n}^{(R)}(F_m)-B_{\bm n,m}
\Bigr)
}{
r_{\bm n}
}
\le
2\sum_{\ell=1}^{d}
\frac{1}{n_{\ell}}
\sum_{\bk\in\mathcal S_m}|k_{\ell}|
\xrightarrow[\bm n\to\infty]{}
0
\qquad
(\text{for each fixed }m).
\]}
\val{For every scalar $m\in\NN$, choose $\bn_m\in\NN^d$ componentwise large
enough that $n_\ell\ge r_\ell(m)+1$ for every $\ell$ and}
\[
\val{\frac{
\operatorname{rank}_{\HH}\!\bigl(
T_{\bn}^{(R)}(F_m)-B_{\bn,m}
\bigr)}{r_{\bn}}
\le \frac{1}{m}
\qquad\text{for every }\bn\ge\bn_m.}
\]
\val{Thus the circulant-versus-polynomial-Toeplitz difference satisfies the
rank bound required for a.c.s.\ convergence, with $c(m)=1/m\to0$.}
By Theorem~\ref{thm:q-schatten-coarse-mv} with $p=1$,
\[
\|T_{\bm n}^{(R)}(F)-T_{\bm n}^{(R)}(F_m)\|_{1}
\ \le\ {4}\,\frac{N_{\bm n}}{(2\pi)^{d}}\ \|F-F_m\|_{L^{1}(\TT^{d})}.
\]
Dividing by $r_{\bm n}=N_{\bm n}\min\{s,t\}$ gives
\[
\frac{\|T_{\bm n}^{(R)}(F)-T_{\bm n}^{(R)}(F_m)\|_{1}}{r_{\bm n}}
\ \le\ \frac{\val{4}}{(2\pi)^{d}\min\{s,t\}}\ \|F-F_m\|_{L^{1}(\TT^{d})}
\ =:\ \omega_{1}(m)\xrightarrow[m\to\infty]{}0.
\]
\val{Finally,
\[
T_{\bn}^{(R)}(F)-B_{\bn,m}
=
\bigl[
T_{\bn}^{(R)}(F)-T_{\bn}^{(R)}(F_m)
\bigr]
+
\bigl[
T_{\bn}^{(R)}(F_m)-B_{\bn,m}
\bigr].
\]
The first bracket satisfies the norm criterion
\(\mathrm{ACS}_{\HH}5\), while the second has normalized rank at most
\(1/m\). Hence
\[
\bigl\{\{B_{\bn,m}\}_{\bn}\bigr\}_{m\in\NN}
\]
is an a.c.s.\ for
\(\{T_{\bn}^{(R)}(F)\}_{\bn}\).}
\end{proof}
\end{MovedProofContext}

\phantomsection\label{proof:p22}
\begingroup\hypersetup{linkcolor=red}\noindent\val{\textbf{Alternative a.c.s.-based proof of the singular-value distribution in Theorem~\ref{thm:SV-qToeplitz-1axis-embed-mv}}}\par\smallskip\endgroup
\begin{MovedProofContext}{p22}
\begin{proof}[\val{a.c.s.-based proof of the singular-value distribution}]
\val{Set}
\[
\val{H(\btheta):=Z(\btheta)+\bigl(W\circ r_{S_L}\bigr)(\btheta)\J.}
\]
\val{By Proposition~\ref{prop:Toeplitz-right-reduction}, $H\in L^1$ and}
\[
\val{
T_{\bn}^{(S_L,S_R)}(F)=T_{\bn}^{(R)}(H)
\qquad\text{for every }\bn.}
\]
\val{Theorem~\ref{thm:right-acs-circ} provides an a.c.s.\ of circulants
from right trigonometric polynomials $H_m$ satisfying
\[
H_m\longrightarrow H
\qquad\text{in }L^1(\TT^d;\HH^{s\times t}).
\]
Set
\[
\widetilde H_m(\btheta):=H_m(-\btheta),
\qquad
\widetilde H(\btheta):=H(-\btheta).
\]
By construction,
\(
B_{\bn,m}=C_{\bn}(\widetilde H_m).
\)
Hence Corollary~\ref{cor:asymp-circ} gives
\[
\{B_{\bn,m}\}_{\bn}\sim_\sigma G_{\mathrm R}[\widetilde H_m].
\]
Reflection preserves the $L^1$ norm, and \eqref{eq:GR-bracket} yields
\[
G_{\mathrm R}[\widetilde H_m]
\longrightarrow
G_{\mathrm R}[\widetilde H]
\qquad\text{in }L^1.
\]
Therefore by $\mathrm{ACS}_{\HH}1$ we find
\[
\{T_{\bn}^{(R)}(H)\}_{\bn}
\sim_\sigma
G_{\mathrm R}[\widetilde H].
\]
Moreover,
\[
G_{\mathrm R}[\widetilde H](\btheta)
=
G_{\mathrm R}[H](-\btheta).
\]
Since reflection preserves Haar measure, we have
\[
\{T_{\bn}^{(R)}(H)\}_{\bn}
\sim_\sigma
G_{\mathrm R}[H].
\]}
\val{Moreover, for a.e.\ $\btheta\in\TT^d$, \eqref{eq:GR-bracket} and the definition of
$G_{S_L,S_R}$ in Theorem~\ref{thm:sandwich+perm-rect-r} lead to }
\[
\val{
G_{\mathrm R}[H](\btheta)
=
\begin{bmatrix}
Z(\btheta)
&
W\!\bigl(r_{S_L}(-\btheta)\bigr)
\\[2pt]
-\overline{W\!\bigl(r_{S_L}(\btheta)\bigr)}
&
\overline{Z(-\btheta)}
\end{bmatrix}
=
G_{S_L,S_R}(\btheta).}
\]
\val{Therefore}
\[
\val{
\{T_{\bn}^{(S_L,S_R)}(F)\}_{\bn}
\sim_\sigma
G_{S_L,S_R}.}
\]
\val{The eigenvalue distribution claims follow by the same argument.}
\end{proof}
\end{MovedProofContext}

\WorkedAppendix

\clearpage
\RevisionBibliography

\end{document}